\documentclass[twoside,11pt]{article}

\usepackage{blindtext}

\usepackage[abbrvbib, preprint, nohyperref]{jmlr2e}
\usepackage{graphicx} \usepackage{xstring}
\usepackage{soul}
\usepackage{booktabs}
\usepackage{caption}
\usepackage{subcaption}

\usepackage{listings}

\usepackage[hyperref, dvipsnames]{xcolor}
\usepackage[colorlinks=true]{hyperref}
\hypersetup{
  colorlinks=true,
  citecolor=blue,
  linkcolor=purple!80
  } 
 
\usepackage{amsmath}
\usepackage{amsfonts}
\usepackage{amssymb}
 \usepackage{mathtools}
\usepackage{tikz}
\usepackage{pgfplots}
\usepgfplotslibrary{statistics}
\pgfplotsset{compat=1.17}
\usepackage{pgfplotstable}
\usepgfplotslibrary{colorbrewer}

\usepackage{enumitem}
\usepackage{algorithm}
\usepackage{algpseudocode}

\usepackage[section]{placeins}

\usepackage{csvsimple}

\usepackage{lmodern}
\usepackage[left=2cm,right=2cm,top=2cm,bottom=2cm]{geometry}

\newtheorem{assumption}[theorem]{Assumption}

\usepackage{lastpage}
\jmlrheading{27}{2026}{1-\pageref{LastPage}}{1/25; Revised
5/26}{8/26}{25-0034}{Amal Alphonse, Michael Hinterm\"uller, Alexander Kister, Chin Hang Lun, and Clemens Sirotenko}

\ShortHeadings{A Neural Network Approach to Learning Solutions of a Class of Elliptic VIs}{Alphonse, Hinterm\"uller, Kister, Lun, and Sirotenko}
\firstpageno{1}

\usepackage{commands}

\begin{document}

\title{A Neural Network Approach to Learning Solutions of a Class of Elliptic Variational Inequalities}

\author{\name Amal Alphonse \email alphonse@wias-berlin.de \\
       \addr Weierstrass Institute\\
        Mohrenstrasse 39\\
         10117 Berlin, Germany
       \AND
       \name Michael Hinterm\"uller \email hintermueller@wias-berlin.de, hint@math.hu-berlin.de \\
       \addr Weierstrass Institute\\
        Mohrenstrasse 39\\
         10117 Berlin, Germany\\
	 \addr Humboldt-Universit\"at zu Berlin\\
	 Unter den Linden 6\\
	 10099 Berlin,  Germany
	\AND
	 \name Alexander Kister \email alexander.kister@bam.de\\
	 \addr Federal Institute for Materials Research and Testing (BAM)\\
	  Unter den Eichen 87\\
	   12205 Berlin, Germany  	
	   \AND 
	   \name Chin Hang Lun \email clun@tripadvisor.com\\
	   \addr Tripadvisor Ltd\\
            10 Norton Folgate   \\
            London, E1 6DB, United Kingdom
	   \AND 
	   \name Clemens Sirotenko \email sirotenko@wias-berlin.de \\
	   \addr Weierstrass Institute\\
        Mohrenstrasse 39\\
         10117 Berlin, Germany
	 }

\editor{Silvia Villa}

\maketitle

\begin{abstract}We develop a weak adversarial approach to solving obstacle problems using neural networks. By employing (generalised) regularised gap functions and their properties we rewrite the obstacle problem (which is an elliptic variational inequality) as a minmax problem, providing a natural formulation amenable to learning. Our approach, in contrast to much of the literature, does not require the elliptic operator to be symmetric. We provide an error analysis for suitable discretisations of the continuous problem, estimating in particular the approximation and statistical errors.  Parametrising the solution and test function as neural networks, we apply a modified gradient descent ascent algorithm to treat the problem and conclude the paper with various examples and experiments. Our solution algorithm is in particular able to easily handle obstacle problems that feature biactivity (or lack of strict complementarity), a situation that poses difficulty for traditional numerical methods. 
\end{abstract}


\section{Introduction}
In this paper, we use neural networks to find solutions of variational  inequalities (VIs) of the following type: 
\begin{equation}
\text{find $u \in K$} : \langle Au-f , u-v \rangle_{\X^*,\X} \leq 0 \quad \text{for all $v \in K$,}\label{eq:VI}
\end{equation}
where $\X:= H^1(\Omega)$ is the usual Sobolev space on a bounded Lipschitz domain $\Omega \subset \mathbb{R}^n$, $\langle \cdot,\cdot\rangle_{V^*,V}$ denotes the duality pairing between $V$ and its topological dual $V^*$, the constraint set $K$ is defined as 
\begin{equation}
K := \left \lbrace u \in H^1(\Omega) \mid u \geq \psi \text{ in $\Omega$}, u = h \text{ on $\partial\Omega$} \right \rbrace,\label{eq:defn_K}
\end{equation}
$h \in H ^{1\slash 2}(\partial\Omega)$ is given boundary data, $\psi \in H^1(\Omega)$ is a given obstacle that satisfies $\psi \leq h$ on $\partial\Omega$, and $f \in L^2(\Omega)$ is a given source term. Concerning the involved Sobolev and Lebesgue spaces we refer, e.g., to \citet{MR2424078}. The operator $A\colon K \subset \X \to \X^*$ appearing in \eqref{eq:VI} is assumed to be Lipschitz and coercive, i.e., there exist constants $C_a, C_b > 0$ such that
\begin{subequations} \label{eq:conditions_on_A}
    \begin{align}
        \norm{Au-Av}{\X^*} &\leq   C_b\norm{u-v}{\X} \quad \forall u, v \in K,\label{eq:conditions_on_A_1}\\
    \langle Au-Av, u-v \rangle_{V^*, V} &\geq C_a\norm{u-v}{\X}^2 \quad\forall u, v \in K,
    \label{eq:conditions_on_A_2}
    \end{align}
\end{subequations}
and for simplicity, we focus our attention on linear differential operators of the form
\begin{equation}
\langle Au, u-v \rangle_{V^*, V} = \int_\Omega \grad u \cdot \grad (u-v) + \sum_{i=1}^n b_i \partial_{x_i}u(u-v) + ku(u-v),\label{eq:operator_A}    
\end{equation}
where $k, b_i \geq 0$, $i=1,\ldots,n$, are constants (which of course have to be such that \eqref{eq:conditions_on_A} is satisfied), $\partial_{x_i}u$ is the weak partial derivative of $u$ with respect to the $i^{\text{th}}$ coordinate and $\grad u$ is the weak gradient of $u$. In operator form we have $A =-\Delta + \sum_{i=1}^n b_i \partial_{x_i} + k\mathrm{Id}$ with $\Delta$ representing the weak Laplacian and $\mathrm{Id}$ the identity map. The setting $b_i = 0$ for $i=1, \hdots, n$, $k=0$, $h\equiv 0$ is the prototypical example of an elliptic variational inequality and is commonly referred to as the obstacle problem \citep{MR0567696}. 

Variational inequalities of the type \eqref{eq:VI} have numerous applications in diverse scientific areas; we mention in particular contact mechanics, processes in biological cells, ecology, fluid flow, and finance, see for example \citet{Rodrigues, MR0567696, MR2488869}. They are also fundamental objects of study in applied analysis due to their interesting structure. Indeed, VIs are examples of free boundary problems. Obstacle  problems sometimes are stated in the form
\begin{subequations}\label{eq:CS}
\begin{align}
0\leq (Au-f)\perp(u-\psi) &\geq  0 \quad\text{a.e. on $\Omega$},\\
u &= h \quad\text{a.e on $\partial\Omega$},
\end{align}
\end{subequations}
where $a\perp b$ stands for $ab=0$. This formulation is equivalent to \eqref{eq:VI} under sufficient regularity (see Proposition \ref{prop:VIandCSEquivalence}).  Classical methods for solving obstacle problems or variational inequalities include projection methods, multilevel and multigrid methods \citep{MR1310316, MR0909064}, primal dual active set strategies and path following schemes \citep{MR2219149}, semismooth Newton schemes \citep{MR1972219,MR1972649}, shape and topological sensitivity based methods \citep{MR2653723,MR2806573},
 level set methods and discontinuous Galerkin schemes \citep{MR2670002}; see also \citet{Glowinski, Bartels,MR2026461, MR0737005} and references therein.

The aim of this work is to formulate, analyse and implement a deep neural network approach to compute solutions of obstacle problems like \eqref{eq:VI} or \eqref{eq:CS}. More specifically, we rephrase the VI as a minmax optimisation problem involving minimisation over the feasible set and maximisation over all feasible test functions, both of which are parametrised by neural networks, and we use a modified gradient descent ascent scheme to numerically solve for the solution.  Our motivation stems from the fact that neural networks can efficiently represent nonlinear, nonsmooth functions and have the added advantage of not being intrinsically reliant on a mesh: they provide a naturally global and meshless representation of the solution, offering an advantage over other methods such as finite elements. Furthermore, they are able to handle complicated geometries and high-dimensional problems without great cost. Our work can also be considered as a first step in studying more complicated problems involving for example operator learning.

Related papers in the literature addressing solving elliptic variational inequalities or hemivariational inequalities via neural networks include \cite{DeepNeural, TwoNN, MR4434275, MR4407899, bahja2023physicsinformedneuralnetworkframework}. A typical path taken by many works entails rewriting \eqref{eq:VI} as a minimisation problem. Indeed, if the operator $A$ generates a bilinear form which is symmetric, then, as explained, e.g., in \cite[\S 4:3, Remark 3.5, p.~97]{Rodrigues}, the VI \eqref{eq:VI} is equivalent to the minimisation problem
 \[\min_{u \in K} \frac 12 \langle Au, u \rangle - \langle f, u \rangle.\]
This formulation gives rise to a natural loss function that can be tackled via neural networks, as done in \cite{DeepNeural, MR4407899, TwoNN}. If $A$ is non-symmetric, this equivalence is not available and one cannot in general pose an associated minimisation problem.  
However, as we shall see later, \eqref{eq:VI} is equivalent to the minmax problem
\begin{equation}\label{weak.VI}
\min_{u \in K}\max_{v \in K} \langle Au-f, u-v \rangle - \frac{1}{2\gamma}\norm{u-v}{V}^2
\end{equation}
for a given parameter $\gamma > 0$, regardless of whether $A$ is symmetric or not.  This problem formulation appears very natural since it resembles the notion of weak formulations in PDEs, which are well understood.

\corr{Classical numerical methods for minmax problems, including asymmetric saddle-point formulations, are typically cast in the framework of variational inequalities and monotone operators. This includes first order schemes such as projection, proximal point, primal-dual, and extragradient methods, which exploit structural properties such as monotonicity and in convex-concave settings often admit rigorous convergence guarantees \citep{Nemirovski2004,CombettesPesquet2012}. Alternatively, one may solve the associated complementarity system directly, or use saddle point, gap function, or regularised gap function reformulations, which recover a merit functional that is often differentiable after regularisation, even in the asymmetric setting, see \citet{Glowinski,Fukushima,MR1274172,MR1430295}. In PDE contexts, discretisation of these formulations leads for instance to monotone multigrid methods, augmented Lagrangian and primal-dual active set strategies, path-following methods, and semismooth Newton algorithms, which exploit the complementarity structure and are attractive because they typically deliver high accuracy together with rigorous convergence and mesh-refinement theory, see e.g. \citet{MR1310316,MR1972219,ItoKunisch2003,MR2026461}. The main drawback of these discretise-then-solve methods is that they remain fundamentally mesh-based, can be less convenient on complicated geometries or in high-dimensional or parametric settings, and may become delicate in the presence of biactivity or lack of strict complementarity. 

In order to approximate \eqref{weak.VI} we express $u$ and the test function $v$ by deep neural networks. This technique falls into the class of weak adversarial network (WAN) problems in the spirit of \citet{Zang}; the maximisation for the test function acts as an adversarial force against the minimisation for the solution. In addition to \citet{Zang}, see also e.g. \citet{Bao, bertoluzza2024bestapproximationresultsessential, MR4416735, Jiao2023} for weak adversarial approaches for PDEs and related theory. Moreover, our approach is related to Physics Informed Neural Networks (PINNs); see, e.g., \citet{Lagaris98, RAISSI2019686}.  More specifically it can be viewed as a weak PINNs-type approach with a hard constrained boundary condition. While neural network approaches enable greater flexibility and scalability in high-dimensional settings, it fundamentally alters the problem structure: the resulting optimisation landscape may contain spurious stationary points that do not correspond to meaningful minmax solutions, and commonly used algorithms may exhibit instability or fail to converge. As a result, theoretical guarantees are typically weaker and limited. Despite these challenges, neural network-based approaches offer significant advantages in terms of mesh-independence, flexibility, scalability, and their ability to handle high-dimensional and complex problem settings, and this essentially motivates this work.}

In this paper, we will provide a theoretical justification for the minmax problem, a discretised formulation of the problem amenable to computation, an error analysis as well as comprehensive numerical simulations.  As we mentioned above, a special highlight of our work is that we can also tackle non-symmetric problems, e.g., $A$ given as in \eqref{eq:operator_A} with $b_i \neq 0$ for at least one $i=1, \hdots, n$.

\corr{We briefly summarise the main results of this paper. At the continuous level we  give a reformulation of the VI \eqref{eq:VI} as the minimisation of a regularised gap function in Lemma \ref{lem:regularisedGapVI}; this gives rise to the minmax problem \eqref{eq:minMaxDiagonalProblem}, involving constrained minimisation with respect to the solution and constrained maximisation with respect to the test function. Regarding the mentioned regularised gap function, an important tool that will facilitate error estimates later is given in Lemma \ref{lem:error_estimate_G_gamma}. We consider various relaxations of the problem that inexactly enforce the boundary condition and/or obstacle constraint via penalisations in \S \ref{sec:relaxations}, including \eqref{eq:cts_ansatz} where the boundary conditions are imposed exactly and the obstacle constraint is penalised; the discretised version of this is what we shall solve numerically later. For the discretised problem \eqref{eq:disc_ansatz} corresponding to \eqref{eq:cts_ansatz}, obtained by restricting the solution and test spaces to neural network classes and replacing integrals by empirical averages, a notion of existence of a solution to the resulting finite-dimensional minmax problem is established under suitable assumptions in Corollary \ref{cor:existence_disc_problems}. An \emph{a priori} error estimate between the exact solution and a computed discrete solution is given in \eqref{eq:full_error_estimate}; it is decomposed into three contributions:
(i) an \emph{approximation error}, reflecting the approximation properties of the chosen neural network classes for both the solution and test variables;
(ii) a \emph{statistical error}, arising from the replacement of integrals by Monte Carlo averages; and
(iii) an \emph{optimisation error}, due to the inexact solution of the discrete minmax problem by a chosen algorithm.
We show that the approximation error can be controlled essentially by making the widths and depths of the neural network classes sufficiently large, see Theorem \ref{thm:approximation_error}, and that the statistical error admits an explicit rate  of order $\mathcal{O}(N^{-1/4})$ with respect to the number of collocation points $N$, see Theorem \ref{thm:stat_error}. Finally, in Section \ref{sec:numerical_examples}, we  close the paper with a number of examples and examine some qualitative properties such as mesh independence.
}

\section{Analysis of the continuous problem}
We begin with some theoretical results concerning the VI \eqref{eq:VI}.
\subsection{Basic properties and saddle points}
Let us first of all address existence and uniqueness for \eqref{eq:VI}. Since $A$ is coercive and Lipschitz on $H^1(\Omega)$ and $K$ is non-empty\footnote{By properties of the trace operator \cite[Theorem 8.8, Chapter 1]{MR895589}, there exists a function $w \in H^1(\Omega)$ with $w|_{\partial\Omega} = h$, and the function $\max(w, \psi) \in H^1(\Omega)$ belongs to $K$.}, closed and convex, well posedness follows from the classical Lions--Stampacchia theorem \cite[\S 4:3, Theorem 3.1]{Rodrigues}, see also \citet[\S 4:2, Theorem 2.3]{Rodrigues} for the case where $A=-\Delta$.

\begin{proposition}[$H^2$-regularity]\label{prop:h2_regularity}
Let $A:= -\Delta + \sum_{i=1}^n b_i\partial_{x_i} + c\mathrm{Id}$ 
    with coefficients $b_i, c \in L^\infty(\Omega)$, $c \geq c_0 \geq 0$ for a constant $c_0$, such that the coercivity condition \eqref{eq:conditions_on_A_2} is satisfied, $f \in L^2(\Omega)$, $h \in H^{3\slash 2}(\partial\Omega)$, $\psi \in H^2(\Omega)$ with $\psi|_{\partial\Omega} \leq h$, and let $\Omega$ be convex or $C^{1,1}$. Then the solution of \eqref{eq:VI} has the regularity $u \in H^2(\Omega)$. Furthermore, we have the \emph{a priori} estimate
$\norm{u}{H^2(\Omega)} \leq C^*$,
where $C^*$ is a constant (that depends in particular on $f, \psi$ and  $h$).

\end{proposition}
\begin{proof}
In the case $h \equiv 0$, this follows by the Lewy--Stampacchia inequality \cite[Theorem 3.3, \S 5:3]{Rodrigues} 
    $f \leq Au \leq \max(f, A\psi)$
and from elliptic estimates which bound the $H^2$-norm of $u$ by the $L^2$ norm of $Au$ and the $H^1$-norm of $u$,
    see e.g. \citet[Theorem 3.4, \S 5:3]{Rodrigues} for the case $\Omega \in C^{1,1}$, and \citet[Proposition 2.5]{MR2333653} for when $\Omega$ is convex.  In the general inhomogeneous case of $h \in H^{3\slash 2}(\partial\Omega)$, we can find $\tilde h \in H^2(\Omega)$ with $\tilde h|_{\partial\Omega} = h$ and make a substitution $\tilde u := u-\tilde h$ like done in the proof of \citet[Theorem 2.3, \S 4:2]{Rodrigues} and resort to the above situation.
\end{proof}
For a direct proof in the setting when $A=-\Delta$, see \citet[Proposition 2.2 and Corollary 2.3, \S 5:2]{Rodrigues}. 
We will now consider equivalent reformulations of the problem \eqref{eq:VI}. Since the proof of the following result mostly follows the argumentation in \citet[\S 1:3, p.~4]{Rodrigues}, we have placed it in Appendix \ref{app:proofs}.
\begin{proposition}\label{prop:VIandCSEquivalence}
    If $u \in H^1(\Omega)$ satisfies $Au \in L^2(\Omega)$ and \eqref{eq:CS}, then it solves \eqref{eq:VI}. Conversely, if $u$ solves \eqref{eq:VI} and satisfies $u \in C^0(\overline{\Omega})$ and $Au \in L^2(\Omega)$, and $\psi \in C^0(\overline{\Omega})$, then $u$ solves \eqref{eq:CS}. \end{proposition}
The VI \eqref{eq:VI} has a saddle point reformulation \citep[as discussed in][page 14]{Glowinski}. Let us recall the definition first. \begin{definition}
Given a map $\mathfrak{f}\colon X \times Y \to \mathbb{R}$, a pair $(x^*,y^*) \in X \times Y$ is called a \emph{saddle point} \citep[\S 9.6]{ZeidlerI} if it satisfies 
\[    \mathfrak{f}(x^*,y) \leq \mathfrak{f}(x^*, y^*)  \leq \mathfrak{f}(x,y^*) \quad \text{$\forall x \in X,$ $\forall y \in Y$,}     
\]
or equivalently
\[    \max_{y \in Y}\mathfrak{f}(x^*, y) = \mathfrak{f}(x^*, y^*)  = \min_{x \in X}\mathfrak{f}(x,y^*) \quad  \text{$\forall x \in X,$ $\forall y \in Y$.}     
\]
\end{definition}
Associated to the VI \eqref{eq:VI} we define a map $L\colon V \times V \to \mathbb{R}$ as
\begin{equation}
    L(u,v)  := \langle Au -f,u-v \rangle. \label{eq:def_L}
\end{equation}
Note that $L(\cdot,v)$ is convex and $L(u,\cdot)$ is concave for fixed $u, v \in V$.
From, e.g., \cite[Corollary 9.16]{ZeidlerI}, if $(u, v)$ is a saddle point then it satisfies
\begin{equation}
L(u, v) = \min_{\tilde u \in K} \sup_{\tilde v \in K} L(\tilde u, \tilde v) = \max_{\tilde v \in K}\inf_{\tilde u \in K} L(\tilde u, \tilde v).\label{eq:min_sup_problem}    
\end{equation}
Moreover, $L$ has a saddle point if and only if the second equality holds. Note that we have supremum and infimum above as these optimising elements may not be attained in general. However if $K$ is a bounded set, then the $\sup$ and $\inf$ can be replaced by $\max$ and $\min$. In addition to the above references we cite also \citet[Chapter VI, \S 1]{EkelandTemam} for more details. Now, regarding the connection of saddle points of $L$ to the VI \eqref{eq:VI}, we have the following result, see \citet[Remark 2.7]{Glowinski}. \begin{proposition}\label{prop:saddle_theorem}
If $u \in K$ is a solution of \eqref{eq:VI} then $(u,u)$ is a saddle point of $L$ on $K \times K$. Conversely, if $(u,v)$ is a saddle point of $L$ on $K \times K$ then $u = v$ and $u$ solves \eqref{eq:VI}. 
\end{proposition}

\subsection{Minmax approach via the regularised gap function}

As described above, the saddle point formulation of the VI is a $\min\sup$ problem involving $L(u,v) = \langle Au-f, u-v \rangle$. It turns out that modifying $L$ by adding a quadratic term makes the problem more tractable and endows it with better properties. First, let us make the following assumption which allows us to be more general with the formulation of the $\min\sup$ problem.
\begin{assumption}\label{ass:on_spaces}
Let $u^*$ be a solution of \eqref{eq:VI}, and let $X$ and $Y$ be sets such that
\begin{enumerate}[label=(\roman*)]\itemsep=0cm
\item $u^* \in X \cap Y$, 
\item $K \cap X \subset K \cap Y$,   
\item $K \cap Y$ is convex and closed\footnote{The closedness is required for the unique solvability of the minmax problem \eqref{eq:prelim1}.}.  
\end{enumerate}
\end{assumption}
We have in mind two examples: the standard  setting $X=Y=V$, in which case the intersections above involving $K$ are simply equal to $K$, and secondly $Y=B^{H^2(\Omega)}(r)$ the closed ball in $H^2(\Omega)$ with center $0$ and an appropriately chosen radius $r$ (which is valid when $u^*$ has the regularity $u^* \in H^2(\Omega)$ and we have an estimate for its norm in that space; see Proposition \ref{prop:h2_regularity} for conditions that ensure this) and $X \subset Y$ is some subset. This latter option is useful because it allows us to use neural network approximation results in the error analysis later.

Now, let us introduce, for a fixed $\gamma  > 0$,  the map $L_\gamma \colon V \times V \to \mathbb{R}$ defined
\[L_\gamma(u,v) := L(u,v) - \frac{1}{2\gamma}\norm{u-v}{\X}^2,\]
and the \emph{generalised regularised gap function} $G_\gamma\colon V \to \mathbb{R}$ given by
\begin{align*}
G_\gamma(u) &:= \sup_{v \in K \cap Y} L_\gamma(u,v)
= \sup_{v \in K \cap Y} \left(L(u,v) - \frac{1}{2\gamma}\norm{u-v}{\X}^2\right).
\end{align*}
When $X=Y=V$, $G_\gamma(\cdot) = \sup_{v \in K} L_\gamma(\cdot, v)$ is known in the literature as the \emph{regularised gap function}. In addition, the case $\gamma=\infty$, i.e., the function $G_\infty(u) = \sup_{v \in K} L(u,v)$ is called a \emph{gap function} (it is in general non-differentiable and the supremum may be infinite). For simplicity we will just refer to $G_\gamma$ above without reference to `generalised'. The quadratic term in $G_\gamma$ serves to add differentiability and ensures that the supremum is finite.  These concepts and properties when $X=Y$ is the whole space $V$ are well known, see \citet{Auslender, Fukushima, Auchmuty} for the originating works; we refer also to \citet{MR1274172, MR1430295}.

Since $\norm{\cdot}{\X}^2$ is strongly convex, it follows that $-L_\gamma(u,\cdot)$ is strongly convex for fixed $u$; this will be important later for existence. For now, we prove some basic properties. \begin{lemma}\label{lem:regularisedGapVI} 
The function $G_\gamma$ satisfies the following properties:
\begin{enumerate}[label=(\roman*)]\itemsep=0cm
    \item $G_\gamma$ is finite everywhere on $V$,
    \item\label{item:GgammaNonNeg} $G_\gamma(u)\geq 0$ for $u \in K \cap X$,
    \item\label{item:uiffGgamma0} $u \in K \cap X$ satisfies $G_\gamma(u) = 0$ if and only if $u$  solves \eqref{eq:VI},
\item $G_\gamma \colon K\cap X \subset V \to \mathbb{R}$ is lower semicontinuous.
\end{enumerate}
\end{lemma}
The proof is standard for the usual case $X=Y=V$, see e.g., \citet[Theorem 3.1]{Migorski}. 
\begin{proof}
\begin{enumerate}[label=(\roman*)]\itemsep=0cm
\item This follows easily by Young's inequality with parameter $\epsilon>0$ chosen carefully. 

\item We have $G_\gamma(u) \geq L_\gamma(u,v)$ for all $v \in K \cap Y$ and we can select $v = u$ since $u \in K \cap X \subset K \cap Y$. The claim follows from $L_\gamma(u,u)=0$.

\item If $u$ solves \eqref{eq:VI}, then $G_\gamma(u) \leq \sup_{v \in K \cap Y} -\frac{1}{2\gamma}\norm{u-v}{\X}^2 \leq 0$, which combined with the non-negativity of $G_\gamma$ gives $G_\gamma(u)=0$. 

For the converse, if $u \in K\cap X$ satisfies $G_\gamma(u)=0$, we have 
\[0 = \sup_{v \in K \cap Y} \left(L(u,v) - \frac{1}{2\gamma}\norm{u-v}{\X}^2\right) \geq \left(L(u,v) - \frac{1}{2\gamma}\norm{u-v}{\X}^2\right) \quad \forall v \in K \cap Y,\]
and with $\lambda \in (0,1)$, selecting $v= (1-\lambda) u + \lambda w$ where $w \in K\cap Y$ is arbitrary (by convexity, $v \in K \cap Y$ too),  manipulating and taking $\lambda \to 0$ we get 
\[\langle Au-f, u-w \rangle \leq 0 \quad \forall w \in K\cap Y.\]
Now we also know that \eqref{eq:VI} is uniquely solvable with $u^*$ as solution, thus
\[\langle Au^*-f, u^*-v \rangle \leq 0 \quad \forall v \in K.\]
Picking $w=u^*$ (which by assumption belongs to $Y$) and $v=u$ and using coercivity, we find $u=u^*$.

\item This follows by the continuity of $u \mapsto L_\gamma(u,v)$  from $H^1(\Omega)$ into $\mathbb{R}$ for fixed $v$; the argumentation is the same as in \citet[Lemma 3.1]{Migorski}.
\end{enumerate}
\end{proof}
Due to Lemma \ref{lem:regularisedGapVI} \ref{item:GgammaNonNeg} and \ref{item:uiffGgamma0}, to find the solution of \eqref{eq:VI}, we can solve
\begin{equation}
\min_{u \in K \cap X} G_\gamma(u) \equiv \min_{u \in K \cap X} \sup_{v \in K \cap Y} L_\gamma(u,v)\label{eq:prelim1}
\end{equation}
(cf. \eqref{eq:min_sup_problem}). Now, we would like to consider instead of the supremum above  the maximum. First observe that for all $u \in V$, $-L_\gamma(u,\cdot)$ is strongly convex (as mentioned above), continuous, and also proper thanks to the quadratic term. Since we also assumed that $K \cap Y$ is closed, then by standard optimisation theory the problem
\[\max_{v \in K \cap Y} L_\gamma(u,v)\]
has a unique solution. Thus instead of \eqref{eq:prelim1}, we can consider
\begin{equation} 
\min_{u \in K \cap X} G_\gamma(u) \equiv \min_{u \in K \cap X} \max_{v \in K \cap Y} L_\gamma(u,v).\label{eq:minMaxDiagonalProblem}
\end{equation}
A natural question now arises as to what we mean by a solution to this problem.
\begin{definition}\label{defn:minmax}
    Given a map $\mathfrak{f}\colon X \times Y \to \mathbb{R}$ (which need not be convex-concave), by a \emph{(global) solution} or a \emph{(global) minmax point} of the minmax problem 
    \begin{equation}
\min_{x \in X}\max_{y \in Y} \mathfrak{f}(x,y),\label{eq:f_minmax_problem}        
    \end{equation}
following \citet[Definition 9 and Remark 10]{Jordan}, we mean a pair $(x^*, y^*)$ that satisfies 
\[\mathfrak{f}(x^*, y) \leq \mathfrak{f}(x^*, y^*) \quad \forall y \in Y\quad\text{and} \quad \max_{y \in Y}\mathfrak{f}(x^*, y) \leq \max_{y \in Y} \mathfrak{f}(x,y) \quad \forall x \in X\]
or equivalently
\[\mathfrak{f}(x^*, y) \leq \mathfrak{f}(x^*, y^*) \leq \max_{\hat y \in Y}\mathfrak{f}(x,\hat y) \qquad \forall (x,y) \in X \times Y.\]
\end{definition}
Observe that saddle points are global minmax points. In the context of finite dimensional problems, existence of solutions to \eqref{eq:f_minmax_problem} holds if $X \subset \mathbb{R}^{d_1}$ and $Y \subset \mathbb{R}^{d_2}$ are compact and $\mathfrak{f}$ is continuous \citep[Proposition 11]{Jordan}.  See also Proposition 2.6 of \citet{JiangChen} (and the proceeding paragraph) for another existence result.

The next lemma is a crucial tool as it will give us an error estimate later on.
\begin{lemma}\label{lem:error_estimate_G_gamma}
\corr{With $u^* \in Y$ a solution of \eqref{eq:VI},} we have 
\begin{align*}
\left(C_a-\frac{1}{2\gamma}\right)\norm{u-u^*}{V}^2 + \langle Au^*-f, u-u^*\rangle \leq G_\gamma(u) \qquad \forall u \in V.
\end{align*}
If $u \in K$, the duality product can be omitted.
\end{lemma}

\begin{proof}
For any $u \in V$ (not necessarily in $K \cap X$), since $u^* \in Y$, we have
\begin{align*}
G_\gamma(u) &= \max_{v \in K \cap Y} \langle Au-f, u-v \rangle - \frac{1}{2\gamma}\norm{u-v}{V}^2 \geq \langle Au-f, u-u^* \rangle - \frac{1}{2\gamma}\norm{u-u^*}{V}^2\\
&= \langle A(u-u^*) + Au^*-f, u-u^* \rangle - \frac{1}{2\gamma}\norm{u-u^*}{V}^2\\
&\geq \left(C_a-\frac{1}{2\gamma}\right)\norm{u-u^*}{V}^2 + \langle Au^*-f, u-u^*\rangle.
\end{align*}
\end{proof}
From now on we will assume that $\gamma$ is such that
\begin{equation}
    C_a - \frac{1}{2\gamma} > 0, \qquad \text{i.e.,} \qquad \gamma > \frac{1}{2C_a}.\label{eq:condition_on_gamma}
\end{equation}
\subsection{Relaxations of the problem}\label{sec:relaxations}
We now look at relaxing the minmax problem of interest \eqref{eq:prelim1} in ways that are more amenable to implementation, as optimising over a constraint set involving $K$ is non-trivial. Let us define the loss function $L_o\colon L^2(\Omega) \to \mathbb{R}$ corresponding to the 
obstacle constraint and the loss function $L_b\colon L^2(\partial\Omega) \to \mathbb{R}$ for the boundary condition by
\begin{equation}
L_{o}(u) = \int_\Omega |(\psi-u)^+|^2  \qquad \text{and}\qquad   L_b(u) =  \int_{\partial\Omega}|u-h|^2\label{eq:loss_functions_cts}    
\end{equation}
respectively. These measure violations of the constraints.

\paragraph{Imposition of the boundary conditions}

The idea is the following. We take a function $\bar h \in V$ with $\bar h|_{\partial\Omega}=h$ (i.e., $\bar h$ is a lift or extension of the Dirichlet data to the interior of the domain) and posit that the solution and test functions are of the form $\bar h + z$ for $z \in H^1_0(\Omega)$, and given penalty parameters   $w_{o_1}$ and $w_{o_2}$, we solve
    \begin{equation}
   \min_{u \in H^1_0(\Omega) + \bar h}\max_{v \in H^1_0(\Omega) + \bar h} L_\gamma (u,v)     + w_{o_1}L_o(u) - w_{o_2}L_o(v).\tag{$\text{P}_{\mathrm{bc}}$}\label{eq:cts_ansatz}
    \end{equation}
    That is, the boundary condition is satisfied exactly and  violations of the obstacle constraint are penalised. 

\paragraph{Full penalty approach}In this case we seek to optimise over the entire Hilbert space and penalise violations of \textit{both} the obstacle constraint and boundary condition. In contrast to the previous approach, the boundary condition is not enforced at the outset. This leads us to consider     for penalty parameters $w_{o_i}, w_{b_i}$ (for $i=1,2$) the problem
    \begin{equation}
        \min_{u \in \X}\max_{v \in \X} L_\gamma(u,v)   + w_{o_1}L_o(u) + w_{b_1}L_b(u) - w_{o_2}L_o(v)- w_{b_2}L_b(v). \tag{$\text{P}_{\mathrm{pen}}$}\label{eq:fullPenaltyLoss}        
    \end{equation}
\paragraph{Generalised approach}Clearly, both of these approaches can be viewed as special cases of a single more general problem by choosing the function spaces and the weights for the penalty terms correctly. It becomes convenient later to take this viewpoint. Define for $i=1,2$ the penalty functions $R_i\colon H^1(\Omega) \to \mathbb{R}$ by
\begin{align*}
R_1(u) :=  w_{b_1}L_b(u) +  w_{o_1}L_o(u) \qquad\text{and}\qquad R_2(v) = w_{b_2}L_b(v) + w_{o_2}L_o(v).\end{align*}
A generalised problem can be posed as
\begin{equation}
\min_{u \in X^{\mathrm{s}}}\max_{v \in X^{\mathrm{t}}} L_\gamma(u,v) + R_1(u) -  R_2(v)\tag{${\text{P}}_{\mathrm{gen}}$}\label{eq:proble_gen_cts}
\end{equation}
where $X^{\mathrm{s}}$ and $X^{\mathrm{t}}$ are sets. It is (a discrete version of) this problem that we shall provide an error analysis for inasmuch as it possible to do so, in order to allow for the greatest generality possible.

\section{Neural network approach}\label{sec:nnapproach}
We wish to compute (approximate) solutions of the VI \eqref{eq:VI} using neural networks. The architecture we use is essentially the residual neural network  considered in \citet{DeepRitz} consisting of the usual affine transformations and activations combined with skip connections, inspired by the original work of \citet{DeepResImaging}. Residual networks or ResNets have been empirically observed to be better at training deep networks and they avoid the vanishing gradient problem, see for example \citet{HeEtAll, WU2019119} for some analysis.

Let us describe this special ResNet architecture precisely. Let $\fd, \fw \in \mathbb{N}$ be given positive integers (representing the depth and width of the network respectively). Given an initial weight $A_0 \in \mathbb{R}^{\fw \times n}$ and bias $b_0 \in \mathbb{R}^{\fw}$, and for $i=1, \hdots, \fd$ and $j=1,2$,   weights $A_{ij} \in \mathbb{R}^{\fw \times \fw}$ and biases $b_{ij} \in \mathbb{R}^{\fw}$ determining the affine transforms
\[T_{ij}z := A_{ij}z + b_{ij}\quad\text{for $i=1, \hdots, \fd$ and $j=1,2$},\]
and a weight $A_{\fd+1} \in \mathbb{R}^{1 \times \fw}$ and bias $b_{\fd+1} \in \mathbb{R}$ for the final layer, define
\begin{align*}
    T_0(x) &:= A_0x + b_0,\\
    \mathfrak{B}_i(z) &:= \sigma \circ T_{i2} \circ \sigma \circ T_{i1}(z)  + z \quad \text{for $i= 1, \hdots, \fd$},\\
    T_{\fd+1}(z) &:= A_{\fd+1} z + b_{\fd+1}.
\end{align*}Note that each $\mathfrak{B}_i$ for $i = 1, \hdots, \fd$ has the so-called \emph{block} structure; each block comprises two affine transformations and two activations with a residual connection, and we have that $\mathfrak{B}_i \colon \mathbb{R}^\fw \to \mathbb{R}^\fw$. 

\definecolor{inputcol}{HTML}{AED6F1}
\definecolor{linecol} {HTML}{A9DFBF}
\definecolor{actcol}  {HTML}{F9E79F}
\definecolor{blockcol}{HTML}{EBF5FB}
\definecolor{arrowcol}{HTML}{2C3E50}
\definecolor{skipcol} {HTML}{C0392B}

\tikzset{
  bar/.style     = {draw, rounded corners=2pt,
                    minimum width=0.55cm, minimum height=2.2cm,
                    font=\scriptsize, align=center, inner sep=3pt},
  inbar/.style   = {bar, fill=inputcol, draw=blue!50!black},
  linbar/.style  = {bar, fill=linecol,  draw=green!40!black},
  actbar/.style  = {bar, fill=actcol,   draw=orange!70!black},
  plus/.style    = {circle, draw=gray!70, fill=white, thick,
                    minimum size=0.45cm, font=\small, inner sep=0pt},
  arr/.style     = {-{Stealth[length=5pt,width=4pt]}, thick, color=arrowcol},
  skip/.style    = {-{Stealth[length=5pt,width=4pt]}, thick,
                    color=skipcol, rounded corners=4pt},
  blkbox/.style  = {fill=blockcol, draw=blue!30, dashed,
                    rounded corners=5pt, inner sep=5pt},
  lbl/.style     = {font=\scriptsize},
}

\usetikzlibrary{positioning, arrows.meta, fit, backgrounds, calc}

\begin{figure}[!htb]
\centering
\begin{tikzpicture}[node distance=0.35cm]

\node[inbar, minimum height=0.85cm] (x) {};
\node[lbl, below=0.08cm of x] {$x$};

\node[linbar, right=0.7cm of x] (T0) {};
\node[lbl, below=0.08cm of T0] {$T_0$};

\node[linbar, right=0.8cm of T0] (T11) {};
\node[lbl, below=0.18cm of T11] {$T_{1,1}$};

\node[actbar, minimum height=0.55cm, right=of T11] (s11) {};
\node[lbl, below=0.18cm of s11] {$\sigma$};

\node[linbar, right=of s11] (T12) {};
\node[lbl, below=0.18cm of T12] {$T_{1,2}$};

\node[actbar, minimum height=0.55cm, right=of T12] (s12) {};
\node[lbl, below=0.18cm of s12] {$\sigma$};

\node[plus, right=0.35cm of s12] (p1) {$+$};

\def\skipheight{1.55}
\draw[skip]
  ($(T0.east)!0.5!(T11.west)$)
  -- ++(0,1.2*\skipheight) coordinate (skip1top)
  -| (p1.north);

\begin{pgfonlayer}{background}
\node[blkbox,
      fit=(T11)(s11)(T12)(s12)(p1)(skip1top),
      label={[font=\scriptsize\bfseries, color=blue!60!black]
             above:$\mathfrak{B}_1$}] {};
\end{pgfonlayer}

\node[font=\normalsize, right=0.55cm of p1] (dots) {$\cdots$};

\node[linbar, right=0.55cm of dots] (Td1) {};
\node[lbl, below=0.18cm of Td1] {$T_{d,1}$};

\node[actbar, minimum height=0.55cm, right=of Td1] (sd1) {};
\node[lbl, below=0.18cm of sd1] {$\sigma$};

\node[linbar, right=of sd1] (Td2) {};
\node[lbl, below=0.18cm of Td2] {$T_{d,2}$};

\node[actbar, minimum height=0.55cm, right=of Td2] (sd2) {};
\node[lbl, below=0.18cm of sd2] {$\sigma$};

\node[plus, right=0.35cm of sd2] (pd) {$+$};

\draw[skip]
  ($(dots.east)!0.5!(Td1.west)$)
  -- ++(0,1.2*\skipheight) coordinate (skipdtop)
  -| (pd.north);

\begin{pgfonlayer}{background}
\node[blkbox,
      fit=(Td1)(sd1)(Td2)(sd2)(pd)(skipdtop),
      label={[font=\scriptsize\bfseries, color=blue!60!black]
             above:$\mathfrak{B}_d$}] {};
\end{pgfonlayer}

\node[linbar, right=0.8cm of pd] (Tout) {};
\node[lbl, below=0.08cm of Tout] {$T_{d+1}$};

\node[inbar, minimum height=0.55cm, right=0.7cm of Tout] (u) {};
\node[lbl, below=0.08cm of u] {$u(x)$};

\foreach \a/\b in {
  x/T0, T0/T11, T11/s11, s11/T12, T12/s12, s12/p1,
  p1/dots, dots/Td1, Td1/sd1, sd1/Td2, Td2/sd2, sd2/pd,
  pd/Tout, Tout/u}
  \draw[arr] (\a) -- (\b);

\end{tikzpicture}
\caption{\corr{Sketch of the $\mathcal{F}_{\mathrm{DRR}}$ architecture given in \eqref{eq:NN}.
Red lines denote residual connections.}}
\label{fig:new_FDRR}
\end{figure}

With this notation at hand, the class of neural networks under consideration in this work consists of functions with the form
\begin{equation}
u \colon \mathbb{R}^n \to \mathbb{R}, \qquad u(x) = T_{\fd+1} \circ \mathfrak{B}_{\fd}  \circ \cdots \circ \mathfrak{B}_1 \circ T_0(x).\label{eq:NN} 
\end{equation}
The coefficients of the matrices and vectors appearing above determine the learnable parameters of the neural network. The neural network above has $\fd$ blocks and in total a depth of $2\fd+2$ layers (if we understand 
each block to have two layers; hence there are $2\fd$ hidden layers). We will write the set of all such neural networks of block depth $\fd$, width $\fw$ and activation function $\sigma$ as 
\[\mathcal{F}_{\mathrm{DRR}}(\fd,\fw,\sigma) := \{u \colon \mathbb{R}^n \to \mathbb{R} : u \text{ is given by \eqref{eq:NN}}\}\]
($\mathrm{DRR}$ for \textit{D}eep \textit{R}itz \textit{R}esidual neural network, from the title of the work by \citet{DeepRitz} from where this structure originated), \corr{see Figure \ref{fig:new_FDRR} for a diagram}. Elements of this set differ by different choices of the learnable parameters, i.e., the weights and biases. By denoting all the learnable parameters of the network \eqref{eq:NN} as 
\[\theta = \left\{(A_0, b_0), (A_{\fd+1}, b_{\fd+1}) \right\} \bigcup_{i=1}^{\fd} \left \{(A_{i1}, b_{i1}), (A_{i2}, b_{i2})\right\},\]
we can view any $u \in \mathcal{F}_{\mathrm{DRR}}(\fd,\fw,\sigma)$ as a function of the parameters, i.e.,  $u=u(\theta)$. The element $\theta$ belongs to the \emph{weight space} (or \emph{parameter space})
\begin{align*}
    \nonumber \Theta := &\Big\{\theta = \{(A_0, b_0), (A_\fd, b_\fd)\} \bigcup_{i=1}^{\fd} \{(A_{i1}, b_{i1}), (A_{i2}, b_{i2})\} \\
&\qquad\qquad : \quad (A_0, b_0) \in \mathbb{R}^{\fw \times n} \times \mathbb{R}^{\fw}, (A_{\fd+1}, b_{\fd+1}) \in \mathbb{R}^{1\times \fw} \times  \mathbb{R}, (A_{ij}, b_{ij}) \in \mathbb{R}^{\fw \times \fw} \times \mathbb{R}^{\fw}\Big\}.\end{align*}
In total, the network has
\[\mathfrak{m}:= n\fw + \fw + 2(\fw^2+\fw)\fd + \fw + 1 = 2\fd\fw^2 + (n+2\fd +2)\fw + 1\] 
parameters to be determined hence $\Theta \subset \mathbb{R}^\mathfrak{m}$ can be viewed as a subset of Euclidean space.
Note that the first layer is chosen so that the input $x \in \mathbb{R}^n$ is  transformed to a $\fw$-dimensional object (which is necessary to apply the block structure). In the case that $n < \fw$, we could alternatively have padded the input with zeros. 

The standard feedforward neural network structure ubiquitous in machine learning has a very similar architecture as described above except (more or less) without the addition of the $z$ term in the definition of $\mathfrak{B}$. Let us write this set of neural networks as $\mathcal{F}_{\mathrm{FFN}}(\fd, \fw, \sigma)$ where an element of this space has the structure
\[u\colon \mathbb{R}^n \to \mathbb{R}\qquad u(x) = T_{\fd+1} \circ \sigma  \circ T_{\fd} \circ \sigma \circ \cdots \circ T_1 \circ \sigma \circ T_0(x).\]
Note that $\cF_{\mathrm{DRR}}(\fd, \fw, \sigma)$ has $2\fd + 2$ layers while $\cF_{\mathrm{FFN}}(\fd, \fw, \sigma)$ has $\fd+2$ layers; this mismatch is of little consequence as it is just a notational choice.

For convenience, we write 
$\mathcal{F}$ or $\cF(\fd, \fw, \sigma)$
for a set of neural networks of arbitrary architecture (with depth $\fd$, width $\fw$ and activation $\sigma$; the depth and width are ambiguous), typically $\cF \in \{\cF_{\mathrm{DRR}}, \cF_{\mathrm{FFN}}\}$ is what we have in mind.

\paragraph{Neural networks with homogeneous BC.} Given a neural network class $\cF$, it is useful to consider another set of functions $\mathcal{F}_{0}$
that have the same type of architecture as $\cF$ and that are zero on the boundary. 

One way to practically construct such a set is by taking a function $\eta \in C^1(\overline{\Omega})$ satisfying $\eta|_{\partial\Omega} = 0,$ then given $v \in \mathcal{F}$, the function $v\eta \in \mathcal{F}_{0}$. That is, given $\tilde u \in \cF$ this amounts to appending an extra final layer
$M_{\eta} \circ \tilde u$
where $M_\eta\colon C^1(\overline{\Omega}) \to C^1(\overline{\Omega})$ defined by $M_\eta(\tilde u) := \tilde u \eta$ is the pointwise multiplication operator. Let us label this subset of $\cF_0$ with a subscript $\eta$, e.g., as 
\begin{equation}\label{eq:defnCf0Eta}
\cF_{0,\eta} := \{ M_\eta \circ \tilde u : u \in \cF \}.
\end{equation}
Thus in particular, 
$\mathcal{F}_{\mathrm{DRR}, 0, \eta}(\fd,\fw,\sigma) := \{u \colon \mathbb{R}^n \to \mathbb{R} : u(x) = M_\eta \circ T_{\fd+1} \circ \mathfrak{B}_{\fd}  \circ \cdots \circ \mathfrak{B}_1 \circ T_0(x) \}$. The function $\eta$ is easy to construct  when the domain is an interval or rectangle.  For more complex geometries it is non-trivial and the desired $C^1$ property may not hold; see e.g. \citet{SUKUMAR2022114333} for a way to construct such a function $\eta$ (with lower regularity).

\medskip 

\noindent \textbf{Neural networks with additional final layer.} We can generalise this structure to cater to a more general final layer that could be different to pointwise multiplication. Indeed, given a operator $M \colon C^1(\overline{\Omega}) \to C^1(\overline{\Omega})$, let us define 
\begin{equation}
\cF(\fd, \fw, \sigma, M) := \{ M \circ u : u \in \cF(\fd, \fw, \sigma)\}.\label{eq:defn_cF_general_M}
\end{equation}
Note that $\cF_{0,\eta}(\fd,\fw,\sigma) \equiv 
\cF(\fd, \fw, \sigma, M_\eta)$ in this notation. 

We finish with a simple but useful result which follows by calculus arguments, see Appendix \ref{app:proofs} for the proof.
\begin{lemma}\label{lem:NNisC1_general}
Let $u \in \cF_{\mathrm{DRR}}(\fw, \fd, \sigma, M) \cup \cF_{\mathrm{FFN}}(\fw, \fd, \sigma, M)$ where $\sigma \in C^1(\mathbb{R})$ and $M \colon C^1(\bar\Omega) \to C^1(\bar\Omega)$ is continuous.   
Then the map $\theta \mapsto u(\theta, \cdot)$ is continuous from $\mathbb{R}^{\mathfrak{m}}$ to $C^1(\overline{\Omega})$.
\end{lemma}
Note that this result applies to elements of the space $\cF_{0,\eta}$ for both architectures.  It also implicitly tells us that if $\sigma \in C^1(\mathbb{R})$ and $M \in C^0(C^1(\bar\Omega), C^1(\bar\Omega))$, every function in $\cF_{\mathrm{DRR}}(\fw, \fd, \sigma, M)\cup \cF_{\mathrm{FFN}}(\fw, \fd, \sigma, M)$ is $C^1(\overline{\Omega})$ (which implies that \eqref{eq:minmax_ell} below is well defined) and hence in $H^1(\Omega)$, i.e.,
$\cF(\fw, \fd, \sigma, M) \subset C^1(\overline{\Omega}) \subset H^1(\Omega)$
for both architectures. Though $\mathrm{ReLU}$ is not $C^1$, it is well known that we still have $\mathcal{F}(\fw, \fd, \mathrm{ReLU}) \subset H^1(\Omega)$ for both architectures.

\subsection{Discretised problems}\label{sec:disc_problems}
We need discrete versions of the objective function appearing in \eqref{eq:minMaxDiagonalProblem}. We shall henceforth require the regularity
$f \in L^2(\Omega) \cap C^0(\Omega)$.
For simplicity \corr{of exposition}, we assume that  $A=-\Delta + k\mathrm{Id}$ for a constant $k \geq 0$, generating the pairing
\[\langle Au, u-v \rangle = \int_\Omega \grad u \cdot \grad (u-v) + ku(u-v)\]
  (other cases can be easily handled with appropriate modifications). Take a set of collocation points $\{x_i\}_{i=1}^N \subset \Omega$ and define $\hatL\colon C^1(\overline{\Omega}) \times C^1(\overline{\Omega}) \to \mathbb{R}$ by
\begin{align*}
\hatL(u,v) := \frac{|\Omega|}{N} \sum_{i=1}^N \nabla u(x_i) \cdot (\nabla u(x_i) - \nabla v(x_i) ) + k u(x_i)(u(x_i)-v(x_i)) - \frac{|\Omega|}{N}\sum_{i=1}^N f(x_i)(u(x_i) - v(x_i)).  
\end{align*}
This is the discrete version of $L$. 
Similarly, the discrete version of $L_\gamma$ is $\hatL_\gamma\colon C^1(\overline{\Omega}) \times C^1(\overline{\Omega}) \to \mathbb{R}$ given by 
\begin{align*}
    \hatL_\gamma(u,v) := \hatL(u,v) - \frac{|\Omega|}{2\gamma N} \sum_{i=1}^N (u(x_i) - v(x_i))^2 + |\grad u(x_i) - \grad v(x_i)|^2.
\end{align*}
\begin{remark}\label{rem:dependence_on_gridpoints}
    One should bear in mind that $\hatL$ and $\hatL_\gamma$ (and all other discretised quantities that we shall introduce) are also functions of the grid points, i.e., $\hatL(u,v) = \hatL^{\{x_i\}}(u,v)$. A different choice of $\{x_i\}$ gives a different $\hatL$. We usually omit this dependence for aesthetic reasons.
\end{remark}
Let $\cXs$ and $\cXt$ be arbitrary subsets of $C^1(\bar\Omega)$. Later, we will take these sets as sets of neural networks representing the solution and test function space respectively (with  potentially different widths, depths and activations, etc., cf. the Petrov--Galerkin method) of the form \eqref{eq:defn_cF_general_M}, intersected with a ball of a carefully chosen radius. A discretised problem corresponding to \eqref{eq:minMaxDiagonalProblem} is
\begin{equation}
    \min_{u \in \cXs} \max_{v \in \cXt} \; \hatL_\gamma(u,v). \tag{$\widehat{\text{P}}$}\label{eq:discrete_minmax}
\end{equation}
In general,  $\cXs$ and $\cXt$ need not be subsets of $K$ and it could be that they are not rich enough for the above problem to be a good approximation \corr{of \eqref{eq:minMaxDiagonalProblem}} (see Remark \ref{rem:exact_constraints} below). In view of this and the relaxations we surveyed in \S \ref{sec:relaxations}, let us consider some alternatives.

\subsubsection{Discrete version of \eqref{eq:cts_ansatz} (Imposition of the boundary conditions)}\label{sec:imposition_bcs}
     To formulate a discrete version of \eqref{eq:cts_ansatz}, let us first assume the regularity
     $\psi \in C^0(\Omega)$
and   
     define the discrete obstacle constraint loss corresponding to $L_o$ as
\[\hatL_o(u) := \frac{|\Omega|}{N} \sum_{i=1}^N |(\psi(x_i)-u(x_i))^+|^2.\]  
For a given function $\tilde h \in C^1(\overline{\Omega})$  satisfying $\tilde h|_{\partial\Omega} =h$  and  taking
\begin{equation*}
    \mathcal{F}^{\mathfrak{s}}_0 := \mathcal{F}_0(\fw^{\mathfrak{s}}, \fd^{\mathfrak{s}}, \sigma^{\mathfrak{s}}) \qquad\text{and}\qquad \mathcal{F}^{\mathfrak{t}}_0 := \mathcal{F}_0(\fw^{\mathfrak{t}}, \fd^{\mathfrak{t}}, \sigma^{\mathfrak{t}})\end{equation*}
(as mentioned, a practical way to realise these sets is to use $\cF_{0,\eta}$ as defined in \eqref{eq:defnCf0Eta}), we pose 
\begin{equation}
   \min_{u \in \mathcal{F}_0^{\mathfrak{s}} + \tilde h}\max_{v \in \mathcal{F}_0^{\mathfrak{t}}  + \tilde h} \hatL_\gamma(u,v) + w_{o_1}\hatL_o(u) - w_{o_2}\hatL_o(v).\tag{$\widehat{\text{P}}_{\mathrm{bc}}$}\label{eq:disc_ansatz}
\end{equation}

\begin{remark}
The problem \eqref{eq:disc_ansatz} can be formulated as
\begin{equation}
   \min_{\tilde u \in \mathcal{F}_0^{\mathfrak{s}}}\max_{\tilde v \in \mathcal{F}_0^{\mathfrak{t}}} \hatL_\gamma(\tilde u+\tilde h,\tilde v+\tilde h) + w_{o_1}\hatL_o(\tilde u+\tilde h) - w_{o_2}\hatL_o(\tilde v+\tilde h)\label{eq:disc_ansatz_2}.
\end{equation}
Then if $\tilde u$ denotes a solution of \eqref{eq:disc_ansatz_2}, to recover a solution of \eqref{eq:disc_ansatz} we set $u := \tilde u + \tilde h.$ 
\end{remark}

\subsubsection{Discrete version of \eqref{eq:fullPenaltyLoss} (Full penalty approach)}
Here we wish to formulate the approach of \eqref{eq:fullPenaltyLoss} in the discrete setting. 
In order to measure the boundary loss for the discretised problem, we introduce $\{x_i^b\}_{i=1}^{N_b} \subset \partial\Omega$ as a set of  boundary collocation points, and assuming
$\psi \in C^0(\Omega),\; h \in C^0(\partial\Omega),$
defining
\[\hatL_b(u) := \frac{|\partial \Omega|}{N_b} \sum_{i=1}^{N_b} (u(x_i^b)-h(x_i^b))^2,\]
we can pose 
\begin{equation}
    \min_{u \in \cFs(\fd^{\mathfrak{s}}, \fw^{\mathfrak{s}}, \sigma^{\mathfrak{s}})} \max_{v \in \cFt(\fd^{\mathfrak{t}}, \fw^{\mathfrak{t}}, \sigma^{\mathfrak{t}})} \; \hatL_\gamma(u,v) + w_{o_1}\hatL_o(u)+ w_{b_1}\hatL_b(u) - w_{o_2}\hatL_o(v) - w_{b_2} \hatL_b(v).\tag{$\widehat{\text{P}}_{\mathrm{pen}}$}\label{eq:discrete_minmax_general}
\end{equation}

\subsubsection{Discrete version of \eqref{eq:proble_gen_cts} (Generalised approach)}\label{sec:generalised_approach_disc}
Now let us address the generalised problem \eqref{eq:proble_gen_cts}.  Defining the discrete versions $\hat R_1$ and $\hat R_2$ of $R_1$ and $R_2$ (by using $\hat L_b$ and $\hat L_o$), we can consider
\begin{equation}
\min_{u \in \cXs}\max_{v \in \cXt} \hatL_\gamma(u,v) + \hat R_1(u) - \hat R_2(v)\tag{$\widehat{\text{P}}_{\mathrm{gen}}$}\label{eq:minmax_general_new}
\end{equation}
given sets of neural networks $\cXs$ and $\cXt$. 
This structure generalises all of the above-mentioned approaches \corr{\textit{and therefore we will perform the error analysis (in \S \ref{sec:error_analysis}) for this problem}}. We will denote by
\[\hatu \quad\text{and}\quad \hatuA\] 
an exact solution to \eqref{eq:minmax_general_new} and a computed solution to \eqref{eq:minmax_general_new} via a numerical algorithm\footnote{\corr{In our numerical implementation, we will use a gradient descent ascent algorithm (GDA) to solve our discrete problem, see Algorithm \ref{alg:1} on page~\pageref{alg:1}. GDA can be thought of as an inexact version of Uzawa's algorithm.}} $\mathcal{A}$, respectively.

\begin{remark}[Exact satisfaction of the constraints]\label{rem:exact_constraints}
In certain situations, it is possible to choose $\cXs$ and $\cXt$ such that both are subsets of $K$. This may be regarded as an ideal setting, since all elements are feasible and no additional effort is required to enforce the constraints. Moreover, the error analysis presented in \S \ref{sec:error_analysis}  simplifies considerably in this case.

A common scenario where $\cXs, \cXt \subset K$ can be easily achieved is when the VI has the homogeneous boundary condition $h \equiv 0$ (a typical setting in literature) and when, with only a small loss of generality\footnote{If the obstacle $\psi$ belongs to $H^1_0(\Omega)$ then one can always transform the VI to a VI with zero obstacle by $u \mapsto u-\psi$; this essentially moves the obstacle into the source term.}, the obstacle is also zero. Then we can solve
\begin{equation*}
    \min_{u \in \cFs(\fd^{\mathfrak{s}}, \fw^{\mathfrak{s}}, \sigma^{\mathfrak{s}}, Q \circ M_\eta)} \max_{v \in \cFt(\fd^{\mathfrak{t}}, \fw^{\mathfrak{t}}, \sigma^{\mathfrak{t}}, Q \circ M_\eta)}  \hatL_\gamma(u,v)
\end{equation*}
where $M_\eta$ is as before and $Q$ is the map $Q(u)=u^2$. Satisfying the constraints in a more general setting (with non-zero boundary conditions) appears to be non-trivial. We could consider a class of neural networks that satisfies also the obstacle condition in addition to the boundary condition. Take $u \in \mathcal{F}_{0,\eta}(\fd,\fw, \sigma)$, a function $\tilde h \in C^1(\overline{\Omega})$ with $\tilde h|_{\partial\Omega} = h$ (like described above) and make the transformation $u \mapsto \max(u+\tilde h, \psi).$ Label the set of such functions $\mathcal{F}^\psi_h(\fd,\fw, \sigma)$, which is a subset of $K$. One could then consider
\begin{equation*}
   \min_{u \in \mathcal{F}^\psi_h(\fd^{\mathfrak{s}},\fw^{\mathfrak{s}}, \sigma^{\mathfrak{s}})}\max_{v \in \mathcal{F}^\psi_h(\fd^{\mathfrak{t}},\fw^{\mathfrak{t}}, \sigma^{\mathfrak{t}})} \hatL_\gamma(u,v),
\end{equation*}
however, from experience, this does not work as well as expected due in part to the nonsmooth $\max$ operation which hinders training. 
\end{remark}

Before we proceed to the error analysis, let us briefly discuss existence of solutions to these neural network problems.
\subsubsection{Remarks on existence}\label{remark_challenges_nonconvex_nonconcave}
It is well known that sets of neural networks may not be closed (or convex), hampering the use of standard theory to deduce existence of optimal points for the discretised problems. One way to work around this issue is to use the notion of \emph{quasi-minimisation} \citep{Shin2023}, see also \citet{BREVIS2022115716}. An element $\bar u \in U \subset X$ of a set $U$ in a Hilbert space $X$ is said to be a \emph{quasi-minimiser} of the functional $J\colon U \to \mathbb{R}$ if it satisfies $J(\bar u) \leq \inf_{u \in U} J(u) + \epsilon$ for some $\epsilon > 0$. One could then weaken the notion of global minmax points (see Definition \ref{defn:minmax} with the corresponding function $\mathfrak{f}$) and define a \emph{quasi-minimax point} as a point $(x^*,y^*)$ that satisfies
\[\mathfrak{f}(x^*, y) - \epsilon \leq \mathfrak{f}(x^*, y^*) \quad \forall y \in Y\quad\text{and} \quad \max_{y \in Y}\mathfrak{f}(x^*, y) \leq \max_{y \in Y} \mathfrak{f}(x,y) + \epsilon \quad \forall x \in X\]
for some $\epsilon > 0$, and then study the associated theory. However, as this is not the focus of our work, let us content ourselves with giving existence results under a simplified setting where the associated parameter space to the sets of neural networks is sufficiently regular.

Let us consider a general problem
\begin{equation}
\min_{u \in \cXs}\max_{v \in \cXt} \hat\ell(u,v)\tag{$\widehat{\text{P}}_{\hat\ell}$}\label{eq:minmax_ell}
\end{equation}
under the following assumption.
\begin{assumption}[Assumptions on \eqref{eq:minmax_ell}]\label{ass:on_hat_ell}
Let $\hat \ell\colon C^1(\overline{\Omega}) \times C^1(\overline{\Omega}) \to \mathbb{R}$ be a given map and assume that $\cXs$ and $\cXt$ are of the form
  $$   \cXs \subset \cF(\fds, \fws, \sigmas, M^\mathrm{s}),\quad
     \cXt \subset \cF(\fdt, \fwt, \sigmat, M^\mathrm{t}),
  $$  
with associated  parameter spaces $\Theta^{\mathfrak{s}}$ and $\Theta^{\mathfrak{t}}$ respectively  and where $\sigma^{\mathfrak{s}}, \sigma^{\mathfrak{t}} \in C^1(\mathbb{R})$ and $M^{\mathrm{s}}, M^{\mathrm{t}}\colon C^1(\overline{\Omega}) \to C^1(\overline{\Omega})$ are given.  Furthermore, assume that
\begin{enumerate}[label=(\roman*)]\itemsep=0pt
    \item\label{item:ass_on_ell} $\hat \ell\colon C^1(\overline{\Omega}) \times C^1(\overline{\Omega}) \to \mathbb{R}$ is continuous,
    \item $M^{\mathrm{s}}, M^{\mathrm{t}} \in C^0(C^1(\overline{\Omega}), C^1(\overline{\Omega}))$,
\item $\Theta^{\mathfrak{s}}, \Theta^{\mathfrak{t}}$ are non-empty and compact.
\end{enumerate}
\end{assumption}
Here, we take the perspective that a \emph{solution} to problem \eqref{eq:minmax_ell} is associated with solving the finite-dimensional problem
\[\min_{\theta \in \Theta^{\mathfrak{s}}}\max_{\vartheta \in \Theta^{\mathfrak{t}}} \hat \ell(u(\theta, \cdot) ,v(\vartheta, \cdot)).\]
To be precise, in Assumption \ref{ass:on_hat_ell} we enforced conditions on the parameter spaces $\Theta^{\mathfrak{s}}$ and $\Theta^{\mathfrak{t}}$ associated to $\cXs$ and $\cXt$ which constrain the latter sets: elements of $\cXs$ and $\cXt$ have a limitation on how large the weights can be. 
\begin{proposition}\label{lem:existence_discrete_general}
Let Assumption \ref{ass:on_hat_ell} hold. Then there exists a solution $\hatu \in \cXs$  to the problem \eqref{eq:minmax_ell}.
\end{proposition}
\begin{proof}
As discussed above, we have by definition of our solution concept that
\[\min_{u \in \cXs}\max_{v \in \cXt} \hat \ell(u,v) = \min_{\theta \in \Theta^{\mathfrak{s}}}\max_{\vartheta \in \Theta^{\mathfrak{t}}} \hat \ell(u(\theta, \cdot) ,v(\vartheta, \cdot)).\]
Using Lemma \ref{lem:NNisC1_general}  and the continuity of $\hat\ell$, it follows that $(\theta, \vartheta) \to \hat \ell(u_\theta, v_\vartheta)$ is continuous. Applying \citet[Proposition 11]{Jordan}, we get the result.
\end{proof}
We now address \eqref{eq:disc_ansatz},  \eqref{eq:discrete_minmax_general} (note that the former is not a special case of the latter because the architectures are potentially different) \corr{and \eqref{eq:minmax_general_new}}.
\begin{corollary}\label{cor:existence_disc_problems}
We have the following.
\begin{enumerate}[label=(\roman*)]\itemsep=0cm
 \item\label{item:1} In the context of \eqref{eq:disc_ansatz}, let $\sigma^{\mathfrak{s}}, \sigma^{\mathfrak{t}} \in C^1(\mathbb{R})$ and let the parameter spaces $\Theta^{\mathfrak{s}}$ and $\Theta^{\mathfrak{t}}$ associated to $\cFs_0$ and $\cFt_0$ respectively be non-empty and compact. Then  the problem \eqref{eq:disc_ansatz}  possesses a solution.
\item\label{item:2} In the context of \eqref{eq:discrete_minmax_general}, let $\sigma^{\mathfrak{s}}, \sigma^{\mathfrak{t}} \in C^1(\mathbb{R})$ and let the parameter spaces $\Theta^{\mathfrak{s}}$ and $\Theta^{\mathfrak{t}}$ associated to $\cFs(\fd^{\mathfrak{s}}, \fw^{\mathfrak{s}}, \sigma^{\mathfrak{s}})$ and $\cFt(\fd^{\mathfrak{t}}, \fw^{\mathfrak{t}}, \sigma^{\mathfrak{t}})$ respectively be non-empty and compact. Then the problem \eqref{eq:discrete_minmax_general} possesses a solution.
\item\label{item:3} In the context of \corr{\eqref{eq:minmax_general_new}, let $(\theta, \vartheta) \mapsto (u(\theta, \cdot), v(\vartheta, \cdot))$ be continuous from the respective parameter spaces into $C^1(\overline{\Omega})\times C^1(\overline{\Omega})$ for all $u \in \cXs$ and $v \in \cXt$}  and let the parameter spaces $\Theta^{\mathfrak{s}}$ and $\Theta^{\mathfrak{t}}$ associated to $\cXs$ and $\cXt$ respectively be non-empty and compact. Then the problem  \corr{\eqref{eq:minmax_general_new}} possesses a solution.
\end{enumerate}
\end{corollary}
\begin{proof}
For \ref{item:1}, we consider the reformulation \eqref{eq:disc_ansatz_2} of \eqref{eq:disc_ansatz}. Thanks to Proposition \ref{lem:existence_discrete_general}, it suffices to show that $M_\eta \in C^0(C^1(\overline{\Omega}), C^1(\overline{\Omega}))$ and that the resulting loss function is continuous from $C^1(\overline{\Omega})\times C^1(\overline{\Omega})$ to $\mathbb{R}$. The former is clear because $\eta \in C^1(\overline{\Omega})$. For the latter, we begin by setting $\tilde h=0$ for  simplicity. Define the pointwise evaluations $S_i(u) := u(x_i)$ and $T_i(u) := \grad u(x_i)$. Then we can  write
\begin{align*}
    \hatL(u,v) = \frac{|\Omega|}{N}\sum_{i=1}^N (T_i(u), T_i(u) - T_i(v))_{\mathbb{R}^n} + \frac{|\Omega|}{N}\sum_{i=1}^N  kS_i(u)(S_i(u)-S_i(v)) - \frac{|\Omega|}{N}\sum_{i=1}^N f(x_i)(S_i(u) - S_i(v)).
\end{align*}
As $S_i\colon C^1(\overline{\Omega}) \to \mathbb{R}$ and $T_i \colon C^1(\overline{\Omega}) \to \mathbb{R}^n$ are clearly continuous, the composition of continuous maps being continuous gives continuity of $\hatL\colon C^1(\overline{\Omega}) \times C^1(\overline{\Omega}) \to \mathbb{R}$. Writing also
\begin{align*}
\hatL_{o}(u) = \frac{|\Omega|}{N}\sum_{i=1}^N |(S_i(\psi)-S_i(u))^+|^2,
\end{align*}
the continuity of the functions $(\cdot)^+, |\cdot|^2 \colon \mathbb{R} \to \mathbb{R}$ yields continuity of this object too. The remaining terms in the objective function in \eqref{eq:disc_ansatz_2} can be tackled similarly. The claim \ref{item:2} follows similarly,  \corr{whereas \ref{item:3} follow as in the proof of Proposition \ref{lem:existence_discrete_general}, wherein the evocation of Lemma \ref{lem:NNisC1_general} is unnecessary as its result is directly assumed.}
\end{proof}
It is clear that existence for the problem \eqref{eq:discrete_minmax} can also be handled with a similar argument like above.
 
Note that when transiting from the continuous problem \eqref{eq:prelim1} to any of the associated discrete problems above, the (strong) convexity and concavity properties for the former are lost. This is due to the structure of the mapping from the neural network weights to the neural network output $(\theta_1,\theta_2) \mapsto (u_{\theta_1}, v_{\theta_2})(x)$. As a consequence, the discrete problems are non-convex in the $\theta_1$ variable and non-concave in the $\theta_2$ variable which imposes major difficulties concerning the characterisation and efficient computation of solutions. For an overview in the smooth setting we refer to the recent works \citet{Jordan,razaviyayn2020nonconvex} and the references therein. Clearly, the design of tailored numerical solvers for such neural network based problems is a very important research task which, however, goes beyond the scope of our present work. Subsequently, we  focus on a comprehensive error analysis regarding the true and a computed solution for \eqref{eq:VI}. 
 
\subsection{Error analysis}\label{sec:error_analysis}
Recall that we consider the general problem \eqref{eq:minmax_general_new}, as described in \S \ref{sec:generalised_approach_disc}. As mentioned before, by choosing the sets $\cXs$ and $\cXt$ and the weights appearing in $R_1$ and $R_2$ appropriately, we can cover both cases \eqref{eq:disc_ansatz} and \eqref{eq:discrete_minmax_general} in this formulation.
 From the error estimate in Lemma \ref{lem:error_estimate_G_gamma} and the minimisation property Lemma \ref{lem:regularisedGapVI} \ref{item:uiffGgamma0}, 
\begin{align}
    \left( C_a - \frac{1}{2\gamma} \right)\| \hatuA - u^* \|_{\X}^2 &\leq  G_\gamma(\hatuA) - 
     G_\gamma(u^*) - \langle Au^*-f, u^*-\hatuA\rangle, \label{eq:error_upper_estimate}
\end{align} 
so we need to focus on estimating the right-hand side in an appropriate way. 
\corr{\corr{Associated to \eqref{eq:minmax_general_new},} define $\hatG_\gamma\colon \cXs \to \mathbb{R}$ by
\begin{equation}
\hatG_\gamma(u) := \corr{\sup_{v \in \cXt}} \hatL_\gamma(u,v) + \hat  R_1(u) -\hat R_2(v).\label{eq:defn_hatG_gamma_general}
\end{equation}
This plays the role of a discrete version of the gap function associated to \eqref{eq:minmax_general_new}.  \corr{(Existence of a \textit{maximiser} to \eqref{eq:defn_hatG_gamma_general} can be shown under appropriate assumptions in a similar way to Proposition \ref{lem:existence_discrete_general}, utilising the continuity of $\hat L_\gamma + \hat R_1 - \hat R_2$).} Note that if $\hatL_\gamma(u,u) + \hat  R_1(u) -\hat R_2(u) \geq 0$ and $\cXs \subseteq \cXt$, then it is easy to see that $\hatG_\gamma(u) \geq 0$ for all $u \in \cXs$. Without imposing Assumption \ref{ass:on_hat_ell}, we may still deduce that $\hatG_\gamma(\cdot) $ is finite everywhere by exploiting the fact that, by construction, $\cXt \subset V$.
\begin{remark}
Suppose that $\cFs_0 \subseteq \cFt_0$  and $w_{o_1}=w_{o_2}=0$. If $u \in \cFs_0$ satisfies
    \begin{equation*}
    \frac{|\Omega|}{N} \sum_{i=1}^N \nabla u(x_i) \cdot (\nabla u(x_i) - \nabla v(x_i) ) - \frac{|\Omega|}{N}\sum_{i=1}^N f(x_i)(u(x_i) - v(x_i)) \leq 0 \quad \forall v \in \cFt_0,  \end{equation*}    
    then $\hatG_\gamma(u) = 0.$ That is, solutions of this VI solve the discrete  problem \eqref{eq:disc_ansatz}.
\end{remark}}
Now let us recall (from \S \ref{sec:generalised_approach_disc}) that we denote by $\hatu$ a solution of \eqref{eq:minmax_general_new},  so that
\begin{equation}\label{mh:24*}
\hatG_\gamma (\hatu) \leq \hatG_\gamma(u) \quad \forall u \in \cXs.
\end{equation}
Define $H_\gamma\colon V \to \mathbb{R}$ by
\[H_\gamma(u) :=\corr{\sup_{v \in \cXt}}L_\gamma(u,v) + R_1(u) - R_2(v),\]
which can be thought of as the continuous version of $\hatG_\gamma$. Using the same argument as above, $H_\gamma(\cdot)$ is finite everywhere.
We begin with the following decomposition of the right-hand side of \eqref{eq:error_upper_estimate}: for arbitrary $\bar{u} \in \cXs$, we have
\begin{align*}
G_\gamma(\hatuA) - 
     G_\gamma(u^*)  &\leq \underbrace{G_\gamma(\hatuA) -    H_\gamma(\hatuA)}_{I}  + \underbrace{H_\gamma(\hatuA) - \hatG_\gamma(\hatuA) + \hatG_\gamma(\bar u) - H_\gamma (\bar u)}_{II}\\
     &\quad + \underbrace{\hatG_\gamma(\hatuA ) -       \hatG_\gamma(\hatu)}_{III} + \underbrace{H_\gamma (\bar u) - H_\gamma(u^*)}_{IV} + \underbrace{H_\gamma(u^*) - G_\gamma(u^*)}_{V}
\end{align*}
where we used $\hatG_\gamma(\hatu)\leq\hatG_\gamma(\bar u)$ by \eqref{mh:24*}.
We need the following two lemmas, wherein $\mathfrak{R}\colon V^*\to V$ is the Riesz map.

\begin{lemma}\label{lem:useful_estimate_NEW}
    For all $u, v \in \X$,
        \begin{equation}
L_\gamma(u,z)-L_\gamma(u,v)      =  \frac{1}{2\gamma}\norm{v-z}{V}^2 + \left(\mathfrak{R}(Au-f) + \frac{z-u}{\gamma}, v-z\right)_V.\label{eq:useful_estimate_NEW}
    \end{equation} 
\end{lemma}
\begin{proof}
    This follows easily from the definition of $L_\gamma$     and the three-point identity     
\[\norm{v-u}{}^2 - \norm{z-u}{}^2  =   \norm{v-z}{}^2 + 2(z-u, v-z).\]
    \end{proof}
\begin{lemma}For all $u,v \in H^1(\Omega)$, we have for $i=1,2$ the estimate
\begin{align}
\nonumber R_i(u)-R_i(v) &\leq 2w_{b_i}(\norm{h}{L^2(\partial\Omega)}+\norm{u}{L^2(\partial\Omega)})\norm{u-v}{L^2(\partial\Omega)}\\
&\quad + w_{o_i}(\norm{u+v}{L^2(\Omega)}+2\norm{\psi}{L^2(\Omega)})\norm{u-v}{L^2(\Omega)} - w_{b_i}\norm{u-v}{L^2(\partial\Omega)}^2.\label{eq:penalty_estimate}
\end{align}
\end{lemma}
\begin{proof}
Let us just set $i=1$. We start with
\begin{align*}
R_1(u)  - R_1(v)  &=  w_{b_1}\int_{\partial\Omega} (u - h)^2 - (v-h)^2 + w_{o_1}\int_\Omega |(\psi-u)^+|^2 - |(\psi-v)^+|^2.
\end{align*}
By straightforward computations, we estimate  
\begin{align*}
\int_{\partial\Omega} (u - h)^2 - (v-h)^2 &= \int_{\partial\Omega}2(u -h)(u - v) - |u-v|^2
\leq 2(\norm{h}{L^2(\partial\Omega)}+\norm{u}{L^2(\partial\Omega)})\norm{u-v}{L^2(\partial\Omega)}\\
&\quad- \norm{u-v}{L^2(\partial\Omega)}^2,
\end{align*}
while for the second integral, using the fact that $(\cdot)^+$ is Lipschitz, we get
\begin{align*}
\int_\Omega |(u - \psi)^+|^2 - |(v-\psi)^+|^2 
&\leq \int_\Omega |u-v||(\psi-u)^+ + (\psi-v)^+)|\\
&\leq (\norm{u+v}{L^2(\Omega)}+2\norm{\psi}{L^2(\Omega)})\norm{u-v}{L^2(\Omega)}.
\end{align*}
\end{proof}

From now on we assume
\begin{equation}\label{eq:uniform_bound_NN_spaces}
\begin{aligned}
\exists M^{\mathrm{s}}, B^{\mathrm{s}} > 0 &: \quad \norm{u}{V} \leq M^{\mathrm{s}}  \quad\text{and}\quad \norm{u}{L^2(\partial\Omega)} \leq B^{\mathrm{s}} \quad\forall u \in \cXs,\\
\exists M^{\mathrm{t}}, B^{\mathrm{t}} > 0 &: \quad \norm{u}{V} \leq M^{\mathrm{t}} \quad\text{and}\quad \norm{u}{L^2(\partial\Omega)} \leq B^{\mathrm{t}} \quad\forall u \in \cXt.
\end{aligned}    
\end{equation}
\corr{In the sequel, we will make use of the  assumption \eqref{eq:conditions_on_A_1} on the boundedness of $A$.}
\begin{proposition}[Estimate on I]
We have
\begin{align*}
G_\gamma(\hatuA) -    H_\gamma(\hatuA) 
&\leq \max_{v \in K \cap Y}\corr{\inf_{w \in \cXt}}  \Big( K_1\norm{w-v}{V} + K_2\norm{w-v}{L^2(\Omega)} + K_3\norm{w-v}{L^2(\partial\Omega)} \\
&\quad\quad\quad\quad\quad\quad\quad- w_{b_2}\norm{w-v}{L^2(\partial\Omega)}^2 - \frac{1}{2\gamma}\norm{w-v}{V}^2 \Big)- R_1(\hatuA)  
\end{align*}
where
$$
K_1 := C_bM^{\mathrm{s}} + \norm{f}{V^*} + \frac{1}{\gamma}(M^{\mathrm{s}} + M^{\mathrm{t}}), \quad K_2 := w_{o_2}(r+M^{\mathrm{t}}+2\norm{\psi}{L^2(\Omega)}),\quad K_3 := 2w_{b_2}(\norm{h}{L^2(\partial\Omega)}+B^{\mathrm{t}}).
$$
\end{proposition}
\begin{proof}
We begin with
\begin{align*}
G_\gamma(\hatuA) -    H_\gamma(\hatuA) 
&= -R_1(\hatuA) + \max_{v \in K \cap Y} L_\gamma(\hatuA, v) - \corr{\sup_{w \in \cXt}}\left(L_\gamma(\hatuA, w) - R_2(w)\right)\\
&=- R_1(\hatuA) + \max_{v \in K \cap Y}\corr{\inf_{w \in \cXt}}\left( L_\gamma(\hatuA, v) -L_\gamma(\hatuA, w) + R_2(w)\right).
\end{align*}
Now from \eqref{eq:useful_estimate_NEW}, we can estimate the middle two terms as
\begin{align*}
 L_\gamma(\hatuA, v) -L_\gamma(\hatuA, w) 
&=   \frac{1}{2\gamma}\norm{w-v}{}^2 + (\mathfrak{R}(A\hatuA-f) + \frac{v-\hatuA}{\gamma}, w-v)_V  \\
&\leq   \frac{1}{2\gamma}\norm{w-v}{}^2 + (\norm{A\hatuA-f}{V^*} + \frac{1}{\gamma}\norm{w-\hatuA}{V})\norm{w-v}{V}  - \frac{1}{\gamma}\norm{w-v}{V}^2 \\
&\leq   -\frac{1}{2\gamma}\norm{w-v}{}^2 + (C_b\norm{\hatuA}{V} + \norm{f}{V^*} + \frac{1}{\gamma}(M^{\mathrm{t}}+\norm{\hatuA}{V}))\norm{w-v}{V},
\end{align*}
\corr{where we used \eqref{eq:conditions_on_A_1} and \eqref{eq:uniform_bound_NN_spaces}.} Plugging this in, writing $R_2(w)=R_2(w)-R_2(v)$  and using the estimate on $R_2$ from \eqref{eq:penalty_estimate}, we get the result.
\end{proof}
\begin{proposition}[Estimate on II]
For all $\bar u \in \cXs$, we have
\begin{align*}
H_\gamma(\hatuA) - \hatG_\gamma(\hatuA)  + \hatG_\gamma(\bar u) - H_\gamma (\bar u) 
&\leq 2\sup_{\substack{u \in \cXs\\ v \in \cXt}} |L_\gamma(u, v) + R_1(u) - R_2(v) - (\hatL_\gamma(u, v) + \hat R_1(u) - \hat R_2(v))|.
\end{align*} 
\end{proposition}
\begin{proof}
\corr{
Define $F(u,v) = L_\gamma(u,v) + R_1(u) - R_2(v)$ and $\hat F(u,v) = \hat L_\gamma(u,v) + \hat R_1(u) - \hat R_2(v)$. Then by definition,
\begin{align*}
H_\gamma(\hatuA) - \hatG_\gamma(\hatuA)  + \hatG_\gamma(\bar u) - H_\gamma (\bar u) &= \sup_{v \in \cXt} F(\hatuA,v) - \sup_{v \in \cXt} \hat F(\hatuA, v) + \sup_{v \in \cXt} \hat F(\bar u, v) - \sup_{v \in \cXt} F(\bar u,v)\\
&\leq \sup_{v \in \cXt} |F(\hatuA,v) -  \hat F(\hatuA, v)| + \sup_{v \in \cXt} |\hat F(\bar u, v) -  F(\bar u,v)|\\
&\leq 2\sup_{\substack{u \in \cXs\\ v \in \cXt}} |F(u,v) -  \hat F(u, v)|
\end{align*}
which is precisely the desired result.
}

\end{proof}
Now for the fourth difference the following estimate holds true.
\begin{proposition}[Estimate on IV]
For all $\bar u \in \cXs$, we have
\begin{align*}
H_\gamma (\bar u) - H_\gamma(u^*) 
&\leq  K_4\norm{\bar u - u^*}{V}   + K_5\norm{\bar u-u^*}{L^2(\Omega)} + K_6\norm{\bar u - u^*}{L^2(\partial\Omega)} -w_{b_1}\norm{\bar u - u^*}{L^2(\partial\Omega)}^2\\
&\quad  - \frac{1}{2\gamma}\norm{u^*-\bar u}{V}^2
\end{align*}
where, \corr{with $C_1 := \norm{Au^*-f}{V^*}$,}
\begin{align*}
K_4 &:= C_b(M^{\mathrm{s}} + M^{\mathrm{t}}) + C_1 + \frac{1}{\gamma}(\norm{u^*}{V}+M^{\mathrm{t}}), \quad  
K_5 := w_{o_1}(M^{\mathrm{s}}+\norm{u^*}{V} + 2\norm{\psi}{L^2(\Omega)}),\\
K_6 &:= w_{b_1}2(\norm{h}{L^2(\partial\Omega)}+B^{\mathrm{s}}).
\end{align*}
\end{proposition}
\begin{proof}
The left-hand side satisfies
\begin{align*}
H_\gamma (\bar u) - H_\gamma(u^*) &\leq \corr{\sup_{v \in \cXt}} \left(L_\gamma(\bar u, v) + R_1(\bar u)  - L_\gamma(u^*,v) - R_1(u^*)\right).
\end{align*}
We have
\begin{align*}
 L_\gamma(\bar u, v) - L_\gamma(u^*,v) &= L(\bar u,v) - L(u^*,v) + \frac{1}{2\gamma}\left(\norm{u^*-v}{V}^2 - \norm{\bar u - v}{V}^2\right).
\end{align*}
Firstly, let us estimate the first two terms on the right-hand side:
\begin{align*}
    L(\bar{u}, v) - L(u^*, v) &= \langle A\bar u -f, \bar u - v \rangle - \langle Au^*-f, u^*-v \rangle  = \langle A(\bar u-u^*), \bar u - v \rangle  + \langle Au^* -f, \bar u - u^* \rangle     \\    
   &\leq C_b(M^{\mathrm{s}}+M^{\mathrm{t}})\norm{\bar u - u^*}{V} + C_1\norm{\bar u - u^*}{V}.
\end{align*}
\corr{Here, we set $C_1 := \norm{Au^*-f}{V^*}$ and utilised \eqref{eq:conditions_on_A_1} and \eqref{eq:uniform_bound_NN_spaces}.} Now for the squared norm terms, we again use the three-point equality to obtain
\begin{align*}
    \norm{u^*-v}{V}^2 - \norm{\bar u - v}{V}^2 = 2(u^* - v, u^* - \bar u) - \norm{u^*-\bar u}{V}^2
        \leq 2(\norm{u^*}{V}+M^{\mathrm{t}})\norm{u^* - \bar u}{V} - \norm{u^*-\bar u}{V}^2.
\end{align*}
Putting everything together and using again the estimate \eqref{eq:penalty_estimate}, we obtain
\begin{align*}
H_\gamma (\bar u) - H_\gamma(u^*) 
&\leq  C_b(M^{\mathrm{s}}+M^{\mathrm{t}})\norm{\bar u - u^*}{V} + C_1\norm{\bar u - u^*}{V} + \frac{1}{\gamma}(\norm{u^*}{V}+M^{\mathrm{t}})\norm{u^* - \bar u}{V} \\
&\quad - \frac{1}{2\gamma}\norm{u^*-\bar u}{V}^2 + w_{b_1}2(\norm{h}{L^2(\partial\Omega)}+\norm{\bar u}{L^2(\partial\Omega)})\norm{\bar u-u^*}{L^2(\partial\Omega)} \\
&\quad + w_{o_1}(M^{\mathrm{s}}+\norm{u^*}{V} + 2\norm{\psi}{L^2(\Omega)})\norm{\bar u-u^*}{L^2(\Omega)}-w_{b_1}\norm{\bar u - u^*}{L^2(\partial\Omega)}^2.
\end{align*}
\end{proof}Combining the previous results, we have shown 
\begin{align}
\nonumber  \left(C_a-\frac{1}{2\gamma}\right)\norm{\hatuA-u^*}{V}^2    
     &\leq  \max_{v \in K \cap Y}\corr{\inf_{w \in \cXt}}   K_1\norm{w-v}{V} + K_2\norm{w-v}{L^2(\Omega)}  + K_3\norm{w-v}{L^2(\partial\Omega)} 
      \\ 
       \nonumber &\quad +  \inf_{u \in \cXs} \left(K_4\norm{u - u^*}{V}   + K_5\norm{u-u^*}{L^2(\Omega)} +  K_6\norm{u-u^*}{L^2(\partial\Omega)} \right) \\
          \nonumber &\quad + 2\sup_{\substack{u \in \cXs\\v \in \cXt}} |L_\gamma(u, v) + R_1(u) - R_2(v) - (\hatL_\gamma(u, v) + \hat R_1(u) - \hat R_2(v))|\\
     \nonumber &\quad + \hatG_\gamma(\hatuA ) -      \hatG_\gamma(\hatu)\\
     &\quad +  H_\gamma(u^*) - G_\gamma(u^*) + \langle Au^*-f, u^*-\hatuA\rangle  - R_1(\hatuA).  \label{eq:full_error_estimate}
\end{align}
The quantity comprising the first two lines on the right-hand side above
\begin{align}
    \nonumber \xi_{\mathrm{app}} &:= \max_{v \in K \cap Y}\corr{\inf_{w \in \cXt}}   K_1\norm{w-v}{V} + K_2\norm{w-v}{L^2(\Omega)}  + K_3\norm{w-v}{L^2(\partial\Omega)}\\ 
       &\quad +  \inf_{u \in \cXs} \left(K_4\norm{u - u^*}{V}   + K_5\norm{u-u^*}{L^2(\Omega)} +  K_6\norm{u-u^*}{L^2(\partial\Omega)} \right)\label{eq:defn_approx_error}
\end{align}
is known as the \emph{approximation error}. It is related to how well the spaces $\cXs$ and $\cXt$  approximate $u^*$ and the set $K \cap Y$ respectively. The \emph{statistical error} is the quantity
\begin{align*}
    \xi_{\mathrm{stat}} := \sup_{\substack{u \in \cXs\\v \in \cXt}} |L_\gamma(u, v) + R_1(u) - R_2(v) - (\hatL_\gamma(u, v) + \hat R_1(u) - \hat R_2(v))|,
\end{align*}
i.e., it is a measure of the error arising from the numerical approximation of the integral. 
We shall estimate both of these in the next subsections. The term 
\[\xi_{\mathrm{opt}}= \hatG_\gamma(\hatuA ) - \hatG_\gamma(\hatu)\] 
is called the \emph{optimisation error}; as explained in \S \ref{remark_challenges_nonconvex_nonconcave}, this is not something we shall explore in this paper. Finally, the two terms
$\langle Au^*-f, u^*-\hatuA\rangle - R_1(\hatuA)$ and $H_\gamma(u^*) - G_\gamma(u^*)$
essentially arise due to the presence of $K$ and the inexactness of our approximating spaces. Let us deal with these terms first.
\subsubsection{Bounds on the error arising from the constraint set}
\begin{proposition}
    If $\hatuA|_{\partial\Omega} = h$ (which is the case in the situation of \eqref{eq:disc_ansatz}) and $u^* \in H^2(\Omega)$, then for every $\epsilon > 0$, if the penalty weight $w_{o_1}$ satisfies $w_{o_1} \geq \frac{1}{4\epsilon}\norm{Au^*-f}{L^2(\Omega)}^2$, we have
    \[\langle Au^*-f, u^*-\hatuA\rangle - R_1(\hatuA) \leq \epsilon.\]
\end{proposition}
\begin{proof}
We have that $\max(\psi, \hatuA)$  belongs to $K$ (since $\hatuA|_{\partial\Omega} = h$ and $\psi|_{\partial\Omega} \leq h$). Thus, using Young's inequality with $\epsilon > 0$ and noting that $\max(\psi, \hatuA)-\hatuA = (\psi-\hatuA)^+$, 
\begin{align*}
    \langle Au^*-f, u^*-\hatuA\rangle &= \langle Au^*-f, u^*-\max(\psi, \hatuA) \rangle + \langle Au^*-f, \max(\psi, \hatuA)-\hatuA\rangle \\
    &\leq \norm{Au^*-f}{L^2(\Omega)}\norm{(\psi-\hatuA)^+}{L^2(\Omega)}\\
        &\leq \epsilon + \frac{1}{4\epsilon}\norm{Au^*-f}{L^2(\Omega)}^2\norm{(\psi-\hatuA)^+}{L^2(\Omega)}^2.
\end{align*}
The claim then follows from 
\begin{align*}
    \langle Au^*-f, u^*-\hatuA\rangle - R_1(\hatuA)  &\leq \epsilon + \left(\frac{1}{4\epsilon}\norm{Au^*-f}{L^2(\Omega)}^2- w_{o_1}\right)\norm{(\psi-\hatuA)^+}{L^2(\Omega)}^2.
\end{align*}
\end{proof}
\corr{Here and below, we recall the Sobolev trace inequality, see e.g. \citet[Theorem 8.7, Chapter 1]{MR895589}.}    
\begin{proposition}[Estimate on V]
For every $\epsilon > 0$, there exists an $\alpha_0 > 0$ such that if the penalty parameters $w_{o_2}, w_{b_2}$ satisfy $\min(w_{o_2}, w_{b_2}) \geq \alpha_0$, then
\begin{align*}
H_\gamma(u^*) - G_\gamma(u^*) &\leq \epsilon - \frac{1}{2\gamma}\norm{u^*-v^*}{V}^2 - R_2(v^*),
\end{align*}
where $v^* \in \argmax_{w \in H^1(\Omega)} L_\gamma(u^*,w) -R_2(w)$.
\end{proposition}
\begin{proof}
\corr{
In this proof, $P_K\colon H^1(\Omega) \to K$ denotes the orthogonal projection operator onto the set $K$.  Let us set $v_\alpha \in \argmax_{w \in H^1(\Omega)} L_\gamma(u^*,w) -R_2(w)$ where $\alpha$ represents the penalty parameters in $R_2$. We begin with the following calculation, wherein we employ the facts that $G_\gamma(u^*)= 0$ and $\cXt\subset H^1(\Omega)$:
\begin{align}
\nonumber H_\gamma(u^*) - G_\gamma(u^*) &= \sup_{w \in \cXt} L_\gamma(u^*,w) -R_2(w)
\leq \max_{w \in H^1(\Omega)} L_\gamma(u^*,w) -R_2(w)\\
\nonumber &= \langle Au^*-f, u^*-v_\alpha\rangle - \frac{1}{2\gamma}\norm{u^*-v_\alpha}{V}^2 - R_2(v_\alpha)\\
\nonumber &\leq \langle Au^*-f, P_K(v_\alpha)-v_\alpha\rangle - \frac{1}{2\gamma}\norm{u^*-v_\alpha}{V}^2 - R_2(v_\alpha)\\
&\leq \norm{Au^*-f}{V^*}\norm{P_K(v_\alpha)-v_\alpha}{V} - \frac{1}{2\gamma}\norm{u^*-v_\alpha}{V}^2 - R_2(v_\alpha).\label{eq:to_control}
\end{align}
It remains only to estimate the first term on the right-hand side above, i.e., we are left to show that
$\norm{Au^*-f}{V^*}\norm{P_K(v_\alpha)-v_\alpha}{V} \leq \epsilon$
for arbitrary $\epsilon$ provided we make $\alpha$ large enough. This can be achieved if we show that $v_\alpha$ converges strongly in $V$ to some $\hat v \in K$ as $\alpha \to \infty$ because of the estimate
\begin{equation}
    \norm{P_K(v_\alpha)-v_\alpha}{V} \leq \norm{P_K(v_\alpha) - P_K(\hat v)}{V} + \norm{\hat v-v_\alpha}{V} \leq 2\norm{\hat v-v_\alpha}{V}\label{eq:new_1}
\end{equation}
where we used the triangle inequality and that $\hat v = P_K(\hat v)$. 

We will show this desired convergence of $v_\alpha$ in multiple steps.
\begin{enumerate}[label=(\roman*), wide, labelwidth=!, labelindent=0pt]\itemsep=0em
\item \textbf{Weak subsequential convergence of $v_\alpha$ and feasibility of limit.} 
Let us first show that $v_\alpha$ converges weakly for a subsequence. Writing occasionally $R_2^\alpha$ instead of $R_2$ to emphasise the dependence of $R_2$ on $\alpha$ and setting
\[J_\alpha(w) := R_2^\alpha(w)-L_\gamma(u^*,w),\]
we have by definition of $v_\alpha$ that $J_\alpha(v_\alpha) \leq J_\alpha(w)$ for all $w \in H^1(\Omega)$; taking $w=u^*\in K$ and manipulating gives
\begin{equation}
R_2^\alpha(v_\alpha) + \frac{1}{4\gamma}\norm{u^*-v_\alpha}{V}^2 \leq \gamma\norm{Au^* -f}{V^*}^2.\label{eq:leads_to_conv}
\end{equation}
This uniform bound implies the existence of $\hat v \in V$ such that, for a subsequence (that we relabel), $v_\alpha \weaklyto \hat v$ in $V$ and strongly in $L^2(\Omega)$. Since $R_2^\alpha(v_\alpha) = w_{o_2} \int_\Omega |(\psi-v_\alpha)^+|^2 + w_{b_2} \int_{\partial\Omega} (v_\alpha -h)^2$, if we divide the above inequality by $w_{o_2}$, disregard the boundary term and vice versa with $w_{b_2}$, we find $\int_\Omega |(\psi-v_\alpha)^+|^2 +   \int_{\partial\Omega} (v_\alpha -h)^2 \to 0$ in the limit $w_{b_2}, w_{o_2} \to \infty$. Using the convergence result for $v_\alpha$ and the fact that the trace operator is linear and continuous as well as the weak lower semicontinuity of the functionals involved, we then get $\hat v \in K$.
\item \textbf{Strong subsequential convergence of $v_\alpha \to \hat v$.} Now, we prove that the above weak convergence (for a subsequence) can be upgraded to strong. Let
$J(w) := -L_\gamma(u^*,w)$.
We have $J_\alpha(w) \geq J(w)$ and so, using weak lower semicontinuity,
\begin{align*}
\liminf_{\alpha \to \infty} J_\alpha(v_\alpha) &\geq \liminf_{\alpha \to \infty} J(v_\alpha) = 
\liminf_{\alpha \to \infty} \frac{1}{2\gamma}\norm{u^*-v_\alpha}{V}^2 - \langle Au^*-f, u^*-v_\alpha \rangle\\
&\geq \frac{1}{2\gamma}\norm{u^*-\hat v}{V}^2 - \langle Au^*-f, u^*-\hat v \rangle = J(\hat v).
\end{align*}
On the other hand, since $J_\alpha(v_\alpha) \leq J_\alpha(\hat v)$ and using the fact that $\hat v \in K$,
\begin{align*}
\liminf_{\alpha \to \infty} J_\alpha(v_\alpha) \leq \liminf_{\alpha \to \infty} J_\alpha(\hat v) = \liminf_{\alpha \to \infty} -L_\gamma(u^*, \hat v) = J(\hat v).
\end{align*}
Hence,  $\liminf_{\alpha \to \infty} J(v_\alpha) 
= J(\hat v)$ 
and therefore for a subsequence (relabelled), we get $J(v_\alpha) \to J(\hat v)$. That is,
\[\frac{1}{2\gamma}\norm{u^*-v_\alpha}{V}^2 - L(u^*,v_\alpha) \to \frac{1}{2\gamma}\norm{u^*-\hat v}{V}^2 - L(u^*,\hat v),\] 
and since $L(u^*,v_\alpha) \to L(u^*, \hat v)$, it follows that $\norm{u^*-v_\alpha}{V} \to \norm{u^*-\hat v}{V}$. The weak convergence $u^*-v_\alpha \weaklyto u^*-\hat v$ together with the above norm convergence then yields $v_\alpha \to \hat v$ in $H^1(\Omega)$. 

\item \textbf{Conclusion.} Now we have $J_\alpha(v_\alpha) \leq J_\alpha(k) = J(k)$ for all $k \in K$, since $K \subset H^1(\Omega)$ and $v_\alpha$ is the minimiser. Taking the limit inferior in this inequality and using the first chain of inequalities above in step (ii),
$J(\hat v) \leq J(k)$ for all $k \in K$,
i.e., $\hat v$ is the minimiser of $J$ over $K$ (which is uniquely determined). The subsequence principle then tells us that the entire sequence $\{v_\alpha\}$ converges. Thanks to this and \eqref{eq:to_control} and \eqref{eq:new_1}, we obtain the result.   \end{enumerate}
}
\end{proof}
The proof reveals an error rate. \corr{Indeed, we get from \eqref{eq:leads_to_conv} and the definition of $R_2$ that 
\begin{align*}
 w_{o_2} \int_\Omega |(\psi-v_\alpha)^+|^2 + w_{b_2} \int_{\partial\Omega} (v_\alpha -h)^2 + \frac{1}{4\gamma}\norm{u^*-v_\alpha}{V}^2 \leq \gamma\norm{Au^* -f}{V^*}^2, 
\end{align*}
whence, recalling that $\alpha$ represents the weights $w_{o_1}, w_{o_2}$,}
\begin{align*}
\norm{(\psi-v_\alpha)^+}{L^2(\Omega)} = \mathcal{O}(1\slash\sqrt{\alpha})\qquad\text{and}\qquad \norm{v_\alpha-h}{L^2(\partial\Omega)} = \mathcal{O}(1\slash\sqrt{\alpha}).
\end{align*}

\subsubsection{Bounds on the approximation error}
To control the approximation error we need to ensure that the two network spaces $\cXs$ and $\cXt$ are sufficiently rich. This relates to universal approximation theorems for neural networks. 
The strength of the approximation results strongly depends on the norm used, and some care is required in ensuring that one can apply them to $\xi_{\mathrm{app}}$ since we need uniformity in $K \cap Y$ in order to bound the maxmin term there. In order to keep our theory as general as reasonably possible and bearing in mind that approximation results for different architectures are rapidly evolving, we make the next assumption and provide a scenario (see Proposition \ref{prop:satisfaction_of_density}) in which it is satisfied.
\begin{assumption}[Assumptions on the density of neural network spaces in $H^1(\Omega)$]\label{ass:approx}
~
\begin{enumerate}[label=(\roman*)]\itemsep=0cm
\item\label{ass:approx_solution} There exists a neural network architecture $\cFs$ with activation function $\sigmas$ satisfying: for every $w \in H^2(\Omega)$ and every $\epsilon > 0$, there exists  $\fds, \fws  \in \mathbb{N}$ and  $u \in \cFs(\fds, \fws, \sigmas)$ such that
\[\norm{u-w}{V} \leq \epsilon.\]
\item\label{ass:approx_test} 
There exists a neural network architecture $\cFt$ with activation function $\sigmat$ satisfying: for every $k \in Y$ and every $\epsilon > 0$, there exist $\fdt, \fwt \in \mathbb{N}$ independent of $k$ and  $v \in \cFt(\fdt, \fwt, \sigmat)$ such that
\[\norm{v-k}{V} \leq \epsilon.\]
\end{enumerate}
\end{assumption}
We also need the space $Y$ to be such that $K \cap Y$ is bounded in $V$:
\begin{assumption}\label{ass:on_boundedness_of_KcapY}
Assume that there exists $\bar R >0$ such that $\norm{v}{V} \leq \bar R$ for all $v \in K \cap Y.$   
\end{assumption}
In the next proposition, we show that these assumptions can be satisfied. To keep matters simple we focus only on the two activation functions $\mathrm{ReLU}$ and $\tanh$; others are possible too.
\begin{proposition}\label{prop:satisfaction_of_density}
Let $h \in H^{3\slash 2}(\partial\Omega)$, $\psi \in H^2(\Omega)$, and let $\Omega$ be convex or $C^{1,1}$. Let $\tilde R$ be a constant satisfying
 \[\tilde R \geq  \norm{u^*}{V},\]
set $Y:=B^{H^2(\Omega)}(\tilde R)$, the closed ball in $H^2(\Omega)$ with center $0$ and radius $\tilde R$, and choose $\cFt$ and $\cFs$ as $\mathcal{F}_{\mathrm{FFN}}$ with $\mathrm{ReLU}$ or $\tanh$ as the activation. Then     Assumption \ref{ass:approx} and Assumption \ref{ass:on_boundedness_of_KcapY} are satisfied 
\end{proposition}
\begin{proof}
Proposition 4.8 of \citet{GUHRING2021107} implies the following: let $\mu> 0$ be arbitrary; then there exist constants $\fdt, C, \theta$ and $\bar\epsilon$ such that for every $\epsilon < \bar\epsilon$ and every $k \in B^{H^2(\Omega)}(1)$, there exists a neural network $v \in \mathcal{F}_{\mathrm{FFN}}$ with at most $\fdt$ layers and at most 
\[
\begin{cases}
C\epsilon^{-n/(1-\mu)} &: \text{if activation  is $\tanh$,}\\
C\epsilon^{-n} &: \text{if activation is $\mathrm{ReLU}$,}
\end{cases}
\]
non-zero weights bounded in absolute value by $C\epsilon^{-\theta}$ satisfying
$\norm{v-k}{H^1(\Omega)} \leq \epsilon$.
Now if $g \in B^{H^2(\Omega)}(R)$, the above gives us a neural network $\tilde v$   with $\norm{\tilde v-g\slash R}{H^1(\Omega)} \leq \epsilon\slash R$ as long as $\epsilon \leq R\bar\epsilon$. Setting $v:= R\tilde v$, we get $\norm{v-g}{H^1(\Omega)} \leq \epsilon$ where $v \in \mathcal{F}_{\mathrm{FFN}}$ is also a neural network (with the same architecture as $\tilde v$ except the final layer has been scaled by $R$) with at most
\[
\begin{cases}
C(R^{-1}\epsilon)^{-n/(1-\mu)} &: \text{if activation  is $\tanh$,}\\
C(R^{-1}\epsilon)^{-n} &: \text{if activation is $\mathrm{ReLU}$,}
\end{cases}
\]
non-zero weights.

The key point that gives us the existence of the neural network architecture that is independent of the particular element $k$ that is being approximated is the following: if a network with an arbitrary number of neurons and layers has at most $M$ non-zero weights,  it can be viewed as a network with $M$ non-zero weights and at most $M+1$ neurons and layers. Since the above gives us at most $C(R^{-1}\epsilon)^{-n/(1-\mu)}$ or $C(R^{-1}\epsilon)^{-n}$ non-zero weights and this number is independent of $k$, it can be viewed as a neural network in $\mathcal{F}_{\mathrm{FFN}}
$ with a fixed (independent of $k$) size.

Hence Assumption \ref{ass:approx}  \ref{ass:approx_test} holds if we choose $Y=B^{H^2(\Omega)}(\tilde R)$ and $\cFt = \mathcal{F}_{\mathrm{FFN}}$ as the standard feedforward neural network architecture. This is a valid choice of $Y$ (i.e.,  it satisfies Assumption \ref{ass:on_spaces} by the assumptions on the domain and the data, see Proposition \ref{prop:h2_regularity}). Assumption \ref{ass:approx} \ref{ass:approx_solution} also holds by the above argument since $u^* \in H^2(\Omega)$. 

As $Y=B^{H^2(\Omega)}(\tilde R)$, Assumption \ref{ass:on_boundedness_of_KcapY} follows from the continuous embedding $H^2(\Omega) \cts V$ with $\bar R = \tilde R$.
\end{proof}
Since in our implementation we use the special ResNet structure of \citet{DeepRitz} for which approximation results still appear to be incomplete, we shall postpone the study of approximation theorems for our specific architecture to a later work. For approximation results involving the standard ResNet structure, we refer to e.g. \citet{Yuto, liu2024characterizing, li2022deep, NEURIPS2018_03bfc1d4}; note that  networks with one hidden layer possess universal approximation power in $L^1$ for ResNets \citep{NEURIPS2018_03bfc1d4}. 

Now, take $P \geq 1+ \norm{u^*}{V}$ and $Q \geq 1+\bar R$ and fix
$\cXs := \cFs \cap B^{H^1(\Omega)}(P)$ and $\cXt := \cFt \cap B^{H^1(\Omega)}(Q),$
where $\cFs$ and $\cFt$ are as in Assumption \ref{ass:approx}.
Due to this, if we set $R^* := \max(P, Q, C_TP, C_TQ)$, we get that the constants in \eqref{eq:uniform_bound_NN_spaces} satisfy $\max(M^{\mathrm{s}}, M^{\mathrm{t}}, B^{\mathrm{s}},B^{\mathrm{t}})  \leq R^*$.
Observe that all the constants $K_i$ depend on $R^*$. 
    Define $\kappa_1:= \max(1, K_4+ K_5+C_TK_6)$. Under  Assumption \ref{ass:approx} \ref{ass:approx_solution}, for every $\epsilon \in (0,1)$, there exists $\tilde{u} \in \cXs$ such that
    \begin{equation}\label{mh:square}
    \norm{\tilde u - u^*}{V} \leq \frac{\epsilon}{3\kappa_1}.
    \end{equation}
Note further that since  $\kappa_1 \geq 1$  and $\tilde u \in \cXs$, 
\[\norm{\tilde u}{V} \leq \norm{\tilde u - u^*}{V} + \norm{u^*}{V} \leq \frac{\epsilon}{3\kappa_1} +  \norm{u^*}{V} \leq 1+ \norm{u^*}{V} \leq P.\]
In a similar way, let $\kappa_2:= \max(1, K_1  + K_2 + C_TK_3)$. Under Assumption \ref{ass:approx} \ref{ass:approx_test}, for $\epsilon \in (0,1)$, we can show
\begin{equation}
\max_{k \in  K \cap Y}\corr{\inf_{v \in \cXt}}\norm{k-v}{V} \leq \frac{\epsilon}{3(K_1+ K_2+C_TK_3)}.\label{mh:square2}    
\end{equation}
Indeed, fixing an arbitrary $k \in K\cap Y$, by Assumption \ref{ass:approx} \ref{ass:approx_test} we get the existence of $\fdt, \fwt \in \mathbb{N}$ \textit{uniform in $k$} and $v \in \mathcal{F}(\fdt, \fwt, \sigmat)$ with $\norm{k-v}{V} \leq \epsilon\slash (3\kappa_2).$ This implies, because $\kappa_2 \geq 1$ and by the bound on $k$ from Assumption \ref{ass:on_boundedness_of_KcapY}, 
\[\norm{v}{V} \leq \norm{v-k}{V} + \norm{k}{V} \leq \frac{\epsilon}{3\kappa_2} + \bar R \leq 1 + \bar R \leq Q.\]
This shows that $v \in B^{H^1(\Omega)}(Q)$ and thus $v \in \cXt$. Then \eqref{mh:square2} follows from 
\[(K_1 + K_2 +  C_TK_3)\norm{k-v}{V} \leq \kappa_2\norm{k-v}{V} \leq \frac{\epsilon}{3}\]
\corr{and the fact that the $\max_{k \in K \cap Y}$ in \eqref{mh:square2} exists. Indeed, this is because the map
$ k \mapsto \inf_{v \in \cX_t} \|k-v\|_V $
is continuous on $V$ and moreover, by compact embedding every maximising sequence in $K \cap Y$ admits a subsequence that converges strongly in $V$ and whose limit is again in $K \cap Y$ by Assumption \ref{ass:on_spaces}. } 
\begin{theorem}\label{thm:approximation_error}
Let Assumption \ref{ass:approx} and Assumption \ref{ass:on_boundedness_of_KcapY} hold. For any $\epsilon > 0$,  there exist architectures $\cFs, \cFt$, activations $\sigmas, \sigmat$, numbers $\fds, \fws, \fdt, \fwt \in \mathbb{N}$ such that, with $\cXs = \cFs(\fds, \fws, \sigmas)$ and $\cXt = \cFt(\fdt, \fwt, \sigmat)$, we have
\begin{align*}
     \xi_{\mathrm{app}} \leq \epsilon.
\end{align*}
\end{theorem}
\begin{proof}
Fix $\epsilon > 0$. \corr{Beginning with the definition of $\xi_{\mathrm{app}}$ in  \eqref{eq:defn_approx_error}, we find}
\begin{align*}
\xi_{\mathrm{app}}  &\leq  \max_{k \in K \cap Y}\corr{\inf_{v \in \cXt}}   (K_1+K_2+C_TK_3)\norm{k-v}{V} +  \inf_{\bar u \in \cXs} (K_4+K_5+C_TK_6)\norm{\bar u - u^*}{V}
\leq \epsilon
\end{align*}
\corr{where we employed the trace theorem for the boundary term $\norm{k-v}{L^2(\partial\Omega)} \leq C_T\norm{k-v}{V}$ to obtain the first inequality and made use of \eqref{mh:square} and \eqref{mh:square2} for the final inequality.}
\end{proof}
At this point let us remark that if we had $\hatuA \in K$ the above error analysis could have been greatly simplified.

\subsubsection{The \eqref{eq:disc_ansatz} setting}
In the case of \eqref{eq:disc_ansatz} we wish to be able to choose $\cFs$ to be $\cFs_{0,\eta} + \tilde h$, or at least $\cFs_{0} + \tilde h$ .  In practice, the above results imply that we can `almost' do this. Indeed, recalling $\tilde u$ from \eqref{mh:square}, we see that
\[\norm{\tilde u - h}{H^{1\slash 2}(\partial\Omega)} \leq C_T\norm{\tilde u - u^*}{V} \leq \frac{C_T\epsilon}{3\kappa_1},\]
and setting $\hat u:= (\tilde u - \tilde h) + \tilde h$, which satisfies $\norm{\hat u - u^*}{V} = \norm{\tilde u - u^*}{V} \leq \epsilon\slash (3\kappa_1),$ 
we can conclude that $\cFs$ can be taken to be $\cFs_{\approx 0} + \tilde h$ where $\cFs_{\approx 0}$ means a set of neural networks that are almost zero on the boundary. A similar argument applies for $\cFt$ too. To be able to make this exact, we need to know if Sobolev functions that are zero on the boundary can be approximated by neural networks that are zero on the boundary and we need to do so in a uniform way (the width and depth of the network should be independent of the target function) for $\mathcal{F}^{\mathrm{t}}_0$.  Theorem 2 in \citet{MR4455184} tells us that $H^1_0(\Omega)$ functions can be realised by $\mathrm{ReLU}$ neural networks of depth $\lceil \log_2(d+1) \rceil +1$. In view of works such as \citet{GUHRING2021107}, the desired uniform approximation result appears reasonable, however, we are not aware at present of any literature with quantitative rates that supply such a result.

\subsubsection{Bounds on the statistical error}\label{sec:statError}
In this section, we bound the following quantity
\[
    \corr{\mathbb{E}_{\{x_i\}_{i=1}^N}} \left[ \sup_{\substack{u \in \cFs\\v \in \cFt}} \left| L_\gamma(u, v) + R_1(u) - R_2(v) - (\hatL_\gamma(u, v) + \hat R_1(u) - \hat R_2(v)) \right| \right],
\]
in the context of the problem \eqref{eq:disc_ansatz}, i.e., when the boundary condition is met and the penalty terms only contain contributions of the obstacle loss $L_o.$  We carry out the calculations for neural neworks of the special ResNet structure $\mathcal{F}_{\mathrm{DRR}}$. A statistical error analysis for $\mathcal{F}_{\mathrm{FFN}}$ involving different loss functions can be found in \citet[\S \corr{4.4}]{Jiao2023}, whose arguments we adapt.

We first need the concept of Rademacher complexity and covering numbers.
\begin{definition}[Rademacher complexity]
    Let $\mathcal{F}$ be a family of functions from $\Omega$  into $\mathbb{R}$ and let $P$ be a distribution over $\Omega$ and $\{X_i\}_{i=1}^N$ be independent identically distributed (iid) samples from $P$. The \emph{Rademacher complexity} of $\mathcal{F}$ associated with the distribution $P$ and sample size $N$ is defined as
    \[
        \mathcal{R}(\mathcal{F}) := \mathbb{E}_{\{X_i\}_{i=1}^N} \mathbb{E}_{\{\sigma_i\}_{i=1}^N} \left[ \sup_{u\in\mathcal{F}} \frac{1}{N}\sum_{i=1}^N \sigma_i u(X_i) \right],
    \]
    where $\{\sigma_i\}_{i=1}^N$ are iid random variables such that $\mathbb{P}[\sigma_i = 1] = \mathbb{P}[\sigma_i = -1] = \frac{1}{2}$.
\end{definition}

\begin{definition}[Covering number]
Given $\varepsilon > 0$, we say that $\mathcal{A} \subset \mathbb{R}^n$ is an \emph{$\varepsilon$-cover of $\mathcal{B}\subset \mathbb{R}^n$ with respect to a metric $\rho$} if for all $v^\prime\in \mathcal{B}$, there exists $v\in \mathcal{A}$ such that $\rho(v, v^\prime) \leq \varepsilon$.

    The \emph{$\varepsilon$-covering number of $\mathcal{B}$}, denoted as $\mathfrak{C}(\varepsilon, \mathcal{B}, \rho)$,is the minimum cardinality among all $\varepsilon$-covers of $\mathcal{B}$ with respect to the metric $\rho$.
\end{definition}
Rademacher complexity is useful because it bounds the statistical error from above as the next result shows. Below, we use the notation $U(\Omega)$ to denote the uniform distribution on $\Omega$. 
\begin{lemma}
Assume that the 
collocation points $\{x_i\}_{i=1}^N$ (see \S \ref{sec:disc_problems}) are iid drawn from $U(\Omega)$. Then 
\begin{equation*}\corr{\mathbb{E}_{\{x_i\}_{i=1}^N}} \left[ \sup_{\substack{u \in \cFs\\ v \in \cFt}}|L_\gamma(u, v) + R_1(u) - R_2(v) - (\hatL_\gamma(u, v) + \hat R_1(u) - \hat R_2(v))| \right] \leq 2|\Omega|\sum_{i=1}^{10} \mathcal{R}(\cF_i),
\end{equation*}
where
    \begin{equation*}   
    \begin{split}    
        \mathcal{F}_1 &= \{\Vert \nabla u\Vert^2 : u\in\cFs \}, \\
        \mathcal{F}_3 &= \{ ku^2: u\in\cFs\}, \\
                \mathcal{F}_5 &= \{ fu: u\in\cFs\}, \\
        \mathcal{F}_7 &= \left\{ \frac{1}{2\gamma}(u - v)^2: u\in\cFs, v\in\cFt\right\}, \\
        \mathcal{F}_9 &= \{ w_{o_1}|(\psi-u)^+|^2 : u \in \cFs\},
    \end{split}
    \hspace{2cm}
    \begin{split}        
            \mathcal{F}_2 &= \{ \nabla u \cdot \nabla v: u\in\cFs, v\in\cFt \}, \\
        \mathcal{F}_4 &= \{ kuv: u\in\cFs, v\in\cFt\},     \\
        \mathcal{F}_6 &= \{ fv: v\in\cFt\}, \\
        \mathcal{F}_8 &= \left\{ \frac{1}{2\gamma}|\nabla u - \nabla v|^2: u\in\cFs, v\in\cFt\right\},\\
        \mathcal{F}_{10} &= \{ w_{o_2}|(\psi-v)^+|^2 : u \in \cFt\}.
    \end{split}
\end{equation*}

\end{lemma}
\begin{proof}
 We can write $L_\gamma$ as  
\begin{align*}
    L_\gamma(u,v) &= |\Omega|\mathbb{E}_{X\sim U(\Omega)}\big[ \nabla u(X) \cdot (\nabla u(X) - \nabla v(X)) + ku(X)(u(X) - v(X)) - f(X)(u(X) - v(X)) \big] \\
    &\quad- |\Omega|\mathbb{E}_{X\sim U(\Omega)}\left[\frac{1}{2\gamma}\big((u(X) - v(X))^2 + |\nabla u(X) - \nabla v(X)|^2 \big)\right].
\end{align*}
With this in mind, we can decompose
\begin{equation*}
    |L_\gamma(u, v) + R_1(u) - R_2(v) - (\hatL_\gamma(u, v) + \hat R_1(u) - \hat R_2(v))| \leq \sum_{i=1}^8 |L_{\gamma, i}(u, v) - \hatL_{\gamma, i}(u, v)| + \sum_{i=1}^2 |R_i(u,v) - \hat R_i(u,v)| 
\end{equation*}
where
\begin{equation*}
  \begin{split}  
    L_{\gamma, 1}(u, v) &= |\Omega|\mathbb{E}_{X\sim U(\Omega)}\left[ \Vert \nabla u(X) \Vert^2 \right ], \\
    L_{\gamma, 3}(u, v) &= k|\Omega|\mathbb{E}_{X\sim U(\Omega)}\left[ u(X)^2 \right ], \\
    L_{\gamma, 5}(u, v) &= |\Omega|\mathbb{E}_{X\sim U(\Omega)}\left[ f(X)u(X) \right ], \\
    L_{\gamma, 7}(u, v) &= \frac{|\Omega|}{2\gamma}\mathbb{E}_{X\sim U(\Omega)}\left[ (u(X) - v(X))^2 \right ], \\
    R_1(u,v) &= |\Omega||\mathbb{E}_{X\sim U(\Omega)}\left[ |(\psi(X)-u(X)^+|^2\right],
  \end{split}
  \hspace{1.5cm}
  \begin{split}
    L_{\gamma, 2}(u, v) &= |\Omega|\mathbb{E}_{X\sim U(\Omega)}\left[ \nabla u(X) \cdot \nabla v(X) \right ], \\
    L_{\gamma, 4}(u, v) &= k|\Omega|\mathbb{E}_{X\sim U(\Omega)}\left[ u(X)v(X) \right ], \\
    L_{\gamma, 6}(u, v) &= |\Omega|\mathbb{E}_{X\sim U(\Omega)}\left[ f(X)v(X) \right ], \\
    L_{\gamma, 8}(u, v) &= \frac{|\Omega|}{2\gamma}\mathbb{E}_{X\sim U(\Omega)}\left[ |\nabla u(X) - \nabla v(X)|^2 \right ],\\ 
    R_2(u,v) &= |\Omega||\mathbb{E}_{X\sim U(\Omega)}\left[ |(\psi(X)-v(X)^+|^2\right],
  \end{split}
\end{equation*}
and
\begin{equation*}
 \begin{split}  
    \hatL_{\gamma, 1}(u, v) &= \frac{|\Omega|}{N} \sum_{i=1}^N \Vert \nabla u(x_i) \Vert^2, \\
    \hatL_{\gamma, 3}(u, v) &= \frac{k|\Omega|}{N} \sum_{i=1}^N u(x_i)^2, \\
    \hatL_{\gamma, 5}(u, v) &= \frac{|\Omega|}{N} \sum_{i=1}^N f(x_i)u(x_i), \\
    \hatL_{\gamma, 7}(u, v) &= \frac{|\Omega|}{2\gamma N} \sum_{i=1}^N (u(x_i) - v(x_i))^2,     \\
    \hat R_1(u,v) &= \frac{|\Omega|}{N} \sum_{i=1}^N |(\psi(x_i)-u(x_i))^+|^2,
  \end{split}
  \hspace{2cm}
  \begin{split}    
    \hatL_{\gamma, 2}(u, v) &= \frac{|\Omega|}{N} \sum_{i=1}^N \nabla u(x_i) \cdot \nabla v(x_i), \\
    \hatL_{\gamma, 4}(u, v) &= \frac{k|\Omega|}{N} \sum_{i=1}^N u(x_i)v(x_i), \\
    \hatL_{\gamma, 6}(u, v) &= \frac{|\Omega|}{N} \sum_{i=1}^N f(x_i)v(x_i), \\
    \hatL_{\gamma, 8}(u, v) &= \frac{|\Omega|}{2\gamma N} \sum_{i=1}^N |\nabla u(x_i) - \nabla v(x_i)|^2,\\
    \hat R_2(u,v) &= \frac{|\Omega|}{N} \sum_{i=1}^N |(\psi(x_i)-v(x_i))^+|^2.
  \end{split}
\end{equation*}
Now if we let $F_{u,v}$ denote a function of $u$ and $v$ and their first derivatives, we can write
\begin{align*}
    \left| |\Omega| \mathbb{E}_{X\sim U(\Omega)} \left[ F_{u,v}(X) \right] - \frac{|\Omega|}{N} \sum_{i=1}^N F_{u,v}(x_i) \right|
&= \frac{|\Omega|}{N} \left| \mathbb{E}_{\{X_i\}_{i=1}^N} \left[ \sum_{i=1}^N  F_{u,v}(X_i)  - F_{u,v}(x_i) \right] \right|,
\end{align*}
where $\{X_i\}_{i=1}^N$ are iid random variables drawn from $U(\Omega)$ and $\mathbb{E}_{\{X_i\}_{i=1}^N}$ means the (multiple) expectation with respect to $X_1, \cdots, X_N$. Taking the supremum over $\cFs$ and $\cFt$ and taking the expectation with respect to $\{x_i\}_{i=1}^N$ each of which are drawn from $U(\Omega)$, we have by Jensen's inequality 
\begin{align*}&\mathbb{E}_{\{x_i\}_{x=1}^N} \left[ \sup_{\substack{u \in \cFs\\v \in \cFt}} \frac{|\Omega|}{N} \left| \mathbb{E}_{\{X_i\}_{i=1}^N} \left[ \sum_{i=1}^N  F_{u,v}(X_i)  - F_{u,v}(x_i) \right] \right| \right] \\
    &\quad \leq \frac{|\Omega|}{N} \mathbb{E}_{\{x_i\}_{i=1}^N} \mathbb{E}_{\{X_i\}_{i=1}^N} \left[ \sup_{\substack{u \in \cFs\\v \in \cFt}} \left| \sum_{i=1}^N F_{u,v}(X_i) - F_{u,v}(x_i) \right| \right].
\end{align*}
Let $\{\sigma_i\}_{i=1}^N$ be iid random variables such that $\mathbb{P}[\sigma_i = 1] = \mathbb{P}[\sigma_i = -1] = \frac{1}{2}$ for all $i$. Note that for all $i$, $F_{u,v}(X_i) - F_{u,v}(x_i)$ and $F_{u,v}(x_i) - F_{u,v}(X_i)$ are equal in distribution and so the right-hand side above is equal to
\begin{align*}
    &\frac{|\Omega|}{2N} \mathbb{E}_{\{x_i\}_{i=1}^N}  \mathbb{E}_{\{X_i\}_{i=1}^N} \left[ \sup_{\substack{u \in \cFs\\v \in \cFt}} \left| \sum_{i=1}^N F_{u,v}(X_i) - F_{u,v}(x_i) \right| \right]\\
    &+   \frac{|\Omega|}{2N}\mathbb{E}_{\{x_i\}_{i=1}^N} \mathbb{E}_{\{X_i\}_{i=1}^N} \left[ \sup_{\substack{u \in \cFs\\v \in \cFt}} \left| \sum_{i=1}^N F_{u,v}(x_i) - F_{u,v}(X_i) \right| \right] \\
    &\quad = \frac{|\Omega|}{N} \mathbb{E}_{\{x_i\}_{i=1}^N} \mathbb{E}_{\{X_i\}_{i=1}^N} \mathbb{E}_{\{\sigma_i\}_{i=1}^N} \left[ \sup_{\substack{u \in \cFs\\v \in \cFt}} \left| \sum_{i=1}^N \sigma_i (F_{u,v}(X_i) - F_{u,v}(x_i)) \right| \right] \\
    &\quad \leq \frac{|\Omega|}{N} \mathbb{E}_{\{x_i\}_{i=1}^N} \mathbb{E}_{\{X_i\}_{i=1}^N} \mathbb{E}_{\{\sigma_i\}_{i=1}^N} \left[ \sup_{\substack{u \in \cFs\\v \in \cFt}} \left| \sum_{i=1}^N \sigma_i F_{u,v}(X_i) \right|\right]\\
    &\quad\quad   + \frac{|\Omega|}{N} \mathbb{E}_{\{x_i\}_{i=1}^N} \mathbb{E}_{\{X_i\}_{i=1}^N} \mathbb{E}_{\{\sigma_i\}_{i=1}^N} \left[ \sup_{\substack{u \in \cFs\\v \in \cFt}} \left| \sum_{i=1}^N \sigma_i F_{u,v}(x_i) \right|\right] \\
    &\quad \leq \frac{2|\Omega|}{N}\mathbb{E}_{\{x_i\}_{i=1}^N} \mathbb{E}_{\{\sigma_i\}_{i=1}^N} \left[ \sup_{\substack{u \in \cFs\\v \in \cFt}} \sum_{i=1}^N \sigma_i F_{u,v}(x_i) \right] \\
    &\quad = 2|\Omega| \mathcal{R}(\mathcal{F}),
\end{align*}
where $\mathcal{F} = \{F_{u,v}: u\in\cFs, v\in\cFt\}$ is a function class parameterised by $\Theta^{\mathfrak{s}} \times \Theta^{\mathfrak{t}}$ where $\Theta^{\mathfrak{s}}$, $\Theta^{\mathfrak{t}}$ are the network weights of $u$, $v$ respectively. The above can be employed to show that for each $i$, \corr{the expectation of} $\sup_{\substack{u \in \cFs\\v \in \cFt}} |L_{\gamma, i}(u, v) - \hatL_{\gamma, i}(u, v)| + |R_{i}(u, v) - \hat R_{i}(u, v)|$ is bounded by the Rademacher complexity of the function class $\mathcal{F}_i$. The claim follows.    
\end{proof}
Now, \citet[Lemma \corr{4.11}]{Jiao2023} shows that to bound the Rademacher complexity of a function class, we can bound the associated covering number which for Lipschitz functions we can bound by the covering number of the parameter space \citep[Lemma \corr{4.9}]{Jiao2023}.
Hence, our strategy is to first show that the set of neural networks defined by \eqref{eq:NN} and the first derivative of each element of the set are bounded and Lipschitz and from that conclude that each of the $\mathcal{F}_i$ (being a certain function of \eqref{eq:NN}) is also bounded and Lipschitz. This can then be used to bound their respective Rademacher complexities using the aforementioned lemmata.

From now on we assume that the activation function $\sigma$ of our neural network architecture is sufficiently regular, see Assumption \ref{ass:on_act_for_stat_error}---$\tanh$ is a valid activation function. The next lemma shows that in this case, neural networks given by \eqref{eq:NN} are Lipschitz in their parameters in the $C^1(\Omega)$ norm.
For the proof, it is more convenient to rewrite the definition of our neural network architecture in a recursive manner and also to index by network layers rather than by network blocks.
For $l = 1,\ldots,2\fd$, let
\begin{align*}
    f^{(0)}(x) &= A^{(0)} x + b^{(0)}, \\
    f^{(l)}(x) &= \sigma(A^{(l)}f^{(l-1)}(x) + b^{(l)}) + f^{(l-2)}(x) \mathbf{1}_{\{l \text{ even}\}}, \\
    u(x) &= A^{(2\fd + 1)} f^{(2\fd)}(x) + b^{(2\fd + 1)},
\end{align*}
where $A^{(0)} = A_0$, $b^{(0)} = b_0$, $A^{(2\fd + 1)} = A_{\fd + 1}$, $b^{(2\fd + 1)} = b_{\fd + 1}$, $A^{(l)} = A_{ij}$ and $b^{(l)} = b_{ij}$ with $i = \lceil l/2 \rceil, j = 1$ for $l$ odd and $j = 2$ otherwise, where $A_0, A_{ij}, A_{\fd + 1}, b_0, b_{ij}$ and $b_{\fd + 1}$ are as in the definition of \eqref{eq:NN}.
We also use the notation $A^{(l)} = ( a_{qj}^{(l)} )_{q,j=1}^{\fw}$ and $b^{(l)} = (b_1^{(l)},\ldots,b_{\fw}^{(l)})$. Each $f^{(i)} = (f_1^{(i)},\ldots,f_{\fw}^{(i)})$ and $u$ are functions of their network weights $\theta$.
For calculations involving two different sets of network weights $\theta$ and $\tilde{\theta}$, we adorn a variable with a tilde (e.g. $\tilde{f}^{(i)}$, $\tilde{b}^{(i)}$) to indicate that the function or variable is with respect to $\tilde{\theta}$.
Moreover, let $n_i$ denote the number of network weights in the $i$th layer and $N_i$ to be the total number of weights up to and including the $i$th layer.
Without loss of generality, we assume $\fw \geq n_0$. Below, we let $C_\Omega \geq 1$ be a constant such that  $|x|\leq C_\Omega$ for all $x\in\Omega$.

\begin{assumption}\label{ass:on_act_for_stat_error}
Suppose that 
\begin{enumerate}[label=(\roman*)]\itemsep=0cm
\item $\sigma \in C^1(\mathbb{R}) \cap W^{1,\infty}(\mathbb{R})$ and $\sigma$ and \corr{its derivative} $\sigma^\prime$ are Lipschitz with Lipschitz constants $L_\sigma$ and $L_{\sigma^\prime}$ respectively,
\item $\Theta$ is bounded by $B_\theta$.
\end{enumerate}

\end{assumption}

\begin{lemma}\label{lem:w_theta_bounded_lipschitz}
Under Assumption \ref{ass:on_act_for_stat_error}, \corr{and recalling that $\mathfrak{w}$  and $\mathfrak{m}$ denote  the network width and the total number of learnable parameters respectively (as introduced in Section \ref{sec:nnapproach})}, $u$ defined by \eqref{eq:NN}
satisfies                $\norm{u}{C^0(\Omega)} \leq B_u$ and $\norm{\partial_{x_p}  u}{C^0(\Omega)} \leq B_{u^\prime}$ 
where 
\begin{align*}
    B_u := (\fw + 1) \bar{C}^2, \qquad 
                    B_{u^\prime} := 2^\fd \fw^{2\fd + 1} \bar{C}^{4\fd + 2}, \qquad \bar{C} &:= \max(1, \norm{\sigma}{L^\infty(\Omega)}, \norm{\sigma^\prime}{L^\infty(\Omega)}, L_\sigma, L_{\sigma^\prime}, B_\theta).
\end{align*}
Furthermore, the map $\theta \mapsto u$ is Lipschitz from $\ell^2(\mathbb{R})$ into $C^1(\Omega)$ with
\begin{align*}
                \norm{u - \tilde{u}}{C^0(\Omega)} &\leq L_u \Vert \theta - \tilde{\theta} \Vert_2 , \qquad             \norm{\partial_{x_p} u - \partial_{x_p} \tilde{u}}{C^0(\Omega)} \leq L_{u^\prime} \Vert \theta - \tilde{\theta} \Vert_2,
\end{align*}
where  
\begin{align*}
                L_u &:= \sqrt{\mathfrak{m}} 2^{2\fd-1} \fw^{2\fd + 1} \bar{C}^{4\fd + 1} C_\Omega, \qquad         L_{u^\prime} := \sqrt{\mathfrak{m}} \left(\sum_{k=0}^{2\fd} 2^k \right) 2^{4\fd - 2} \fw^{4\fd + 1} \bar{C}^{8\fd + 1} C_\Omega.                
\end{align*}   
\end{lemma}

\begin{proof}
First note that $u$ is bounded because $\sigma$ is bounded:
\[
    |u| \leq \sum_{j=1}^\fw |a_j^{(2\fd + 1)}||f_j^{2\fd}| + |b^{(2\fd + 1)}| \leq (\fw + 1) \bar{C}^2.
\]
We now prove the Lipschitzness of $u$. Let us begin with
\begin{align*}
    |f_q^{(i)} - \tilde{f}_q^{(i)}| &\leq \left| \sigma\left( \sum_{j=1}^{\fw} a_{qj}^{(i)} f_j^{(i - 1)} + b_q^{(i)} \right) - \sigma\left( \sum_{j=1}^{\fw} \tilde{a}_{qj}^{(i)} \tilde{f}_j^{(i - 1)} + \tilde{b}_q^{(i)} \right) \right| + |f_q^{(i - 2)} - \tilde{f}_q^{(i - 2)}| \mathbf{1}_{\{i \text{ even}\}}.
\end{align*}
We can bound the first term on the right-hand side using the same method as in the proof of \citet[Lemma 4.\corr{12}]{Jiao2023} to derive the following recurrence relation:
\begin{equation*}
    |f_q^{(i)} - \tilde{f}_q^{(i)}| \leq \bar{C}^2 \sum_{j=1}^{\fw} |f_j^{(i - 1)} - \tilde{f}_j^{(i - 1)}| + \bar{C}F^{({i-1})} \sum_{j=1}^{\fw} |a_{qj}^{(i)} - \tilde{a}_{qj}^{(i)}| + \bar{C}|b_q^{(i)} - \tilde{b}_q^{(i)}| +  |f_q^{(i - 2)} - \tilde{f}_q^{(i - 2)}| \mathbf{1}_{\{i \text{ even}\}},
\end{equation*}
where $F^{(i)}$ is a constant satisfying $F^{(i)} \leq \norm{\sigma}{L^\infty(\Omega)}$ for $i\geq 1$ and
\[
    F^{(0)} = \sup_q |f^{(0)}_q| = \sup_q \left| \sum_{j=1}^n a_{qj}^{(0)}x_j + b_q^{(0)} \right| \leq \sup_q \sum_{j=1}^n |a_{qj}^{(0)}||x_j| + |b_q^{(0)}| \leq n_0 \bar{C}C_\Omega.
\]
For $i = 0$,
\begin{align*}
    | f_q^{(0)} - \tilde{f}_q^{(0)}| &\leq \left\vert \sum_{j=1}^n a_{qj}^{(0)} x_j + b_q^{(0)} - \sum_{j=1}^n \tilde{a}_{qj}^{(0)}x_j - \tilde{b}_q^{(0)} \right\vert 
    \leq \sum_{j=1}^n |x_j| |a_{qj}^{(0)} - \tilde{a}_{qj}^{(0)} | + |b_q^{(0)} - \tilde{b}_q^{(0)}| 
    \leq C_\Omega \sum_{j=1}^{n_0} |\theta_j - \tilde{\theta}_j |.\end{align*}
Assume for $i \geq 1$ that
\begin{equation}\label{eq:f_lipschitz_induction_hyp}
    |f_q^{(i)} - \tilde{f}_q^{(i)}| \leq 2^{i-1} {\fw}^i \bar{C}^{2i} C_\Omega \sum_{j=1}^{N_i} |\theta_j - \tilde{\theta}_j|,
\end{equation}
then
\begin{align*}
    |f_q^{(i+1)} - \tilde{f}_q^{(i+1)}| &\leq \bar{C}^2 \sum_{j=1}^{\fw} |f_q^{(i)} - \tilde{f}_q^{(i)}|\\
    &\quad + \bar{C}^2 \sum_{j=1}^{\fw} |a_{qj}^{(i+1)} - \tilde{a}_{qj}^{(i+1)}| + \bar{C}|b_q^{(i+1)} - \tilde{b}_q^{(i+1)}| + |f_q^{(i-1)} - \tilde{f}_q^{(i-1)}|  \mathbf{1}_{\{i + 1 \text{ even}\}} \\
    &\leq 2^{i-1} {\fw}^i \bar{C}^{2i} \bar{C}^2 C_\Omega \sum_{j=1}^{\fw} \sum_{k=1}^{N_i} |\theta_k - \tilde{\theta}_k| + \bar{C}^2 \sum_{j=1}^{n_{i+1}} |\theta_j - \tilde{\theta}_j| + 2^{i-2}{\fw}^{i-1}\bar{C}^{2(i-1)}C_\Omega \sum_{j=1}^{N_{i-1}} |\theta_j - \tilde{\theta}_j| \\
    &\leq 2^{i} {\fw}^{i+1} \bar{C}^{2(i+1)} C_\Omega \sum_{j=1}^{N_{i+1}} |\theta_j - \tilde{\theta}_j|.
\end{align*}
Hence, \eqref{eq:f_lipschitz_induction_hyp} is true for $i=1,\ldots,2\fd$ and so
\begin{align*}
    |u(x) - \tilde{u}(x)| &\leq \bar{C} \sum_{j=1}^{\fw} |f_j^{(2\fd)} - \tilde{f}_j^{(2\fd)}| + \bar{C} \sum_{j=1}^{\fw} |a_j^{(2\fd+1)} - \tilde{a}_j^{(2\fd+1)}| + |b^{(2\fd+1)} - \tilde{b}^{(2\fd+1)}| \\
    &\leq 2^{2\fd-1}{\fw}^{2\fd} \bar{C}^{4\fd + 1}C_\Omega \sum_{j=1}^{\fw} \sum_{k=1}^{N_{2\fd}} |\theta_k - \tilde{\theta}_k| + \bar{C}\sum_{j=1}^{N_{2\fd+1}} |\theta_j - \tilde{\theta}_j| \\
    &\leq 2^{2\fd-1} {\fw}^{2\fd+1} \bar{C}^{4\fd+1} C_\Omega \sum_{j=1}^\mathfrak{m} |\theta_j - \tilde{\theta}_j| \\
    &\leq \sqrt{\mathfrak{m}} 2^{2\fd-1} {\fw}^{2\fd+1} \bar{C}^{4\fd+1} C_\Omega \Vert \theta - \tilde{\theta} \Vert_2,
\end{align*}
where the last line follows from H\"older's inequality. For the spatial derivatives of $u$, we have
\[
    \partial_{x_p} f_q^{(i)} = \sum_{j=1}^{\fw} a_{qj}^{(i)} \partial_{x_p} f_j^{(i-1)} \sigma^\prime \left( \sum_{j=1}^{\fw} a_{qj}^{(i)} f_j^{(i)} + b_q^{(i)} \right) + \partial_{x_p} f_q^{(i-2)} \mathbf{1}_{\{i \text{ even}\}},
\]
and so for $i$ even,
\begin{align*}
    |\partial_{x_p} f_q^{(i)}| &\leq \bar{C}^2 \sum_{j=1}^{\fw} |\partial_{x_p} f_j^{(i-1)}| + |\partial_{x_p} f_q^{(i-2)}|  \leq \bar{C}^4 \sum_{j=1}^{\fw} \sum_{k=1}^{\fw} |\partial_{x_p} f_k^{(i-2)}| + |\partial_{x_p} f_q^{(i-2)}|  \leq 2{\fw}\bar{C}^4 \sum_{j=1}^{\fw} |\partial_{x_p} f_j^{(i-2)}|.
\end{align*}
Iterating the above gives, for $i$ even,
\begin{equation}\label{eq:df_bound_even}
    |\partial_{x_p} f_q^{(i)}| \leq 2^{i/2} {\fw}^{i-1} \bar{C}^{2i} \sum_{j=1}^{\fw} |\partial_{x_p} f_j^{(0)}| \leq 2^{i/2} {\fw}^{i-1} \bar{C}^{2i} \sum_{j=1}^{\fw} |a_{jp}^{(0)}| \leq 2^{i/2} {\fw}^i \bar{C}^{2i + 1}.
\end{equation}
In a similar way, for $i$ odd,
\begin{equation}\label{eq:df_bound_odd}
    |\partial_{x_p} f_q^{(i)}| \leq \bar{C}^2 \sum_{j=1}^{\fw} |\partial_{x_p} f_j^{(i-1)}| \leq \bar{C}^2 \sum_{j=1}^{\fw} 2^{(i-1)/2} {\fw}^{i-1} \bar{C}^{2(i-1)+1} \leq 2^{i/2} {\fw}^i \bar{C}^{2i+1}.
\end{equation}
Therefore,
\[
    |\partial_{x_p} u(x)| \leq \left| \partial_{x_p} \left(\sum_{j=1}^\fw a_j^{(2\fd+1)} f_j^{(2\fd)} + b^{(2\fd+1)}\right) \right| \leq \sum_{j=1}^{\fw} |a_j^{(2\fd+1)}| |\partial_{x_p} f_j^{(2\fd)}| \leq 2^{\fd} {\fw}^{2\fd+1} \bar{C}^{4\fd+2}.
\]
We now show that the derivatives of $u$ are Lipschitz.
We have
\begin{align*}
    |\partial_{x_p} f_q^{(i)} - \partial_{x_p} \tilde{f}_q^{(i)}|
    &\leq \left|\partial_{x_p} \sigma\left( \sum_{j=1}^{\fw} a_{qj}^{(i)} f_j^{(i)} + b_q^{(i)} \right) - \partial_{x_p} \sigma\left( \sum_{j=1}^{\fw} a_{qj}^{(i)} f_j^{(i)} + b_q^{(i)} \right) \right|\\
    &\quad + |\partial_{x_p} f_q^{(i-2)} - \partial_{x_p} \tilde{f}_q^{(i-2)}| \mathbf{1}_{\{i \text{ even}\}}.
\end{align*}
We can bound the first term on the right-hand side of the above in the same manner as in the proof of \citet[Lemma 4.\corr{14}]{Jiao2023} and use \eqref{eq:f_lipschitz_induction_hyp}, \eqref{eq:df_bound_even} and \eqref{eq:df_bound_odd} to show that
\begin{align}
    &|\partial_{x_p} f_q^{(i)} - \partial_{x_p} \tilde{f}_q^{(i)}|\\
    &\leq \bar{C}\sum_{j=1}^{\fw} |a_{qj}^{(i)}| |\partial_{x_p} f_j^{(i-1)} - \partial_{x_p} \tilde{f}_j^{(i-1)}| + |\partial_{x_p} f_j^{(i-1)}||a_{qj}^{(i)} - \tilde{a}_{qj}^{(i)}| \notag \\
    &\qquad+ \bar{C} \left( \sum_{j=1}^{\fw} |\tilde{a}_{qj}^{(i)}||\partial_{x_p} \tilde{f}_j^{(i-1)}| \right) \left( \sum_{j=1}^{\fw} |f_j^{(i-1)}||a_{qj}^{(i)} - \tilde{a}_{qj}^{(i)}| + |\tilde{a}_{qj}^{(i)}||f_j^{(i-1)} - \tilde{f}_j^{(i-1)}| + |b_q^{(i)} - \tilde{b}_q^{(i)}| \right) \notag \\
    &\qquad+ |\partial_{x_p} f_q^{(i-2)} - \partial_{x_p} \tilde{f}_q^{(i-2)}| \label{eq:df_diff}\\
    &\leq \bar{C}^2 \sum_{j=1}^{\fw} |\partial_{x_p} f_j^{(i-1)} - \partial_{x_p} \tilde{f}_j^{(i-1)}| + 2^{2(i-1)}{\fw}^{2i}\bar{C}^{4i}C_\Omega \sum_{j=1}^{N_i} |\theta_j - \tilde{\theta}_j| + |\partial_{x_p} f_q^{(i-2)} - \partial_{x_p} \tilde{f}_q^{(i-2)}|. \notag
\end{align}
For $i=0$, we have 
$|\partial_{x_p} f_q^{(0)} - \partial_{x_p} \tilde{f}_q^{(0)}| \leq |a_{qp}^{(0)} - \tilde{a}_{qp}^{(0)}|\leq \sum_{j=1}^{n_0} |\theta_j - \tilde{\theta}_j|.$ 
Assume for $i\geq 1$ that
\[
    |\partial_{x_p} f_q^{(i)} - \partial_{x_p} \tilde{f}_q^{(i)}| \leq 2^{2(i-1)} \left(\sum_{k=0}^i 2^k \right) {\fw}^{2i} \bar{C}^{4i} C_\Omega \sum_{j=1}^{N_i} |\theta_j - \tilde{\theta}_j|,
\]
then
\begin{align*}
    &|\partial_{x_p} f_q^{(i+1)} - \partial_{x_p} \tilde{f}_q^{(i+1)}|\\
    &\leq \bar{C}^2 \sum_{j=1}^{\fw} |\partial_{x_p} f_j^{(i)} - \partial_{x_p} \tilde{f}_j^{(i)}| + 2^{2i}{\fw}^{2(i+1)} \bar{C}^{4(i+1)} C_\Omega \sum_{j=1}^{N_{i+1}} |\theta_j - \tilde{\theta}_j| + |\partial_{x_p} f_q^{(i-1)} - \partial_{x_p} \tilde{f}_q^{(i-1)}| \\
    &\leq 2^{2(i-1)}\left(\sum_{k=0}^i 2^k \right) {\fw}^{2i} \bar{C}^{4i + 2} C_\Omega \sum_{j=1}^m \sum_{k=1}^{N_i} |\theta_k - \tilde{\theta}_k| + 2^{2i}{\fw}^{2(i+1)} \bar{C}^{4(i+1)} C_\Omega \sum_{j=1}^{N_{i+1}} |\theta_j - \tilde{\theta}_j| \\ 
    &\quad + 2^{2(i-2)} \left(\sum_{k=0}^{i-1} 2^k \right) {\fw}^{2(i-1)} \bar{C}^{4(i-1)} C_\Omega \sum_{j=1}^{N_{i-1}} |\theta_j - \tilde{\theta}_j| \\
    &\leq 2^{2i} \left(\sum_{k=0}^{i+1} 2^k \right) {\fw}^{2(i+1)} \bar{C}^{4(i+1)}C_\Omega \sum_{j=1}^{N_{i+1}} |\theta_j - \tilde{\theta}_j|.
\end{align*}
Hence, by induction
\begin{align*}
    |\partial_{x_p} u(x) - \partial_{x_p} \tilde{u}(x)| 
    &\leq \sum_{j=1}^{\fw} |\partial_{x_p} f_j^{(2\fd)}| |a_j^{(2\fd+1)} - \tilde{a}_j^{(2\fd+1)}| + |\tilde{a}_j^{(2\fd+1)}| |\partial_{x_p} f_j^{(2\fd)} - \partial_{x_p} \tilde{f}_j^{(2\fd)}| \\
    &\leq 2^{4\fd - 2} \left(\sum_{k=0}^{2\fd} 2^k \right) {\fw}^{4\fd+1} \bar{C}^{8\fd+1} C_\Omega \sum_{j=1}^\mathfrak{m} |\theta_j - \tilde{\theta}_j| \\
    &\leq 2^{4\fd - 2} \sqrt{\mathfrak{m}} \left(\sum_{k=0}^{2\fd} 2^k \right) {\fw}^{4\fd+1} \bar{C}^{8\fd+1} C_\Omega \Vert \theta_j - \tilde{\theta} \Vert_2.
\end{align*}
\end{proof}
In the setting of \eqref{eq:disc_ansatz}, we use the same neural network architecture for both the solution and the test function, in particular, we take $u = \bar{u}\eta + \bar{h}$ where $\bar{u}\in \mathcal{F}(\fd, \fw, \sigma)$, $\eta\in C^1(\bar{\Omega})$ satisfying $\eta|_{\partial\Omega} = 0$ and $\bar{h}\in C^1(\bar{\Omega})$ satisfies the boundary condition.
In practice, we take $\bar{h}$ to be also a neural network of the form \eqref{eq:NN}.
Assume also that $\eta$ and $\Vert\nabla\eta\Vert^2$ are bounded by $B_\eta$ and $B_{\eta'}$ respectively and the source term $f$ and obstacle $\psi$ are bounded by $B_f$ and $B_\psi$ respectively.
Then, $u$  is bounded with constant $B_uB_\eta + B_u=: D_1$ is  and Lipschitz with constant $L_uB_\eta + L_u =: D_2$,  and $\partial_{x_p} u$ is bounded by $B_{u'}B_\eta + B_uB_{\eta'} + B_{u'} =: D_3$ and is Lipschitz with constant $L_{u'}B_\eta + L_uB_{\eta'} + L_{u'}=: D_4$.
It follows that for any $g\in\mathcal{F}_i$, $i=1,\ldots,10$, we have $|g| \leq B_i$ where
\begin{equation*}
\begin{split}
    B_1 &= B_2 = nD_3^2, \\
        B_7 &= 2\gamma^{-1}D_1^2, \\    
        B_{10} &= 2w_{o_2}(B_\psi^2 + D_1^2),   
\end{split}
\hspace{2cm}
\begin{split}
    B_3 &= B_4 = kD_1^2, \\
            B_8 &= 2\gamma^{-1}B_1,
\end{split}
\hspace{2cm}
\begin{split}
    B_5 &= B_6 = B_fD_1, \\
    B_9 &=   2w_{o_1}(B_\psi^2 + D_1^2),        
\end{split}
\end{equation*}
and $g$ is Lipschitz with constant $L_i$ where
\begin{equation*}
\begin{split}    
    L_1 &= L_2 = 2n D_3D_4, \\    
L_7 &= 4\gamma^{-1}D_1D_2, \\        
    L_{10} &= 2w_{o_2}(B_\psi + D_1)D_2.       
\end{split}
\hspace{2cm}
\begin{split}
L_3 &= L_4 = 2kD_1D_2, \\
L_8 &= 4\gamma^{-1}nD_3D_4,    
\end{split}
\hspace{2cm}
\begin{split}
L_5 &= L_6 = B_fD_2, \\
L_9 &= 2w_{o_1}(B_\psi + D_1)D_2,       
\end{split}
\end{equation*} 
With this we can prove the following theorem which gives a bound of the statistical error.
In the proof, we denote generic constants which may differ line-by-line and on its dependencies by $K(\cdot)$. Recall also that $n$ is the dimension of the domain $\Omega \subset \mathbb{R}^n$.

\begin{theorem}[Statistical error estimate for \eqref{eq:disc_ansatz}]\label{thm:stat_error}
Consider the $\cF_{\mathrm{DRR}}$ case and assume that $f \in C^0(\Omega) \cap L^\infty(\Omega)$, the activation function $\sigma \in C^1(\mathbb{R}) \cap L^\infty(\mathbb{R})$ is Lipschitz and $\Theta$ is bounded. Then we have
\begin{align*}
&\mathbb{E}_{\{x_i\}_{i=1}^N} \left[ \sup_{\substack{u \in \cFs\\v \in \cFt}} |L_\gamma(u, v) + R_1(u) - R_2(v) - (\hatL_\gamma(u, v) + \hat R_1(u) - \hat R_2(v)) | \right]\\
&\quad\quad\quad\leq \frac{K(\eta, f) n^{\frac{3}{2}} \mathfrak{m} \sqrt{\sum_{k=0}^{2\fd} 2^k} 2^{\frac{9\fd - 2}{2}} \fw^{7\fd + 3} \bar{C}^{14\fd + 9} \sqrt{C_\Omega}} {N^{\frac{1}{4}}},
\end{align*}
\corr{where $K(\eta,f)$ is a generic constant depending on the bounds of $\eta$ and $f$.}
A similar bound holds for the $\cF_{\mathrm{FFN}}$ case.
\end{theorem}
Recall again that we do not have boundary penalty terms in this formulation.
\begin{proof}
    We address the $\mathrm{DRR}$ setting; the standard $\mathrm{FFN}$ case follows by the obvious modifications. Let $\mathcal{F}$ be an arbitrary function class such that for $f\in\mathcal{F}$ we have $\norm{f}{L^\infty(\Omega)} \leq B$ and $f$ is $L$-Lipschitz with respect to the parameter $\theta$. Then by \citet[Lemmas \corr{4.11, 4.9, 4.8}]{Jiao2023} in that order, we have \begin{align*}
        \mathcal{R}(\mathcal{F}) &\leq \inf_{0<\delta<B/2} \left( 4\delta + \frac{12}{\sqrt{N}} \int_\delta^{B/2} \sqrt{\log \mathfrak{C}(\varepsilon, \mathcal{F}, \Vert\cdot\Vert_\infty)} \;\mathrm{d}\varepsilon \right) \\
        &\leq \inf_{0<\delta<B/2} \left( 4\delta + \frac{12}{\sqrt{N}} \int_\delta^{B/2} \sqrt{\mathfrak{m}\log \left(\frac{2L\bar{C}\sqrt{\mathfrak{m}}}{\varepsilon}\right)} \;\mathrm{d}\varepsilon \right) \\
        &\leq \frac{4}{\sqrt{N}} + \frac{6\sqrt{\mathfrak{m}}B}{\sqrt{N}} \sqrt{\log(2L\bar{C}\sqrt{N \mathfrak{m}})},
    \end{align*}
    where in the last line we set $\delta = 1 / \sqrt{N}$.
    It is clear that $B_i \leq KB_1$ and $L_i \leq KL_1$, $i = 1,\ldots,10$ for some constant \corr{$K := K(\eta, f)$} so it suffices to bound $\mathcal{R}(\mathcal{F}_1)$ only.     Using the inequality $(a + b)^2 \leq 2(a^2 + b^2)$, we have
    \begin{equation*}
        B_1 \leq 4n(B_u^2 B_{\eta'}^2 + B_{u^\prime}^2 B_\eta^2 + B_{u^\prime}^2) \leq \corr{K} n (B_u^2 + B_{u^\prime}^2) \leq \corr{K} n B_{u^\prime}^2 = \corr{K} n 2^{2\fd} \fw^{4\fd+2}\bar{C}^{8\fd+4}
    \end{equation*}
    and
    \begin{equation*}
        L_1 \leq \corr{K} b (B_u + B_{u^\prime})(L_u + L_{u^\prime}) \leq \corr{K} n B_{u^\prime} L_{u^\prime} = \corr{K} n \sqrt{\mathfrak{m}} \left(\sum_{k=0}^{2\fd} 2^k\right) 2^{5\fd-2} \fw^{6\fd+2} \bar{C}^{12\fd + 3} C_\Omega.
    \end{equation*}
    Therefore,
    \begin{align*}
        \mathcal{R}(\mathcal{F}_1) &\leq \frac{4}{\sqrt{N}} + \frac{\corr{K} \sqrt{\mathfrak{m}} n 2^{2\fd} \fw^{4\fd+2}\bar{C}^{8\fd+4}}{\sqrt{N}} \sqrt{\log\left( 2 \left(\sum_{k=0}^{2\fd} 2^k\right) 2^{5\fd-2} \fw^{6\fd+2} \bar{C}^{12\fd + 4} C_\Omega \sqrt{N} \mathfrak{m}\right)} \\
        &\leq \frac{4}{\sqrt{N}} + \frac{\corr{K}}{N^{\frac{1}{4}}} n^{\frac{3}{2}} \mathfrak{m} \sqrt{\sum_{k=0}^{2\fd} 2^k} 2^{\frac{9\fd - 2}{2}} \fw^{7\fd + 3} \bar{C}^{14\fd + 9} \sqrt{C_\Omega} \\
        &\leq \frac{\corr{K} n^{\frac{3}{2}} \mathfrak{m} \sqrt{\sum_{k=0}^{2\fd} 2^k} 2^{\frac{9\fd - 2}{2}} \fw^{7\fd + 3} \bar{C}^{14\fd + 9} \sqrt{C_\Omega}} {N^{\frac{1}{4}}}.
    \end{align*}
    Recalling that the left-hand side of the statement of the theorem is bounded by $\sum_{i=1}^{10} \mathcal{R}(\mathcal{F}_i)$ completes the proof.
\end{proof}
The theorem tells us that the statistical error can be made arbitrarily small if the number of grid points $N$ chosen is large enough, and it also indicates that the error may grow if the width and depth of the network increase. 
\section{Numerical details and examples}\label{sec:numerical_examples}
We numerically solve the discrete minmax problem \eqref{eq:disc_ansatz} for various examples via the approach detailed in \S \ref{sec:imposition_bcs}. Given a concrete obstacle problem, the first step is to construct the function $\tilde h$, which is supposed to (approximately) satisfy the boundary condition. We find such a function simply by optimising\footnote{We included the obstacle constraint loss in the minimisation problem in order to give the minmax problem \eqref{eq:disc_ansatz} a good start in the sense that $\tilde h \in K$ (again, up to error).}
\[\min_{w \in \mathcal{F}_{\mathrm{DRR}}(\fd,\fw, \tanh) } \hat L_b(w) + \hat L_o(w),\]
so that the output $\tilde h$ is itself a neural network. We stop training when $\tilde h$ satisfies the boundary condition up to some error threshold. With this $\tilde h$ fixed, we then solve \eqref{eq:disc_ansatz} by applying the gradient descent ascent (GDA) scheme in Algorithm \ref{alg:1}. 

\begin{remark}
\corr{We also implemented a version of the algorithm in line with \eqref{eq:discrete_minmax_general} that  does not fix the boundary condition but instead penalises it. In our experience, the computed solutions in this setting were worse than for \eqref{eq:disc_ansatz}. This observation is in line with insights from the multi‑objective optimisation literature on PINNs, where penalty‑based formulations often introduce conflicts between loss components and therefore degrade optimisation performance, see for example \cite{DBLP:conf/nips/HwangL24, DBLP:conf/icml/YaoSHL0Z23, DBLP:journals/siamsc/WangTP21}}.
\end{remark}

\begin{algorithm}[htb!] 

\begin{algorithmic}[1] 
\State \textbf{Input:} Parameters including initial learning rates  $\lambda^{\mathfrak{s}}, \lambda^{\mathfrak{t}}$, learning rate schedulers $\mathrm{LR}^{\mathfrak{s}}, \mathrm{LR}^{\mathfrak{t}}$, weights $w_{o_1}, w_{o_2}, \gamma$,  number of collocation points $N$ and number of epochs $M$
\State \textbf{Output:} Neural network solution $\hatuA$ 
\State Initialize weights $\theta^{\mathfrak{s}}$, $\theta^{\mathfrak{t}}$ of neural networks $u_{\theta^{\mathfrak{s}}}$, $v_{\theta^{\mathfrak{t}}}$ (with zero boundary conditions) representing solution and test function respectively 
\For{epoch $=0,1, \hdots, M$} 
\If{epoch is odd}
\State $\theta_t \leftarrow$ \Call{UpdateTestFunction}{$\theta^{\mathfrak{s}}, \theta^{\mathfrak{t}}, \lambda^{\mathfrak{t}}$}: 
\State $\lambda^{\mathfrak{t}} \leftarrow \mathrm{LR}^{\mathfrak{t}}(\lambda^{\mathfrak{t}}$, epoch) \Else{}
\State $\theta_s \leftarrow$ \Call{UpdateSolution}{$\theta^{\mathfrak{s}}, \theta^{\mathfrak{t}}, \lambda^{\mathfrak{s}}$}: 
\State $\lambda^{\mathfrak{s}} \leftarrow \mathrm{LR}^{\mathfrak{s}}(\lambda^{\mathfrak{s}}$, epoch)\EndIf
\EndFor 
\State Set $\hatuA := u_{\theta^{\mathfrak{s}}}$
\end{algorithmic} 

\begin{algorithmic}[1] 
\Function{UpdateTestFunction}{$\theta^{\mathfrak{s}},\theta^{\mathfrak{t}}, \lambda^{\mathfrak{t}}$} 
\State Sample grid points in the interior $X=\{x_1, \hdots, x_N\}$ 
\State $(\hat L^{X}, \hat L_o^X, \hat D^X) \leftarrow$ \Call{ComputeLosses}{$X, \theta^{\mathfrak{s}},\theta^{\mathfrak{t}}$}
\State\label{line:grad_weights}  $g \leftarrow \nabla_{\theta^{\mathfrak{t}}}(-\hat L^X(\theta^{\mathfrak{t}})- \hat D^X(\theta^{\mathfrak{t}}) +\hat L^X_{o}(\theta^{\mathfrak{t}}) )$
\State\label{line:opt} $\theta^{\mathfrak{t}} \leftarrow \texttt{AdamW}(\theta^{\mathfrak{t}}, \lambda^{\mathfrak{t}}, g)$
\State 
\Return $\theta^{\mathfrak{t}}$
\EndFunction 
\end{algorithmic}
\begin{algorithmic}[1] 
\Function{UpdateSolution}{$\theta^{\mathfrak{s}}$,$\theta^{\mathfrak{t}}$ } 
\State Sample grid points in the interior $X=\{x_1, \hdots, x_N\}$
\State $(\hat L^{X}, \hat L_o^X, \hat D^X) \leftarrow$ \Call{ComputeLosses}{$X, \theta^{\mathfrak{s}}, \theta^{\mathfrak{t}}$}
\State\label{line:grad_weights2} $g \leftarrow \nabla_{\theta^{\mathfrak{s}}}(\hat L^X(\theta^{\mathfrak{s}}) + \hat D^X(\theta^{\mathfrak{s}}) +\hat L^X_{o}(\theta^{\mathfrak{s}}) )$
\State\label{line:opt2} $\theta^{\mathfrak{s}} \leftarrow \texttt{AdamW}(\theta^{\mathfrak{s}}, \lambda^{\mathfrak{s}}, g)$
\State 
\Return $\theta^{\mathfrak{s}}$ 
\EndFunction 
\end{algorithmic}
\begin{algorithmic}[1] 
\Function{ComputeLosses}{$X, u_{\theta^{\mathfrak{s}}}, v_{\theta^{\mathfrak{t}}}$} 
\State Approximate via  Monte Carlo integration
\begin{itemize}[itemsep=0pt]
    \item the loss $L(u_{\theta^{\mathfrak{s}}}, v_{\theta^{\mathfrak{t}}})$ by $\hat L^{X}$ 
\item the gap term loss $\| u_{\theta^{\mathfrak{s}}}-v_{\theta^{\mathfrak{t}}} \|^2 $ by $\hat D^X$ 
\item the obstacle loss $\|(\psi - u_{\theta^{\mathfrak{t}}})^+ \|^2 + \|(\psi - v_{\theta^{\mathfrak{t}}})^+ \|^2 $ by $\hat L_o^X$ 
\end{itemize}
\Return $(\hat L^{X}, \hat L_o^X, \hat D^X)$
\EndFunction 
\end{algorithmic}
\caption{Alternating gradient descent ascent algorithm in the $h \equiv 0$ setting} 
\label{alg:1}
\end{algorithm}
The GDA approach alternates between two steps: in the descent step it keeps the test function fixed and updates the weights of the solution candidate by one gradient step, in the ascent step it keeps the solution candidate fixed and updates the test function by one gradient step. Since this basic form of the GDA approach converges to the minmax point if the function is convex-concave (see \S \ref{remark_challenges_nonconvex_nonconcave}) and since the continuous formulation of the minmax problem is indeed convex-concave, one can at least heuristically hope that the GDA algorithm converges to the solution under suitable conditions.  
\begin{remark}
\corr{In the nonconvex–nonconcave setting, it is well known that GDA may fail to converge under standard stepsize choices, and can exhibit cycling or divergence even in simple examples (e.g. see \cite{2025arXiv250501423S}). Variants such as alternating or two-timescale GDA \citep{lin2020gradient}, as well as schemes incorporating negative stepsizes \citep{2025arXiv250501423S}, can improve convergence behaviour in certain settings. In addition, extragradient and optimistic gradient methods (OGDA) are known to provide better stability and convergence properties \citep{ZamaniAbbaszadehpeivastiDeKlerk2024,MokhtariOzdaglarPattathil2020}. Nevertheless, a general theory guaranteeing convergence to global saddle points in the nonconvex–nonconcave regime appears unclear, and iterates may instead converge to stationary points that do not correspond to meaningful minmax solutions. {Furthermore, even the appropriate notion of local optimality in the fully nonconvex-nonconcave case is subtle, and is typically formulated in terms of local minmax points or related stability concepts (see \cite{lin2020gradient,Jordan}).}
A detailed investigation of these algorithmic issues is beyond the scope of this work; accordingly, we restrict attention to the use of GDA and do not address potential stability or convergence issues.}
\end{remark}
Our implementation and codebase can be found at \cite{github}. It is written in \texttt{Python} and uses the \texttt{PyTorch} framework \citep{NEURIPS2019_9015}. The gradient with respect to the weights $\theta_i \in \Theta_i$ that appear in lines \ref{line:grad_weights} and \ref{line:grad_weights2} of Algorithm \ref{alg:1} are calculated using standard automatic differentiation libraries implemented in \texttt{PyTorch}. The optimization steps are each one step of the built-in optimizer \texttt{AdamW} \citep{loshchilov2018decoupled}. \texttt{AdamW}, which belongs to the family of stochastic gradient descent methods, is an adaptive gradient method that adjusts the learning rate of each parameter of the neural network individually such that the learning rates of parameters that typically have a larger gradient during training are slowed down more than the ones which typically have a small gradient. The major hyperparameters (which are detailed below) were found with the use of the \texttt{Optuna} package \citep{optuna_2019}. The results of the numerical experiments that are given below were performed on an NVIDIA A100 80GB PCIe GPU on \texttt{Python} version 3.12.4. By using the package \texttt{Ray} we were able to parallelise up to 10 training sessions on a single GPU; the training was done on a cluster with 4 GPUs.
 
For all our examples we use one of two sets of training hyperparameters, depending on whether the example is in 1D or 2D, see Tables \ref{tab:NN_params} and \ref{tab:config_params} for the most important parameters. While it is certainly possible to improve the results by using individual hyperparameters for each example, we do not do this for simplicity. 

\begin{table}[!htb]
    \centering
    \begin{tabular}{>{\ttfamily}l r}
        \textbf{Parameter} & \textbf{Value} \\
        \hline
        width $(\mathfrak{w})$ & 80\\
        depth $(\mathfrak{d})$ & 4\\
        activation $(\sigma)$ & $\tanh$ \\
        architecture $(\cF)$ & $\mathcal{F}_{\mathrm{DRR}}$\\
        \bottomrule
    \end{tabular}
    \caption{Architecture for the solution and test function.}
    \label{tab:NN_params}
\end{table}
    \begin{table}[!htb]
        \centering
    \begin{tabular}{>{\ttfamily}l r r}
        \textbf{Parameter} & \textbf{1D}  & \textbf{2D}\\
        \hline
n\_interior $(N)$& 1024&1024\\
n\_boundary $(N^b)$ & 2&256\\
\hline
Epochs $(M)$ & 12000 & 12000\\
lr\_soln $(\lambda^{\mathrm{s}})$ &  0.002 &0.003\\
lr\_testfn $(\lambda^{\mathrm{t}})$& 0.001 & 0.0047\\
\hline
$\mathrm{LR}^{\mathrm{s}},  \mathrm{LR}^{\mathrm{t}}$ &\multicolumn{2}{l}{\texttt{CosineAnnealingWarmRestarts}}\\
T\_0& 2001& 2001\\
  T\_mult&  2&  2\\
\hline
weight\_soln\_obs $(w_{o_1})$& 8000 & 5000 \\
weight\_testfn\_obs $(w_{o_2})$& 1500 & 5000\\
weight\_gap\_term $((2\alpha)^{-1})$& 0.0001 & 0.0005\\
\bottomrule
\end{tabular}
    \caption{Parameters for the 1D and 2D examples. \texttt{T\_0} and \texttt{T\_mult} are parameters in the learning rate scheduler.}
    \label{tab:config_params}
    \end{table}

We now present the examples. Along the way, we shall explore and analyse certain aspects related to the choice of weights and parameters.
\FloatBarrier

\subsection{1D Examples}
\subsubsection{Example 1: A benchmark example from \cite[Example 1, \S 3]{DeepNeural}}
Here we set $\Omega = (0,1)$, $h \equiv 0$, 
\begin{align*}
\psi(x) := \begin{cases}
100x^2 &: x \in [0, 0.25],\\
100x(1-x)-12.5 &: x \in (0.25, 0.5],\\
\psi(1-x) &: x \in (0.5, 1],
\end{cases}
\end{align*}
$f \equiv 0$, and solve problem \eqref{eq:VI} with $Au=-u_{xx}$.
The exact solution is
\begin{equation*}
u(x) = 
\begin{cases}
(100 - 50\sqrt{2})x &:\text{$0 \leq x < {1}\slash (2\sqrt{2})$}, \\
100x(1-x)-12.5  &:\text{${1}\slash (2\sqrt{2}) \leq x < 1-{1}\slash (2\sqrt{2})$}, \\
(50\sqrt{2}-100)(1-x) &:\text{$1-1\slash (2\sqrt{2})\leq x \leq 1$}.
\end{cases}
\end{equation*}
 The results are displayed in Figure \ref{fig:sol-diff-one_dim}. Figure \ref{fig:1dplotsolns} shows that our learned solution more or less coincides with the true solution, at least visually. A plot of the obstacle is also included for convenience. In Figure \ref{fig:1dplotdiff} we plot the difference of the true and learned solutions, which confirms that our approach produces a good result.  

\begin{figure}[!htb]
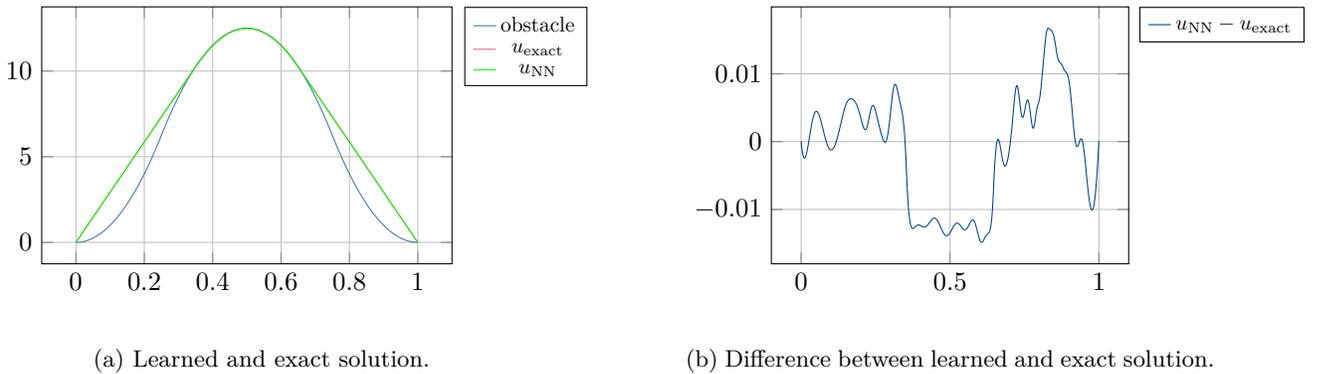
    
\centering
\hspace{-1cm}
    \begin{subfigure}[t]{0.40\textwidth}
           \plotMySolution{tikz_csvs_one_dim}{one_dim}
        \caption{Learned and exact solution. }
        \label{fig:1dplotsolns}
    \end{subfigure}\hspace{1cm}    
    \begin{subfigure}[t]{0.40\textwidth}
            \plotDiff{tikz_csvs_one_dim}{one_dim}        
            \caption{Difference between learned and exact solution.}
                    \label{fig:1dplotdiff}
    \end{subfigure}         
    \caption{\textbf{Example 1}. The difference between the learned solution $u_{\mathrm{NN}}$ and the true solution $u_{\mathrm{exact}}$ is smaller than 0.02. On the coincidence set $[{1}\slash {2\sqrt{2}}, 1-{1}\slash {2\sqrt{2}}]\approx [0.354,0.646]$, $u_{\mathrm{NN}}$ violates the obstacle condition nearly constantly. The maximal difference is comparable to the one obtained in \cite{DeepNeural}.}
    \label{fig:sol-diff-one_dim}
\end{figure}

In Figure \ref{fig:L2_Linf_Errors_1D_eg} we plot the evolution of the $L^2$ and $L^\infty$ errors during training for a number of different runs or initialisations (i.e., different random seeds). The variance in the evolutions for the different runs is explained by this and the Monte Carlo integration, and we see that the variance between the runs is very small.

\begin{figure}[!htb]
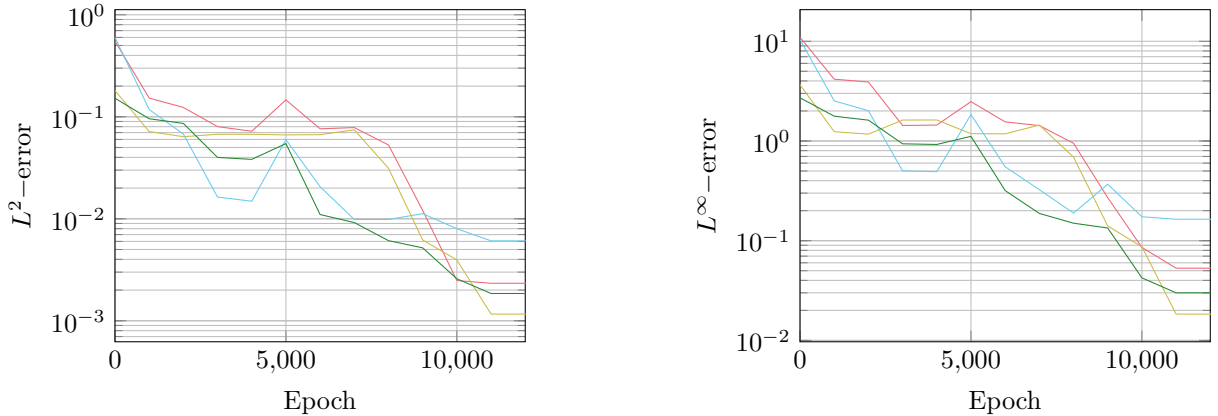

    \centering
    \begin{subfigure}[t]{0.40\textwidth}
\plotTrajectories{tikz_csvs_one_dim}{one_dim}{error_l2}{$L^2-$error}
\label{fig:error_l2_one_dim}        
    \end{subfigure}\hspace{2cm}
\begin{subfigure}[t]{0.40\textwidth}
    \plotTrajectories{tikz_csvs_one_dim}{one_dim}{error_linfty}{$L^\infty-$error}
\label{fig:error_linfty_-one_dim}
\end{subfigure}
\caption{Training trajectories for Example 1. }
\label{fig:L2_Linf_Errors_1D_eg}
\end{figure}

\FloatBarrier

\subsubsection{Example 2: a non-symmetric case}
Let us proceed with an example where the elliptic operator associated to the VI is non-symmetric. As we mentioned, being able to handle non-symmetric VIs is one of the advantages of our work which distinguishes us from other existent work in the literature to the best of our knowledge. We set $\Omega = (-2,2)$, $h \equiv 0$, $\psi(x) = 1-x^2$, define
\begin{align*}
f(x) := \begin{cases}
(4 - 2 \sqrt{3}) &: x \in [-2,-2+\sqrt{3}),\\
-(4 - 2 \sqrt{3}) &: x \in [2-\sqrt{3},2],\\
- ( 2\sqrt{3}-2) &: x \in [-2+\sqrt{3},2-\sqrt{3}],\\
0 &: \text{otherwise},
\end{cases}
\end{align*}
and solve the non-symmetric version of \eqref{eq:VI} with $Au=-u_{xx}+u_x$.
The exact solution is
\begin{equation*}
u(x) = 
\begin{cases}
(4 - 2 \sqrt{3})(x + 2) &:\text{$-2 \leq x < -2 + \sqrt{3}$}, \\
1-x^2 &:\text{$-2 + \sqrt{3}  \leq x < 2 - \sqrt{3}$}, \\
(4 - 2 \sqrt{3})(2-x) &:\text{$2-\sqrt{3} \leq x < 2$}.
\end{cases}
\end{equation*}
This is a modification of the (symmetric) example in \cite[\S 4.1]{TwoNN}. We again observe a pleasing result, see Figure \ref{fig:sol-diff-ns_one_dim}.  
\begin{figure}[!htb]
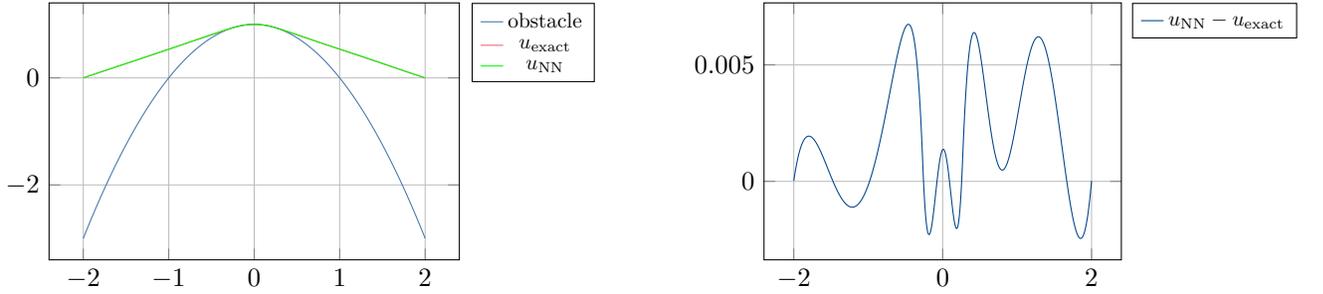

    \centering
    \hspace{-2cm}
    \begin{subfigure}{0.40\textwidth}
        \plotMySolution{tikz_csvs_one_dim}{ns_one_dim}        
    \caption{Learned and exact solution.}
    \end{subfigure}\hspace{1cm}
    \begin{subfigure}{0.40\textwidth}
        \plotDiff{tikz_csvs_one_dim}{ns_one_dim}    
        \caption{Difference between learned and exact solution. }
    \end{subfigure}
    \caption{\textbf{The non-symmetric Example 2}. The difference  has a magnitude smaller than $0.007$.}
    \label{fig:sol-diff-ns_one_dim}
\end{figure}

\begin{figure}[!htb]
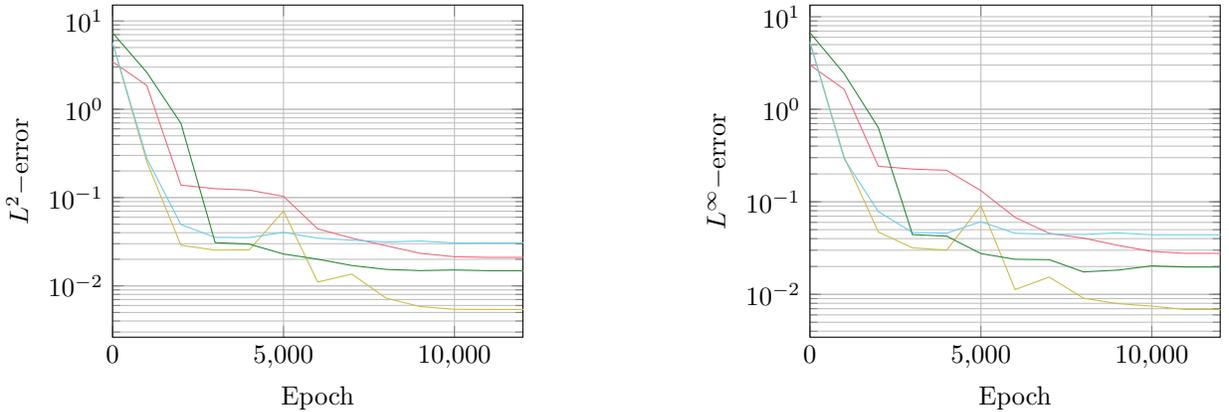

    \centering
    \begin{subfigure}[t]{0.40\textwidth}  
        \plotTrajectories{tikz_csvs_one_dim}{ns_one_dim}{error_l2}{$L^2-$error}
\label{fig:error_l2_ns_one_dim}    
    \end{subfigure}
\hspace{2cm}
\begin{subfigure}[t]{0.40\textwidth}
    \centering
        \plotTrajectories{tikz_csvs_one_dim}{ns_one_dim}{error_linfty}{$L^\infty-$error}
\label{fig:error_linf_ns_one_dim}
    \end{subfigure}
\caption{Training trajectories for Example 2.  Compared to Figure \ref{fig:L2_Linf_Errors_1D_eg}, the convergence happens  earlier.}
\end{figure}

\paragraph{Obstacle weight experiment}
Recall that we enforce the obstacle condition via penalty terms for the solution and test function with weights $w_{o_1}$ and $w_{o_2}$ respectively. In Figure \ref{fig:obs-weight-vs_error-C} we study the effect that different weights $w_{o_1}=w_{o_2}$ have on the error in the context of Example 2. We see that there is an optimal value for the penalty, which is reasonable: for small weights the solution is not punished enough for violating the obstacle constraint and for large weights, during the training phase, the solution candidate is pushed above the obstacle so that halting at the correct solution is overridden by the ``momentum" that arises from the obstacle term.

\begin{figure}[!htb]
\centering
      \makedarkfull{YlOrBr}
\begin{subfigure}[t]{0.30\textwidth}
\begin{tikzpicture}
\begin{axis}[
   xlabel={Obstacle weight},
   ylabel={$L^2$-error},
x tick label style={rotate=25} ,
   yticklabel style={
        /pgf/number format/fixed,
        /pgf/number format/precision=5
	},
   width=\textwidth,
   colormap/YlOrBrDarkFull, 
   cycle list={[of colormap]},
]
\addplot[scatter,
   only marks,] table [x=weight_soln_obs, y=error_l2,col sep=comma] {visu/tikz_csvs_obs_force/mean.csv};
\end{axis}
\end{tikzpicture}
\end{subfigure}
\hspace{2cm}
\begin{subfigure}[t]{0.30\textwidth}
\begin{tikzpicture}
\begin{axis}[
   xlabel={Obstacle weight},
   ylabel={$H^1$-error},
      x tick label style={rotate=25} ,
   yticklabel style={
        /pgf/number format/fixed,
        /pgf/number format/precision=5
	},
width=\textwidth,
   colormap/YlOrBrDarkFull, 
   cycle list={[of colormap]},
]
\addplot[scatter,
   only marks,] table [x=weight_soln_obs, y=H_one_norm, col sep=comma] {visu/tikz_csvs_obs_force/mean.csv};
\end{axis}
\end{tikzpicture}
\end{subfigure}
\caption{Mean (over 50 seeds) of $L^2$ and $H^1$ errors for Example 2 for different values of $w_{o_1}=w_{o_2}$.}
   \label{fig:obs-weight-vs_error-C}
\end{figure}
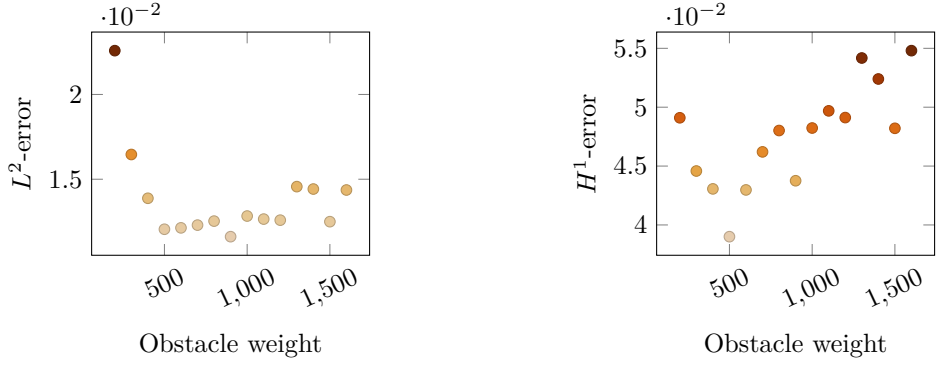
A weight of $900$ leads to an optimum in the $L^2$-sense, but once gradients are taken into account, $500$ appears to be the optimal choice (with $900$ not too distant). By examining the individual errors that produced each mean, $500$ has a lower variance in $L^2$ and the variances are comparable in $H^1$,  but the lowest error for $900$ bests the lowest error for $500$, across both measures of error. That said, we emphasise that both weights produce low errors: they differ only on an order of $10^{-2}$.

\FloatBarrier

\subsubsection{Example 3: a piecewise smooth case}
Our final one-dimensional example is taken from \cite[\S 2]{CM}; we call this `piecewise' because the solution is composed of five separate pieces and to the unsuspecting eye appears to be a piecewise affine function. The setting here is $\Omega = (-1,1)$, $h \equiv 0$, $f \equiv 0$, and with $\alpha = 0.4$, 
\begin{align*}
\psi(x) := \begin{cases}
\varphi\left(x+\frac 12\right)\left(\frac 32 - 12|x+\frac 12|^{2-\alpha}\right) - \frac 12 &: x \in (-1, 0],\\
\varphi\left(x-\frac 12\right)\left(\frac 32 - 12|x-\frac 12|^{2-\alpha}\right) - \frac 12 &: x \in (0,1),
\end{cases}
\end{align*}
where $\varphi \in C_c^\infty(\mathbb{R})$ satisfies $0 \leq \varphi \leq 1$, $\varphi=1 \text{ in } (-0.3, 0.3),$ and $\mathrm{supp}(\varphi) \subset [-0.4, 0.4].$
We solve the problem \eqref{eq:VI} with $Au=-u_{xx}$.
The exact solution is
\begin{equation}
u(x) = 
\begin{cases}
\psi(-\beta-0.5)\frac{x+1}{0.5-\beta} &: x \in (-1, -\beta - 0.5), \\
\psi(x) &: {x \in [-0.5-\beta, -0.5)}, \\
1 &:{x \in [-0.5, 0.5)},\\
\psi(x) &:{x \in [0.5, 0.5+\beta)}, \\
\psi(\beta+0.5)\frac{x-1}{\beta-0.5} &:{x \in [\beta + 0.5, 1)}, 
\end{cases}
\end{equation}
where the constant $\beta$ is the unique solution of the equation
\[\psi(-\beta - 0.5) = (0.5-\beta)\psi'(-\beta - 0.5), \quad \beta \in (0, 0.3).\]
In practice, we take $\beta = 0.02376$ as an approximate solution of the equation, and as for the function $\varphi$, we use
\[\varphi(x) := \frac{\mu (0.4 - |x|)}{\mu(|x| - 0.3) + \mu(0.4 - |x|)},\qquad\text{where}\qquad \mu(x) := \begin{cases}
    \exp(-1\slash x) &: x > 0,\\
    0 &: x \leq 0.
\end{cases}\]
The results can be seen in Figure \ref{fig:sol-diff-CM}.

\begin{figure}[!htb]
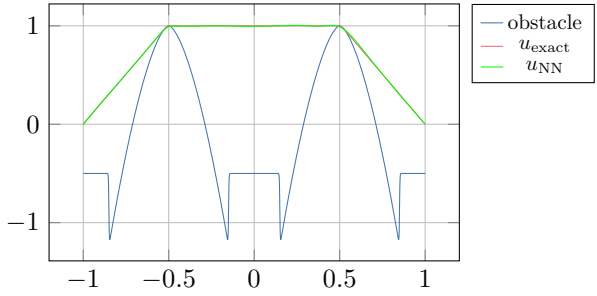
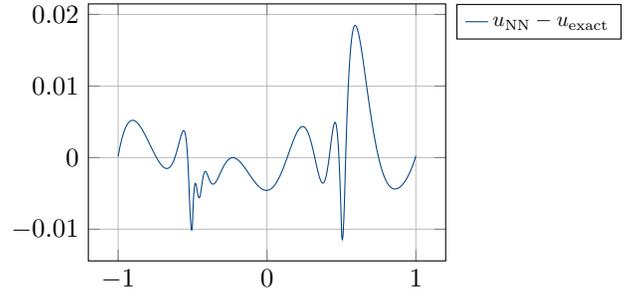

    \centering
    \hspace{-2cm}
    \begin{subfigure}{0.40\textwidth}
        \plotMySolution{tikz_csvs_one_dim}{CM}
    \caption{Learned and exact solution.}
    \end{subfigure}\hspace{1cm}
    \begin{subfigure}{0.40\textwidth}
        \plotDiff{tikz_csvs_one_dim}{CM}
        \caption{Difference between learned and exact solution.}
    \end{subfigure}
    \caption{\textbf{Piecewise smooth case of Example 3}. The difference to the exact solution has a magnitude smaller than $0.03$. Compared to the small average difference the violation of the obstacle condition at the peaks (for $x=0.5$ and $x=-0.5$) is relatively large.}
    \label{fig:sol-diff-CM}
\end{figure}

\begin{figure}[!htb]
\begin{subfigure}[t]{0.40\textwidth}  
        \plotTrajectories{tikz_csvs_one_dim}{CM}{error_l2}{$L^2-$error}
\label{fig:error_l2_CM}    
    \end{subfigure}
\hspace{2cm}
\begin{subfigure}[t]{0.40\textwidth}
    \centering
        \plotTrajectories{tikz_csvs_one_dim}{CM}{error_linfty}{$L^\infty-$error}
\label{fig:error_linfty_CM}
    \end{subfigure}
\caption{Training trajectories for Example 3.}
\end{figure}\paragraph{Mesh independence experiment}
Let us look at how the $L^2$ error for Example 3 behaves as we change the size (i.e., width and depth) of the neural networks parametrising the solution and test function (both networks share the same architecture). In Figure \ref{fig:mi_CM}  (left) we visualise the error for each depth as the width varies over the horizontal axis. We see that when the depth is sufficiently large ($\geq 6)$, the errors for all widths are within a margin of approximately $10^{-2}$. Likewise, when the width is large enough ($\geq 40$), the errors stay below approximately $0.015.$ This can be interpreted as a kind of mesh independence: the size of the network determines the number of learnable parameters, and the figure shows that the quality of the approximation is (roughly) independent of the number of learnable parameters. It is also illustrative to see how the error varies with respect to the total number of learnable weights (which is a function of the width multiplied by the depth), see Figure \ref{fig:mi_CM} (right). The trend is clearly that the more weights the better the solution, but we again see a threshold level of weights beyond which the error is acceptable. 
 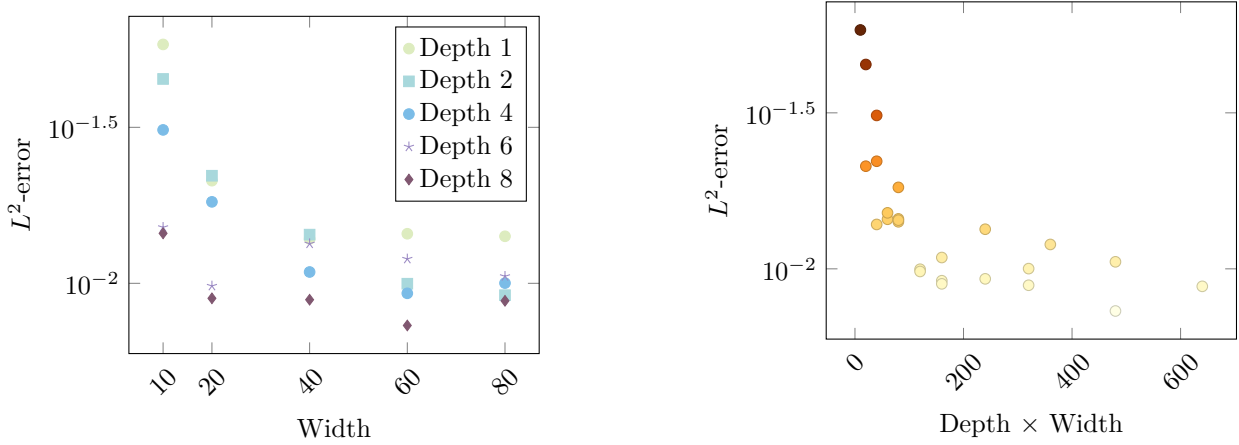
\begin{figure}
 \begin{subfigure}{0.40\textwidth}
\begin{tikzpicture}
\definecolor{clr_mi_1}{HTML}{DDECBF}
\definecolor{clr_mi_2}{HTML}{A8D8DC}
\definecolor{clr_mi_3}{HTML}{7BBCE7}
\definecolor{clr_mi_4}{HTML}{9B8AC4}
\definecolor{clr_mi_5}{HTML}{805770}
\begin{axis}[
   xlabel={Width},
   ylabel={$L^2$-error},
   legend pos=north east,
   xtick=data,
   ymode=log,
   width=\textwidth,
   x tick label style={rotate=45}            
]
\addplot+[only marks, mark options={draw=clr_mi_1, fill=clr_mi_1}] table [x=width, y=error_l2,col sep=comma] {visu/tikz_csvs_NN_dim_CM/mean_D_1.csv};
\addplot+[only marks, mark options={draw=clr_mi_2, fill=clr_mi_2}] table [x=width, y=error_l2,col sep=comma] {visu/tikz_csvs_NN_dim_CM/mean_D_2.csv};
\addplot+[only marks, mark options={draw=clr_mi_3, fill=clr_mi_3}] table [x=width, y=error_l2,col sep=comma] {visu/tikz_csvs_NN_dim_CM/mean_D_4.csv};
\addplot+[only marks, mark options={draw=clr_mi_4, fill=clr_mi_4}] table [x=width, y=error_l2,col sep=comma] {visu/tikz_csvs_NN_dim_CM/mean_D_6.csv};
\addplot+[only marks, mark options={draw=clr_mi_5, fill=clr_mi_5}] table [x=width, y=error_l2,col sep=comma] {visu/tikz_csvs_NN_dim_CM/mean_D_8.csv};
\legend{Depth 1, Depth 2, Depth 4, Depth 6, Depth 8}
\end{axis}
\end{tikzpicture}
\end{subfigure}
\hspace{2cm}
\begin{subfigure}{0.40\textwidth}
\begin{tikzpicture}
\begin{axis}[
   xlabel={Depth $\times$ Width},
   ylabel={$L^2$-error},
   legend pos=north east,
   ymode=log,
   width=\textwidth,
  colormap/YlOrBr,
   cycle list={[of colormap]},
   x tick label style={rotate=45}            
]
\addplot[scatter, only marks] table [x=nb_weights, y=error_l2,col sep=comma] {visu/tikz_csvs_NN_dim_CM/mean.csv};
\end{axis}
\end{tikzpicture}
\end{subfigure}
\caption{Effect of neural network size on the solution of Example 3 (means over 4 seeds are plotted).}
\label{fig:mi_CM}
\end{figure}
\FloatBarrier

\subsection{2D Examples}
\subsubsection{Example 4: an example from optimal control}
We take this example from \cite[\S 7.1]{MT}. Set $\Omega = (0,1)^2$, $A=-\Delta$, $\psi \equiv 0$, $h \equiv 0$, and using the intermediary quantities
\begin{align*}
z_1(x) &:= -4096x^6 + 6144x^5 - 3072x^4 + 512x^3,\\
z_2(x) &:= -244.140625 x^6 + 585.9375x^5 - 468.75x^4 + 125x^3,\\
\zeta &:= \begin{cases}
z_1(x-0.5)z_2(y) &: x \in (0.5,1) \text{ and } y \in (0, 0.8),\\
0 &: \text{otherwise},
\end{cases}
\end{align*}
we define the source term
\begin{align*}
f(x,y) := -\zeta - \begin{cases}
z_1(x)z_2''(y) + z_1''(x)z_2(y) &: x < 0.5 \text{ and } y < 0.8,\\
0 &: \text{otherwise},
\end{cases}
\end{align*}
and  consider the VI \eqref{eq:VI} with $Au=-\Delta u$.
The exact solution is
\begin{align*}
u(x,y) &= \begin{cases}
z_1(x)z_2''(y) + z_1''(x)z_2(y) &: x < 0.5 \text{ and } y < 0.8,\\
0 &: \text{otherwise}.
\end{cases}
\end{align*}
We visualise the results in Figure \ref{fig:sol-diff-MT}. Since the obstacle is the zero function, we do not plot it explicitly. 
\begin{figure}[!htb]
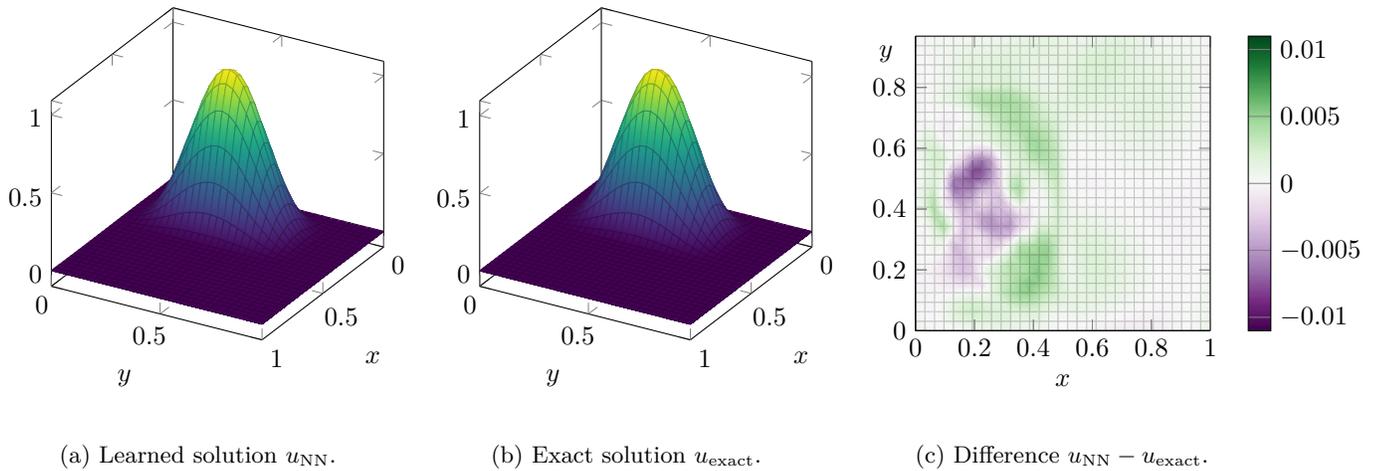

    \begin{subfigure}{0.22\textwidth}
\plotSurface{two_dim}{MT}{sol_nn}
    \caption{Learned solution $u_{\mathrm{NN}}$.}
    \end{subfigure}\hspace{1.5cm}
\begin{subfigure}{0.22\textwidth}
\plotSurface{two_dim}{MT}{gt_nn}
    \caption{Exact solution $u_{\mathrm{exact}}$.}
    \end{subfigure}\hspace{1.6cm}
    \begin{subfigure}{0.23\textwidth} 
\plotSurfaceTopView{two_dim}{MT}{diff}{0.011}
        \caption{Difference $u_{\mathrm{NN}}-u_{\mathrm{exact}}$.}
    \end{subfigure}
    \caption{\textbf{Example 4}. Similarly to the 1D examples the solution candidate rather slightly violates the obstacle conditions (this is illustrated by the purple area in the third plot).}
    \label{fig:sol-diff-MT}
\end{figure}

\begin{figure}[!htb]
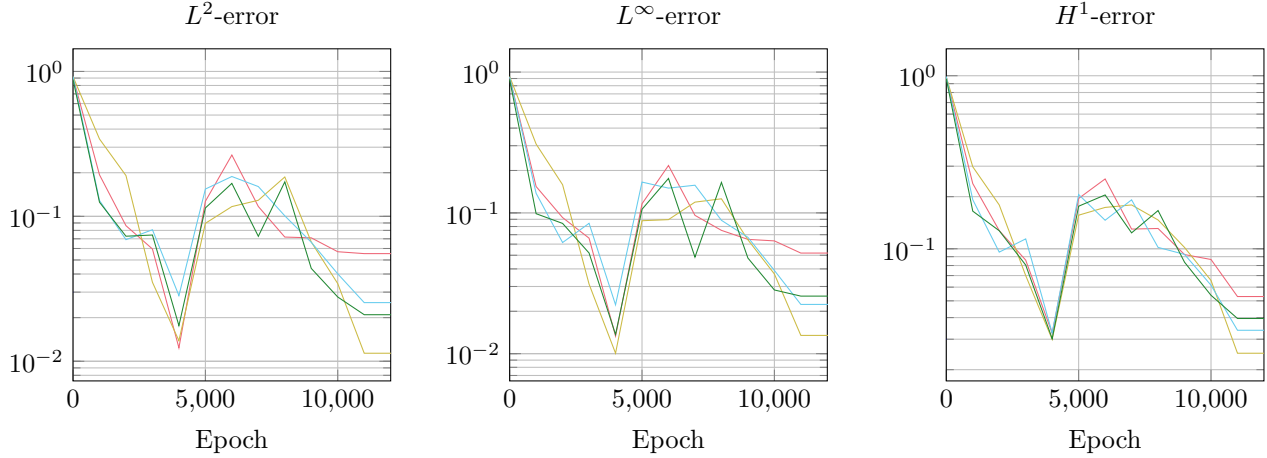

    \centering
\begin{subfigure}[t]{0.33\textwidth}
\plotTrajectoriesLong{tikz_csvs_two_dim}{MT}{error_l2}{$L^2$-error}{12000}
\end{subfigure}\begin{subfigure}[t]{0.33\textwidth}
\plotTrajectoriesLong{tikz_csvs_two_dim}{MT}{error_linfty}{$L^\infty$-error}{12000}
\end{subfigure}\begin{subfigure}[t]{0.33\textwidth}
\plotTrajectoriesLong{tikz_csvs_two_dim}{MT}{H_one_norm}{$H^1$-error}{12000}
\end{subfigure}
\caption{Training trajectories for Example 4. This example converges already at $4000$ epochs; the warm restart phase (which occurs due to the learning rate scheduler) from $4000$ till the end does not greatly improve the error for all runs, \corr{see also Remark \ref{rem:warm_restart}}. Note the correlation between the $L^p$ and $H^1$ errors.}
\label{fig:loss_r_two_dim}
\end{figure}\paragraph{Mesh independence experiment}
We again take a look at how the network sizes affect the quality of the solution, now for Example 4, see Figure \ref{fig:mi_MT}.
\begin{figure}
\begin{subfigure}{0.40\textwidth}
\begin{tikzpicture}
\definecolor{clr_mi_1}{HTML}{DDECBF}
\definecolor{clr_mi_2}{HTML}{A8D8DC}
\definecolor{clr_mi_3}{HTML}{7BBCE7}
\definecolor{clr_mi_4}{HTML}{9B8AC4}
\definecolor{clr_mi_5}{HTML}{805770}
\begin{axis}[
   xlabel={Width},
   ylabel={$L^2$-error},
   legend pos=north east,
   xtick=data,
   ymode=log,
   width=\textwidth,
   x tick label style={rotate=45}            
]
\addplot+[only marks, mark options={draw=clr_mi_1, fill=clr_mi_1}] table [x=width, y=error_l2,col sep=comma] {visu/tikz_csvs_NN_dim_MT/mean_D_1.csv};
\addplot+[only marks, mark options={draw=clr_mi_2, fill=clr_mi_2}] table [x=width, y=error_l2,col sep=comma] {visu/tikz_csvs_NN_dim_MT/mean_D_2.csv};
\addplot+[only marks, mark options={draw=clr_mi_3, fill=clr_mi_3}] table [x=width, y=error_l2,col sep=comma] {visu/tikz_csvs_NN_dim_MT/mean_D_4.csv};
\addplot+[only marks, mark options={draw=clr_mi_4, fill=clr_mi_4}] table [x=width, y=error_l2,col sep=comma] {visu/tikz_csvs_NN_dim_MT/mean_D_6.csv};
\addplot+[only marks, mark options={draw=clr_mi_5, fill=clr_mi_5}] table [x=width, y=error_l2,col sep=comma] {visu/tikz_csvs_NN_dim_MT/mean_D_8.csv};
\legend{Depth 1, Depth 2, Depth 4, Depth 6, Depth 8}
\end{axis}
\end{tikzpicture}
\end{subfigure}
\hspace{2cm}
\begin{subfigure}{0.40\textwidth}
\begin{tikzpicture}
\begin{axis}[
   xlabel={Depth $\times$ Width},
   ylabel={$L^2$-error},
   legend pos=north east,
   ymode=log,
   width=\textwidth,
  colormap/YlOrBr,
   cycle list={[of colormap]},
   x tick label style={rotate=45}            
]
\addplot[scatter, only marks] table [x=nb_weights, y=error_l2,col sep=comma] {visu/tikz_csvs_NN_dim_MT/mean.csv};
\end{axis}
\end{tikzpicture}
\end{subfigure}
\caption{Effect of neural network size on the solution of Example 4 (1 seed is used).}
\label{fig:mi_MT}
\end{figure}
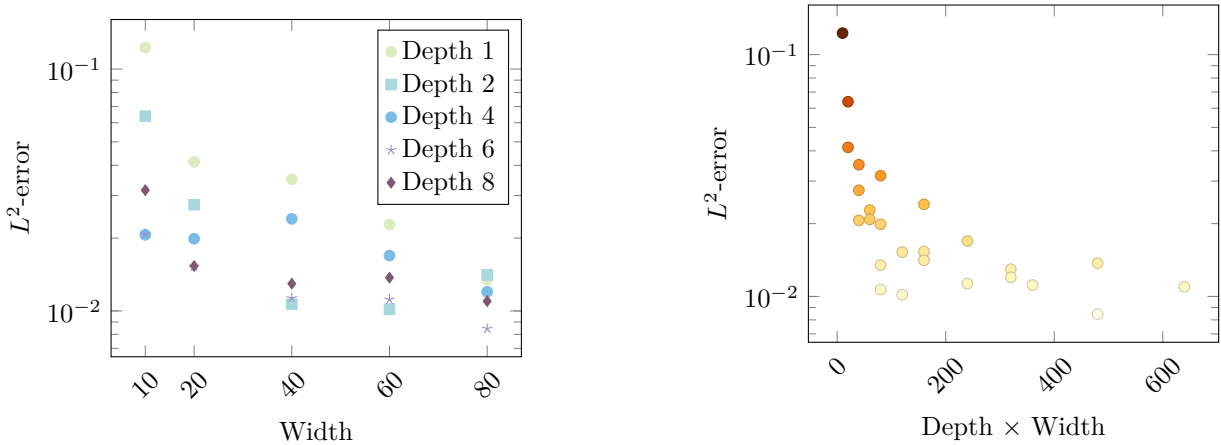
Note that it is not the case that the bigger the network, the better the solution. Larger networks may require a longer training period (number of epochs), whereas smaller networks may suffer from underfitting. The optimal  for this particular example and run appears to be a network architecture of depth $2$ and width $40$; we again recall the fact that we choose our architecture to provide a good result for all the examples that we considered and a case-by-case fine-tuning would likely improve our results, but this is not the focus of the paper.

\begin{remark}\label{rem:warm_restart}
\corr{In Figure \ref{fig:loss_r_two_dim} the training error has a noticeable jump at the beginning of the warm restart period. We investigated this effect in ablation experiments. (1) With a standard plateau scheduler, 
even after hyperparameter tuning, the overall results were significantly worse,
both in terms of final accuracy and stability. (2) With longer restart periods, 
the frequency of jumps is reduced but this does not improve
the final error; in fact, the overall performance slightly degrades.
This indicates that the restart mechanism is beneficial for convergence in this problem class.}

\end{remark}

\FloatBarrier

\subsubsection{Examples 5 and 6: biactive cases}
Here we look at situations where biactivity (or lack of strict complementarity) is present. Let us briefly explain what this means.  Recall from \eqref{eq:CS} that solutions of the VI satisfy the condition $(Au-f)(u-\psi)= 0$. This means that pointwise a.e., either $Au-f$ or $u-\psi$ or both must be equal to zero. When both are zero, we say that biactivity is present, i.e., the multiplier $Au-f$ vanishes on the coincidence set $\{u=\psi\}$. It is well known that traditional numerical methods such as active set update strategies face great difficulty in handling problems with biactivity. 

We consider first an example taken from \cite[\S 4.8.1]{KS}. Set $\Omega = (-1,1)^2$, 
\[f(x,y) = \begin{cases}
    0 &: x <0,\\
    -12x^2 &: \text{otherwise},
\end{cases}\] $\psi \equiv 0$, and the exact solution
\[u(x,y) = \begin{cases}
    0 &: x <0,\\
    x^4 &: \text{otherwise},
\end{cases}
\]
of \eqref{eq:VI} with $Au=-\Delta u$. The boundary data $h$ is determined by $u$. The results are shown in Figure \ref{fig:ks}.

\begin{figure}[!htb]
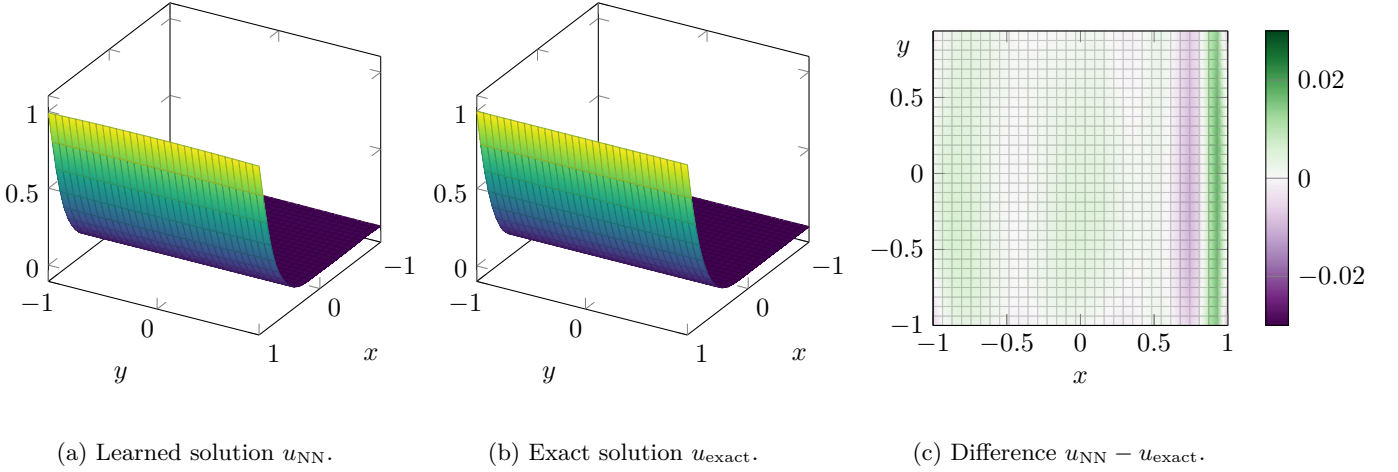

    \begin{subfigure}{0.22\textwidth}
\plotSurface{two_dim}{KS}{sol_nn}
    \caption{Learned solution $u_{\mathrm{NN}}$.}
    \end{subfigure}\hspace{1.5cm}
    \begin{subfigure}{0.22\textwidth}
\plotSurface{two_dim}{KS}{gt_nn}
    \caption{Exact solution $u_{\mathrm{exact}}$.}
    \end{subfigure}\hspace{1.6cm}
        \begin{subfigure}{0.23\textwidth}
\plotSurfaceTopView{two_dim}{KS}{diff}{0.03}
        \caption{Difference $u_{\mathrm{NN}}-u_{\mathrm{exact}}$.}
    \end{subfigure}        
        \caption{\textbf{Example 5}: biactive case from \cite[\S 4.8.1]{KS}.}
        \label{fig:ks}
\end{figure}

     \FloatBarrier

Our next problem, Example 6, also exhibits biactivity and additionally displays a nonsmoothness in the multipler $Au-f$, and is taken from \cite[\S 4.8.3]{KS}. Set $\Omega=(-1,1)^2$, $\psi \equiv 0$, and the exact solution
\[u(x,y)=\begin{cases}
    (1-4x^2-4y^2)^4 &: x^2 + y^2 < \frac 14,\\
    0 &: \text{otherwise},
\end{cases}\]
informing the source term
\[f(x,y)=-\Delta u(x,y) - \begin{cases}
    1 &: x^2 + y^2 > \frac 34,\\
    0 &: \text{otherwise},
\end{cases}\]
for \eqref{eq:VI} with $Au=-\Delta u$. The boundary data $h$ again is taken from $u$. See Figure \ref{fig:ks_ns} for the results. 
\begin{figure}[!htb]
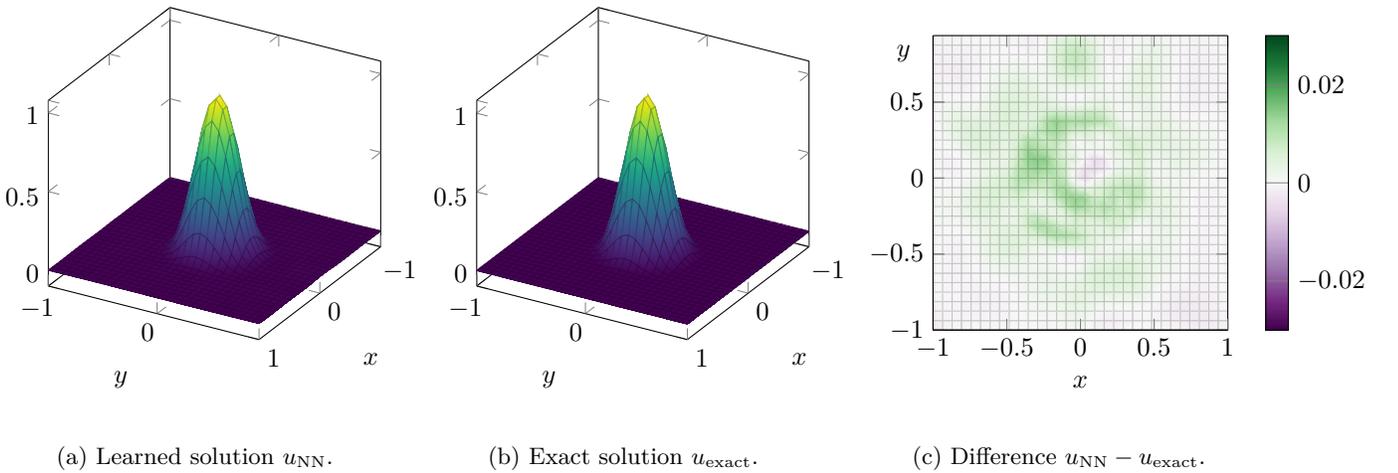

    \begin{subfigure}{0.22\textwidth}
\plotSurface{two_dim}{KS_ns}{sol_nn}
    \caption{Learned solution $u_{\mathrm{NN}}$.}
    \end{subfigure}\hspace{1.5cm}
    \begin{subfigure}{0.22\textwidth}
\plotSurface{two_dim}{KS_ns}{gt_nn}
    \caption{Exact solution $u_{\mathrm{exact}}$.}
    \end{subfigure}\hspace{1.6cm}
        \begin{subfigure}{0.23\textwidth}
\plotSurfaceTopView{two_dim}{KS_ns}{diff}{0.03}
        \caption{Difference $u_{\mathrm{NN}}-u_{\mathrm{exact}}$.}
    \end{subfigure}        
        \caption{\textbf{Example 6}: second biactive case taken from \cite[\S 4.8.3]{KS}.}
        \label{fig:ks_ns}
\end{figure}
 In both examples, we see that our solution algorithm performs well (the results for Example 5 are comparable to the other examples; Example 6  is slightly more tricky, see \S \ref{sec:randomness}) and does not encounter the biactivity as an issue due to our approach. Thus our work offers a potential advantage over active set methods. 
\FloatBarrier
\paragraph{Regularised gap weight experiment} 
We investigate what effect changing the weight $1\slash (2\gamma)$ in front of the regularised gap term has on the solution for Example 6. See Figure \ref{fig:diagonal_forcing_experiment} for a plot of the $H^1$ errors observed for different choices of gap weight. Observe that the error improves as the weight is increased until an optimum point is reached, after which the error increases exponentially. This tallies with our earlier remarks that $\gamma$ should be sufficiently large (hence $1\slash(2\gamma)$ should be sufficently small), but taking too small a value does not lead to a huge improvement, i.e., there is a trade-off.
 
\begin{figure}[!htb]
   \centering
\begin{subfigure}{0.30\textwidth}
   \makedarkfull{BuPu}
\begin{tikzpicture}
\begin{axis}[
   xlabel={Regularised gap weight $1\slash (2\gamma)$},
   ylabel={$H^1$-error},
xmode=log,
ymode=log,
width=\textwidth,
colormap/BuPuDarkFull, 
   cycle list={[of colormap]},
   x tick label style={rotate=45}            
]
\addplot[scatter,
   only marks,] table [x=diagonal_force, y=H_one_norm,col sep=comma] {visu/tikz_csvs_diagonal_force/mean.csv};
\end{axis}
\end{tikzpicture}
\end{subfigure}
   \caption{Mean over 10 seeds of the $H^1$-errors for Example 6 for different regularised gap weights.}
   \label{fig:diagonal_forcing_experiment}
\end{figure}
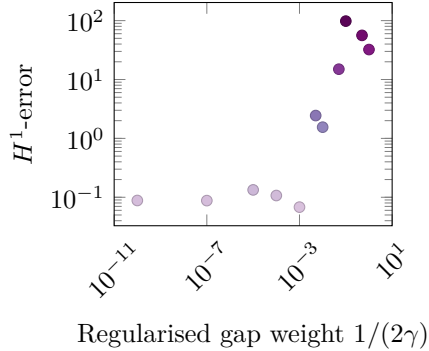

     \FloatBarrier

\subsection{Replicability and randomness}\label{sec:randomness}
It is well known that an improperly designed solution algorithm or loss function can lead to results that are radically different depending on the machine that the code is run on or depending on the choice of the random seed. In the context of our work, this means that it is worthwhile to check whether the hyperparameters we suggest actually produce quantifiably good results (i.e., low error of the learned solution in comparison to the exact solution) when run on different devices. As a statistical analysis of running our code on many different devices is impractical, we instead perform a sensitivity analysis where we vary the random seed and check whether the resulting model is robust with respect to the seed (the idea being that changing the random seed is akin to using a different device). 

We present box plots of $L^2$ errors corresponding to different seeds for our examples in Figure \ref{fig:box1} and Figure \ref{fig:box2}. The bottom of each box corresponds to the 25th percentile and the top to the 75th percentile of 50 runs. For the 1D examples we see that the 75th percentile is smaller than $0.025$; for the 2D examples the 75th percentile is smaller than $0.065$.

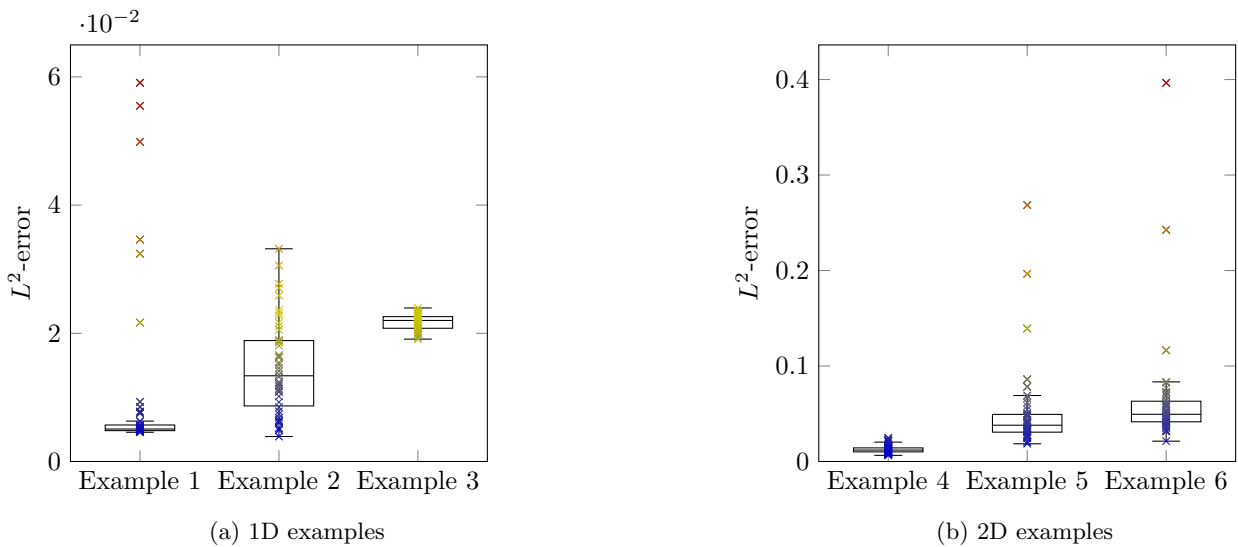
\begin{figure}[!htb]
    \centering
    \begin{subfigure}[t]{0.45\textwidth}
\begin{tikzpicture}
\pgfplotstableread[col sep=comma]{visu/tikz_csvs_hist/hist_one_dim/hist.csv}\dataonedim
    \pgfplotstableread[col sep=comma]{visu/tikz_csvs_hist/hist_ns_one_dim/hist.csv}\dataNSonedim
    \pgfplotstableread[col sep=comma]{visu/tikz_csvs_hist/hist_CM/hist.csv}\dataCM    
    \definecolor{clr_boxplot1}{HTML}{CCDDAA}  
    \definecolor{clr_boxplot2}{HTML}{66CCEE} 
    \definecolor{clr_boxplot3}{HTML}{CCBB44}  
    \begin{axis}[
        width=0.9\textwidth, height=0.9\textwidth, xlabel={},
        ylabel={$L^2$-error},
        ymin=0,
        boxplot/draw direction=y,
        boxplot/box extend=0.5,
        boxplot/whisker extend=0.2, boxplot/every outlier/.style={mark=*, mark size=2, fill=black}, xtick={1,2,3,4,5,6,7,8,9},
        xticklabels={Example 1, Example 2, Example 3},
]
        \addplot[boxplot, boxplot/draw position=1, mark=x, mark options={color=black}] table[y index=3]{\dataonedim};      
        \addplot[boxplot, boxplot/draw position=2, mark=x, mark options={color=black}] table[y index=3]{\dataNSonedim};
        \addplot[boxplot, boxplot/draw position=3, mark=x, mark options={color=black}] table[y index=3]{\dataCM};
       
        \addplot [scatter, only marks, mark=x, mark options={fill=black, color=black}]
            table[x expr=1, y index=3]{\dataonedim};
        \addplot [scatter, only marks, mark=x, mark options={fill=black, color=black}]
            table[x expr=2, y index=3]{\dataNSonedim};
        \addplot [scatter, only marks, mark=x, mark options={fill=black, color=black}]
            table[x expr=3, y index=3]{\dataCM};         
    \end{axis}
\end{tikzpicture}
\caption{1D examples}
    \label{fig:box1}
\end{subfigure}
\hfill
        \begin{subfigure}[t]{0.45\textwidth}
\begin{tikzpicture}
         \pgfplotstableread[col sep=comma]{visu/tikz_csvs_hist_two_dim/hist_MT/hist.csv}\dataMT  
    \pgfplotstableread[col sep=comma]{visu/tikz_csvs_hist_two_dim/hist_two_dim/hist.csv}\dataTwoDim  
    \pgfplotstableread[col sep=comma]{visu/tikz_csvs_hist_two_dim/hist_KS/hist.csv}\dataKS  
    \pgfplotstableread[col sep=comma]{visu/tikz_csvs_hist_two_dim/hist_KS_ns/hist.csv}\dataKSns  
        \pgfplotstableread[col sep=comma]{visu/tikz_csvs_hist_two_dim/hist_NSV_two_dim/hist.csv}\dataNSV  
 
    \begin{axis}[
        width=0.9\textwidth, height=0.9\textwidth, xlabel={},
        ylabel={$L^2$-error},
        ymin=0,
        boxplot/draw direction=y,
        boxplot/box extend=0.5,
        boxplot/whisker extend=0.2, boxplot/every outlier/.style={mark=*, mark size=2, fill=black}, xtick={1,2,3},
        xticklabels={Example 4,
Example 5, 
        Example 6},
]
        \addplot[boxplot, boxplot/draw position=1, mark=x, mark options={color=black}] table[y index=3]{\dataMT};
       \addplot[boxplot, boxplot/draw position=2, mark=x, mark options={color=black}] table[y index=3]{\dataKS};
       \addplot[boxplot, boxplot/draw position=3, mark=x, mark options={color=black}] table[y index=3]{\dataKSns};
           
        \addplot [scatter, only marks, mark=x, mark options={fill=black, color=black}] table[x expr=1, y index=3]{\dataMT};
        \addplot [scatter, only marks, mark=x, mark options={fill=black, color=black}] table[x expr=2, y index=3]{\dataKS};
        \addplot [scatter, only marks, mark=x, mark options={fill=black, color=black}] table[x expr=3, y index=3]{\dataKSns};
    \end{axis}
\end{tikzpicture}
\caption{2D examples}
    \label{fig:box2}
\end{subfigure}
\caption{Box plots of $L^2$ errors corresponding to 50 different seeds for the 1D and 2D examples. }
    \label{fig:all_boxes}
\end{figure}
We refer to \cite{WeNeedTo} for more details and suggestions of good practices on this topic, and also to \cite{ModelStability} for a study of model stability with respect to random seeds and techniques for improvement (see also references therein).

\corr{
\begin{remark}[Comparison to traditional approaches]
    For the numerical examples considered here, unsurprisingly, traditional approaches such as a Moreau--Yosida path-following with finite element methods  achieve highly accurate results with little cost. The aim of our numerics is to demonstrate the applicability of our theory and the aim of this paper is to provide a theoretically sound neural network framework for VIs; optimising our implementation for high-precision solutions is outside the scope of this work. Furthermore, the value of the approach lies in its mesh-free nature, flexibility and applicability to high-dimensional problems, where classical approaches are less practical. 
\end{remark}
}

    \FloatBarrier

\acks{We thank Guozhi Dong (Central South University, Changsha) and Pavel Dvurechensky (WIAS, Berlin) for helpful discussions. MH was partially supported by the Deutsche Forschungsgemeinschaft (DFG, German Research Foundation) through the DFG SPP 1962 Priority Programme \emph{Non-smooth and Complementarity-based
Distributed Parameter Systems: Simulation and Hierarchical Optimization} within project 10, and partially by the DFG under Germany's Excellence Strategy – The Berlin Mathematics Research Center MATH+ (EXC-2046/1, project ID: 390685689). \corr{We also thank the two anonymous referees for their helpful comments.}}
  
\appendix
\appendix
\section{Proofs}\label{app:proofs}
\emph{Proof of Proposition \ref{prop:VIandCSEquivalence}}. We follow \cite[\S 1:3, p.~4]{Rodrigues}. Let $u$ solve \eqref{eq:CS} with the stated regularity.  Taking $v \in K$, we have, making use of the Gelfand triple structure $\X \subset L^2(\Omega) \subset \X^*$ and that $Au \in L^2(\Omega)$,
\begin{align*}
    \langle Au-f, u-v \rangle &=     \int_\Omega (Au-f)(u-\psi) +     \int_\Omega (Au-f)(\psi-v) \leq 0.
\end{align*}
This shows that $u$ solves \eqref{eq:VI}. Now for the reverse direction, suppose $u$ solves \eqref{eq:VI}. Choose the test function $v=u+\varphi$ where $\varphi \in C_c^\infty(\Omega)$ and $\varphi \geq 0$. Then since $Au \in L^2(\Omega)$,  we get
\[\int_\Omega (Au-f)\varphi \geq 0\]
whence the arbitrariness of $\varphi$ yields $Au-f \geq 0$ a.e. Define the non-coincidence set $I := \{u > \psi\}$ which is an open set since $u$ and $\psi$ are both continuous. Take $\varphi \in C_c^\infty(I)$; it follows that there exists an $\epsilon_0$ with $u\pm \epsilon \varphi \in K$ for all $\epsilon \leq \epsilon_0$. Using this as the choice of test function in \eqref{eq:VI}, we obtain
\[\pm \epsilon \langle Au-f,  \varphi \rangle = \pm \epsilon \int_\Omega (Au-f)\varphi \leq 0\]
whence $Au=f \text{ a.e. in $I$}.$ From this one gets $(Au-f)(u-\psi) = 0$ a.e. in $\Omega$.

\emph{Proof of Lemma \ref{lem:NNisC1_general}}. Consider the $\cF_{\mathrm{DRR}}$ case (the other case will follow by trivial adjustments) and write $u = M \circ \tilde u$. Firstly, observe that $\tilde u\colon \mathbb{R}^{\mathfrak{m}} \times \overline{\Omega} \to \mathbb{R}$ is   continuous as it is the composition of  continuous functions. Since $\mathfrak{B}_i$ is in fact continuously differentiable for all $i$ (it is clear for $i=0, \fd$, and for $i \in \{1, \hdots, \fd-1\}$, this follows by the chain rule and the fact that $\sigma \in C^1(\mathbb{R})$), making use of the chain rule, $\grad \tilde u\colon \mathbb{R}^{\mathfrak{m}} \times \overline{\Omega} \to \mathbb{R}^n$ is also continuous.

Now, to prove the desired continuity, we begin by taking a sequence $\theta_n \to \theta$.  The set $K_1 := (\cup_{n\in\mathbb{N}} \{\theta_n\}) \cup \{\theta\}$ is a compact set and hence $\tilde u\colon \overline{\Omega} \times K_1 \to \mathbb{R}$ is uniformly continuous.  For every $\delta > 0$, if $n$ is sufficiently large, we get $|\theta_n-\theta| \leq \delta$. Fix $\epsilon>0$. It follows by uniform continuity that there exists $N_0 \in \mathbb{N}$ such that if $n \geq N_0$, then 
\[|\tilde u(x, \theta_n)-\tilde u(x,\theta)| \leq \epsilon,\]
uniformly in $x$. Here we can take the supremum over $x$ and conclude that $\theta \mapsto \tilde u(\theta, \cdot)$ is continuous as a map from $\mathbb{R}^{\mathfrak{m}}$ to $C^0(\overline{\Omega})$. Using the continuity of $(\theta, x) \mapsto \grad \tilde u(\theta, x)$, we obtain continuity into $C^1(\overline{\Omega})$ by a similar argument.  The continuity of $M$ finishes the argument.

\vskip 0.2in
\bibliography{main}

\begin{thebibliography}{73}
\providecommand{\natexlab}[1]{#1}
\providecommand{\url}[1]{\texttt{#1}}
\expandafter\ifx\csname urlstyle\endcsname\relax
  \providecommand{\doi}[1]{doi: #1}\else
  \providecommand{\doi}{doi: \begingroup \urlstyle{rm}\Url}\fi

\bibitem[Adams and Fournier(2003)]{MR2424078}
R.~A. Adams and J.~J.~F. Fournier.
\newblock \emph{Sobolev spaces}, volume 140 of \emph{Pure and Applied
  Mathematics (Amsterdam)}.
\newblock Elsevier/Academic Press, Amsterdam, second edition, 2003.
\newblock ISBN 0-12-044143-8.

\bibitem[Aizawa et~al.(2024)Aizawa, Kimura, and Matsui]{Yuto}
Y.~Aizawa, M.~Kimura, and K.~Matsui.
\newblock Universal approximation properties for an odenet and a resnet:
  Mathematical analysis and numerical experiments.
\newblock \emph{Discrete and Continuous Dynamical Systems - B}, 29\penalty0
  (1):\penalty0 351--376, 2024.
\newblock ISSN 1531-3492.
\newblock \doi{10.3934/dcdsb.2023099}.
\newblock URL
  \url{https://www.aimsciences.org/article/id/646c64474a9fed1ce4f92c10}.

\bibitem[Akiba et~al.(2019)Akiba, Sano, Yanase, Ohta, and Koyama]{optuna_2019}
T.~Akiba, S.~Sano, T.~Yanase, T.~Ohta, and M.~Koyama.
\newblock Optuna: A next-generation hyperparameter optimization framework.
\newblock In \emph{Proceedings of the 25th {ACM} {SIGKDD} International
  Conference on Knowledge Discovery and Data Mining}, 2019.

\bibitem[Alphonse et~al.(2024)Alphonse, Kister, and Lun]{github}
A.~Alphonse, A.~Kister, and C.~H. Lun.
\newblock {NNVI: Code for the paper}.
\newblock \url{https://github.com/amal-alphonse/NNVI}, 2024.
\newblock URL \url{https://github.com/amal-alphonse/NNVI}.

\bibitem[Auchmuty(1989)]{Auchmuty}
G.~Auchmuty.
\newblock Variational principles for variational inequalities.
\newblock \emph{Numer. Funct. Anal. Optim.}, 10\penalty0 (9-10):\penalty0
  863--874, 1989.
\newblock ISSN 0163-0563,1532-2467.
\newblock \doi{10.1080/01630568908816335}.
\newblock URL \url{https://doi.org/10.1080/01630568908816335}.

\bibitem[Auslender(1976)]{Auslender}
A.~Auslender.
\newblock \emph{Optimisation}.
\newblock Masson, Paris-New York-Barcelona, 1976.
\newblock M\'{e}thodes num\'{e}riques, Ma\^{i}trise de Math\'{e}matiques et
  Applications Fondamentales.

\bibitem[Bahja et~al.(2023)Bahja, Hauffen, Jung, Bah, and
  Karambal]{bahja2023physicsinformedneuralnetworkframework}
H.~E. Bahja, J.~C. Hauffen, P.~Jung, B.~Bah, and I.~Karambal.
\newblock A physics-informed neural network framework for modeling
  obstacle-related equations, 2023.
\newblock URL \url{https://arxiv.org/abs/2304.03552}.

\bibitem[Bao et~al.(2020)Bao, Ye, Zang, and Zhou]{Bao}
G.~Bao, X.~Ye, Y.~Zang, and H.~Zhou.
\newblock Numerical solution of inverse problems by weak adversarial networks.
\newblock \emph{Inverse Problems}, 36\penalty0 (11):\penalty0 115003, 31, 2020.
\newblock ISSN 0266-5611.
\newblock \doi{10.1088/1361-6420/abb447}.
\newblock URL \url{https://doi.org/10.1088/1361-6420/abb447}.

\bibitem[Bartels(2015)]{Bartels}
S.~Bartels.
\newblock \emph{Numerical methods for nonlinear partial differential
  equations}, volume~47 of \emph{Springer Series in Computational Mathematics}.
\newblock Springer, Cham, 2015.
\newblock \doi{10.1007/978-3-319-13797-1}.
\newblock URL \url{https://doi.org/10.1007/978-3-319-13797-1}.

\bibitem[Bertoluzza et~al.(2024)Bertoluzza, Burman, and
  He]{bertoluzza2024bestapproximationresultsessential}
S.~Bertoluzza, E.~Burman, and C.~He.
\newblock W{AN} discretization of {PDE}s: best approximation, stabilization,
  and essential boundary conditions.
\newblock \emph{SIAM J. Sci. Comput.}, 46\penalty0 (6):\penalty0 C688--C715,
  2024.
\newblock ISSN 1064-8275,1095-7197.
\newblock \doi{10.1137/23M1588196}.
\newblock URL \url{https://doi.org/10.1137/23M1588196}.

\bibitem[Bethard(2022)]{WeNeedTo}
S.~Bethard.
\newblock We need to talk about random seeds.
\newblock \emph{arXiv e-prints}, art. arXiv:2210.13393, 10 2022.
\newblock \doi{10.48550/arXiv.2210.13393}.

\bibitem[Brevis et~al.(2022)Brevis, Muga, and {van der Zee}]{BREVIS2022115716}
I.~Brevis, I.~Muga, and K.~G. {van der Zee}.
\newblock Neural control of discrete weak formulations: Galerkin, least squares
  \& minimal-residual methods with quasi-optimal weights.
\newblock \emph{Computer Methods in Applied Mechanics and Engineering},
  402:\penalty0 115716, 2022.
\newblock ISSN 0045-7825.
\newblock \doi{https://doi.org/10.1016/j.cma.2022.115716}.
\newblock URL
  \url{https://www.sciencedirect.com/science/article/pii/S0045782522006715}.
\newblock A Special Issue in Honor of the Lifetime Achievements of J. Tinsley
  Oden.

\bibitem[Cheng et~al.(2023)Cheng, Shen, Wang, and Liang]{DeepNeural}
X.~Cheng, X.~Shen, X.~Wang, and K.~Liang.
\newblock A deep neural network-based method for solving obstacle problems.
\newblock \emph{Nonlinear Anal. Real World Appl.}, 72:\penalty0 Paper No.
  103864, 16, 2023.
\newblock ISSN 1468-1218.
\newblock \doi{10.1016/j.nonrwa.2023.103864}.
\newblock URL \url{https://doi.org/10.1016/j.nonrwa.2023.103864}.

\bibitem[Christof and Meyer(2018)]{CM}
C.~Christof and C.~Meyer.
\newblock A note on a priori {$L^p$}-error estimates for the obstacle problem.
\newblock \emph{Numer. Math.}, 139\penalty0 (1):\penalty0 27--45, 2018.
\newblock ISSN 0029-599X,0945-3245.
\newblock \doi{10.1007/s00211-017-0931-5}.
\newblock URL \url{https://doi.org/10.1007/s00211-017-0931-5}.

\bibitem[Combettes and Pesquet(2012)]{CombettesPesquet2012}
P.~L. Combettes and J.-C. Pesquet.
\newblock Primal-dual splitting algorithm for solving inclusions with mixtures
  of composite, lipschitzian, and parallel-sum type monotone operators.
\newblock \emph{Set-Valued and Variational Analysis}, 20\penalty0 (2):\penalty0
  307--330, 2012.
\newblock \doi{10.1007/s11228-011-0191-y}.

\bibitem[Dondl et~al.(2022)Dondl, M\"uller, and Zeinhofer]{MR4455184}
P.~Dondl, J.~M\"uller, and M.~Zeinhofer.
\newblock Uniform convergence guarantees for the deep {R}itz method for
  nonlinear problems.
\newblock \emph{Adv. Contin. Discrete Models}, pages Paper No. 49, 19, 2022.
\newblock ISSN 2731-4235.
\newblock \doi{10.1186/s13662-022-03722-8}.
\newblock URL \url{https://doi.org/10.1186/s13662-022-03722-8}.

\bibitem[E and Yu(2018)]{DeepRitz}
W.~E and B.~Yu.
\newblock The deep {R}itz method: a deep learning-based numerical algorithm for
  solving variational problems.
\newblock \emph{Commun. Math. Stat.}, 6\penalty0 (1):\penalty0 1--12, 2018.
\newblock ISSN 2194-6701.
\newblock \doi{10.1007/s40304-018-0127-z}.
\newblock URL \url{https://doi.org/10.1007/s40304-018-0127-z}.

\bibitem[Ekeland and T\'{e}mam(1999)]{EkelandTemam}
I.~Ekeland and R.~T\'{e}mam.
\newblock \emph{Convex analysis and variational problems}, volume~28 of
  \emph{Classics in Applied Mathematics}.
\newblock Society for Industrial and Applied Mathematics (SIAM), Philadelphia,
  PA, english edition, 1999.
\newblock ISBN 0-89871-450-8.
\newblock \doi{10.1137/1.9781611971088}.
\newblock URL \url{https://doi.org/10.1137/1.9781611971088}.
\newblock Translated from the French.

\bibitem[Fukushima(1992)]{Fukushima}
M.~Fukushima.
\newblock Equivalent differentiable optimization problems and descent methods
  for asymmetric variational inequality problems.
\newblock \emph{Math. Programming}, 53\penalty0 (1):\penalty0 99--110, 1992.
\newblock ISSN 0025-5610,1436-4646.
\newblock \doi{10.1007/BF01585696}.
\newblock URL \url{https://doi.org/10.1007/BF01585696}.

\bibitem[Glowinski(1984)]{MR0737005}
R.~Glowinski.
\newblock \emph{Numerical methods for nonlinear variational problems}.
\newblock Springer Series in Computational Physics. Springer-Verlag, New York,
  1984.
\newblock ISBN 0-387-12434-9.
\newblock \doi{10.1007/978-3-662-12613-4}.
\newblock URL \url{https://doi.org/10.1007/978-3-662-12613-4}.

\bibitem[Glowinski et~al.(1981)Glowinski, Lions, and
  Tr\'{e}moli\`eres]{Glowinski}
R.~Glowinski, J.-L. Lions, and R.~Tr\'{e}moli\`eres.
\newblock \emph{Numerical analysis of variational inequalities}, volume~8 of
  \emph{Studies in Mathematics and its Applications}.
\newblock North-Holland Publishing Co., Amsterdam-New York, 1981.
\newblock ISBN 0-444-86199-8.
\newblock Translated from the French.

\bibitem[Gühring and Raslan(2021)]{GUHRING2021107}
I.~Gühring and M.~Raslan.
\newblock Approximation rates for neural networks with encodable weights in
  smoothness spaces.
\newblock \emph{Neural Networks}, 134:\penalty0 107--130, 2021.
\newblock ISSN 0893-6080.
\newblock \doi{https://doi.org/10.1016/j.neunet.2020.11.010}.
\newblock URL
  \url{https://www.sciencedirect.com/science/article/pii/S0893608020303956}.

\bibitem[He et~al.(2016{\natexlab{a}})He, Zhang, Ren, and Sun]{DeepResImaging}
K.~He, X.~Zhang, S.~Ren, and J.~Sun.
\newblock Deep residual learning for image recognition.
\newblock In \emph{2016 IEEE Conference on Computer Vision and Pattern
  Recognition (CVPR)}, pages 770--778, 2016{\natexlab{a}}.
\newblock \doi{10.1109/CVPR.2016.90}.

\bibitem[He et~al.(2016{\natexlab{b}})He, Zhang, Ren, and Sun]{HeEtAll}
K.~He, X.~Zhang, S.~Ren, and J.~Sun.
\newblock Identity mappings in deep residual networks.
\newblock In B.~Leibe, J.~Matas, N.~Sebe, and M.~Welling, editors,
  \emph{Computer Vision -- ECCV 2016}, pages 630--645, Cham,
  2016{\natexlab{b}}. Springer International Publishing.
\newblock ISBN 978-3-319-46493-0.

\bibitem[Hieber and Wood(2007)]{MR2333653}
M.~Hieber and I.~Wood.
\newblock The {D}irichlet problem in convex bounded domains for operators in
  non-divergence form with {$L^\infty$}-coefficients.
\newblock \emph{Differential Integral Equations}, 20\penalty0 (7):\penalty0
  721--734, 2007.
\newblock ISSN 0893-4983.

\bibitem[Hinterm\"uller and Kunisch(2006)]{MR2219149}
M.~Hinterm\"uller and K.~Kunisch.
\newblock Path-following methods for a class of constrained minimization
  problems in function space.
\newblock \emph{SIAM J. Optim.}, 17\penalty0 (1):\penalty0 159--187, 2006.
\newblock ISSN 1052-6234,1095-7189.
\newblock \doi{10.1137/040611598}.
\newblock URL \url{https://doi.org/10.1137/040611598}.

\bibitem[Hinterm\"uller and Laurain(2010)]{MR2653723}
M.~Hinterm\"uller and A.~Laurain.
\newblock A shape and topology optimization technique for solving a class of
  linear complementarity problems in function space.
\newblock \emph{Comput. Optim. Appl.}, 46\penalty0 (3):\penalty0 535--569,
  2010.
\newblock ISSN 0926-6003,1573-2894.
\newblock \doi{10.1007/s10589-008-9201-x}.
\newblock URL \url{https://doi.org/10.1007/s10589-008-9201-x}.

\bibitem[Hinterm\"uller and Laurain(2011)]{MR2806573}
M.~Hinterm\"uller and A.~Laurain.
\newblock Optimal shape design subject to elliptic variational inequalities.
\newblock \emph{SIAM J. Control Optim.}, 49\penalty0 (3):\penalty0 1015--1047,
  2011.
\newblock ISSN 0363-0129,1095-7138.
\newblock \doi{10.1137/080745134}.
\newblock URL \url{https://doi.org/10.1137/080745134}.

\bibitem[Hinterm\"uller et~al.(2002)Hinterm\"uller, Ito, and
  Kunisch]{MR1972219}
M.~Hinterm\"uller, K.~Ito, and K.~Kunisch.
\newblock The primal-dual active set strategy as a semismooth {N}ewton method.
\newblock \emph{SIAM J. Optim.}, 13\penalty0 (3):\penalty0 865--888, 2002.
\newblock ISSN 1052-6234,1095-7189.
\newblock \doi{10.1137/S1052623401383558}.
\newblock URL \url{https://doi.org/10.1137/S1052623401383558}.

\bibitem[Hoppe(1987)]{MR0909064}
R.~H.~W. Hoppe.
\newblock Multigrid algorithms for variational inequalities.
\newblock \emph{SIAM J. Numer. Anal.}, 24\penalty0 (5):\penalty0 1046--1065,
  1987.
\newblock ISSN 0036-1429.
\newblock \doi{10.1137/0724069}.
\newblock URL \url{https://doi.org/10.1137/0724069}.

\bibitem[Huang et~al.(2022)Huang, Wang, and Wang]{MR4407899}
J.~Huang, C.~Wang, and H.~Wang.
\newblock A deep learning method for elliptic hemivariational inequalities.
\newblock \emph{East Asian J. Appl. Math.}, 12\penalty0 (3):\penalty0 487--502,
  2022.
\newblock ISSN 2079-7362.
\newblock \doi{10.4208/eajam.081121.161121}.
\newblock URL \url{https://doi.org/10.4208/eajam.081121.161121}.

\bibitem[Hung et~al.(2020)Hung, Mig\'{o}rski, Tam, and Zeng]{Migorski}
N.~V. Hung, S.~Mig\'{o}rski, V.~M. Tam, and S.~Zeng.
\newblock Gap functions and error bounds for variational-hemivariational
  inequalities.
\newblock \emph{Acta Appl. Math.}, 169:\penalty0 691--709, 2020.
\newblock ISSN 0167-8019.
\newblock \doi{10.1007/s10440-020-00319-9}.
\newblock URL \url{https://doi.org/10.1007/s10440-020-00319-9}.

\bibitem[Hwang and Lim(2024)]{DBLP:conf/nips/HwangL24}
Y.~Hwang and D.~Lim.
\newblock Dual cone gradient descent for training physics-informed neural
  networks.
\newblock In A.~Globersons, L.~Mackey, D.~Belgrave, A.~Fan, U.~Paquet, J.~M.
  Tomczak, and C.~Zhang, editors, \emph{Advances in Neural Information
  Processing Systems 38: Annual Conference on Neural Information Processing
  Systems 2024, NeurIPS 2024, Vancouver, BC, Canada, December 10 - 15, 2024},
  2024.
\newblock URL
  \url{http://papers.nips.cc/paper\_files/paper/2024/hash/b2b781badeeb49896c4b324c466ec442-Abstract-Conference.html}.

\bibitem[Ito and Kunisch(2003{\natexlab{a}})]{ItoKunisch2003}
K.~Ito and K.~Kunisch.
\newblock Semi-smooth {N}ewton methods for variational inequalities of the
  first kind.
\newblock \emph{ESAIM: Mathematical Modelling and Numerical Analysis},
  37\penalty0 (1):\penalty0 41--62, 2003{\natexlab{a}}.
\newblock \doi{10.1051/m2an:2003021}.

\bibitem[Ito and Kunisch(2003{\natexlab{b}})]{MR1972649}
K.~Ito and K.~Kunisch.
\newblock Semi-smooth {N}ewton methods for variational inequalities of the
  first kind.
\newblock \emph{M2AN Math. Model. Numer. Anal.}, 37\penalty0 (1):\penalty0
  41--62, 2003{\natexlab{b}}.
\newblock ISSN 0764-583X,1290-3841.
\newblock \doi{10.1051/m2an:2003021}.
\newblock URL \url{https://doi.org/10.1051/m2an:2003021}.

\bibitem[Jiang and Chen(2023)]{JiangChen}
J.~Jiang and X.~Chen.
\newblock Optimality conditions for nonsmooth nonconvex-nonconcave min-max
  problems and generative adversarial networks.
\newblock \emph{SIAM J. Math. Data Sci.}, 5\penalty0 (3):\penalty0 693--722,
  2023.
\newblock ISSN 2577-0187.
\newblock \doi{10.1137/22M1482238}.
\newblock URL \url{https://doi.org/10.1137/22M1482238}.

\bibitem[Jiao et~al.(2023)Jiao, Lai, Wang, Yang, and Yang]{Jiao2023}
Y.~Jiao, Y.~Lai, Y.~Wang, H.~Yang, and Y.~Yang.
\newblock \emph{Convergence Analysis of the Deep Galerkin Method for Weak
  Solutions}, pages 53--82.
\newblock Springer International Publishing, Cham, 2023.
\newblock ISBN 978-3-031-37800-3.
\newblock \doi{10.1007/978-3-031-37800-3_4}.
\newblock URL \url{https://doi.org/10.1007/978-3-031-37800-3_4}.

\bibitem[Jin et~al.(2020)Jin, Netrapalli, and Jordan]{Jordan}
C.~Jin, P.~Netrapalli, and M.~Jordan.
\newblock What is local optimality in nonconvex-nonconcave minimax
  optimization?
\newblock In H.~D. III and A.~Singh, editors, \emph{Proceedings of the 37th
  International Conference on Machine Learning}, volume 119 of
  \emph{Proceedings of Machine Learning Research}, pages 4880--4889. PMLR, 7
  2020.
\newblock URL \url{https://proceedings.mlr.press/v119/jin20e.html}.

\bibitem[K\"arkk\"ainen et~al.(2003)K\"arkk\"ainen, Kunisch, and
  Tarvainen]{MR2026461}
T.~K\"arkk\"ainen, K.~Kunisch, and P.~Tarvainen.
\newblock Augmented {L}agrangian active set methods for obstacle problems.
\newblock \emph{J. Optim. Theory Appl.}, 119\penalty0 (3):\penalty0 499--533,
  2003.
\newblock ISSN 0022-3239,1573-2878.
\newblock \doi{10.1023/B:JOTA.0000006687.57272.b6}.
\newblock URL \url{https://doi.org/10.1023/B:JOTA.0000006687.57272.b6}.

\bibitem[Keith and Surowiec(2026)]{KS}
B.~Keith and T.~M. Surowiec.
\newblock Proximal {G}alerkin: a structure-preserving finite element method for
  pointwise bound constraints.
\newblock \emph{Found. Comput. Math.}, 26\penalty0 (1):\penalty0 385--481,
  2026.
\newblock ISSN 1615-3375,1615-3383.
\newblock \doi{10.1007/s10208-024-09681-8}.
\newblock URL \url{https://doi.org/10.1007/s10208-024-09681-8}.

\bibitem[Kinderlehrer and Stampacchia(1980)]{MR0567696}
D.~Kinderlehrer and G.~Stampacchia.
\newblock \emph{An introduction to variational inequalities and their
  applications}, volume~88 of \emph{Pure and Applied Mathematics}.
\newblock Academic Press, Inc. [Harcourt Brace Jovanovich, Publishers], New
  York-London, 1980.
\newblock ISBN 0-12-407350-6.

\bibitem[Kornhuber(1994)]{MR1310316}
R.~Kornhuber.
\newblock Monotone multigrid methods for elliptic variational inequalities.
  {I}.
\newblock \emph{Numer. Math.}, 69\penalty0 (2):\penalty0 167--184, 1994.
\newblock ISSN 0029-599X,0945-3245.
\newblock \doi{10.1007/BF03325426}.
\newblock URL \url{https://doi.org/10.1007/BF03325426}.

\bibitem[Lagaris et~al.(1998)Lagaris, Likas, and Fotiadis]{Lagaris98}
I.~Lagaris, A.~Likas, and D.~Fotiadis.
\newblock Artificial neural networks for solving ordinary and partial
  differential equations.
\newblock \emph{IEEE Transactions on Neural Networks}, 9\penalty0 (5):\penalty0
  987--1000, 1998.
\newblock \doi{10.1109/72.712178}.

\bibitem[Larsson and Patriksson(1994)]{MR1274172}
T.~Larsson and M.~Patriksson.
\newblock A class of gap functions for variational inequalities.
\newblock \emph{Math. Programming}, 64\penalty0 (1):\penalty0 53--79, 1994.
\newblock ISSN 0025-5610,1436-4646.
\newblock \doi{10.1007/BF01582565}.
\newblock URL \url{https://doi.org/10.1007/BF01582565}.

\bibitem[Li et~al.(2022)Li, Lin, and Shen]{li2022deep}
Q.~Li, T.~Lin, and Z.~Shen.
\newblock Deep learning via dynamical systems: An approximation perspective.
\newblock \emph{Journal of the European Mathematical Society}, 25\penalty0
  (5):\penalty0 1671--1709, 2022.

\bibitem[Lin and Jegelka(2018)]{NEURIPS2018_03bfc1d4}
H.~Lin and S.~Jegelka.
\newblock Resnet with one-neuron hidden layers is a universal approximator.
\newblock In S.~Bengio, H.~Wallach, H.~Larochelle, K.~Grauman, N.~Cesa-Bianchi,
  and R.~Garnett, editors, \emph{Advances in Neural Information Processing
  Systems}, volume~31. Curran Associates, Inc., 2018.
\newblock URL
  \url{https://proceedings.neurips.cc/paper_files/paper/2018/file/03bfc1d4783966c69cc6aef8247e0103-Paper.pdf}.

\bibitem[Lin et~al.(2020)Lin, Jin, and Jordan]{lin2020gradient}
T.~Lin, C.~Jin, and M.~Jordan.
\newblock On gradient descent ascent for nonconvex-concave minimax problems.
\newblock In H.~D. III and A.~Singh, editors, \emph{Proceedings of the 37th
  International Conference on Machine Learning}, volume 119 of
  \emph{Proceedings of Machine Learning Research}, pages 6083--6093. PMLR,
  13--18 Jul 2020.
\newblock URL \url{https://proceedings.mlr.press/v119/lin20a.html}.

\bibitem[Liu et~al.(2024)Liu, Liang, and Chen]{liu2024characterizing}
C.~Liu, E.~Liang, and M.~Chen.
\newblock Characterizing {R}es{N}et’s universal approximation capability.
\newblock In R.~Salakhutdinov, Z.~Kolter, K.~Heller, A.~Weller, N.~Oliver,
  J.~Scarlett, and F.~Berkenkamp, editors, \emph{Proceedings of the 41st
  International Conference on Machine Learning}, volume 235 of
  \emph{Proceedings of Machine Learning Research}, pages 31477--31515. PMLR,
  21--27 Jul 2024.
\newblock URL \url{https://proceedings.mlr.press/v235/liu24am.html}.

\bibitem[Loshchilov and Hutter(2019)]{loshchilov2018decoupled}
I.~Loshchilov and F.~Hutter.
\newblock Decoupled weight decay regularization.
\newblock In \emph{International Conference on Learning Representations}, 2019.
\newblock URL \url{https://openreview.net/forum?id=Bkg6RiCqY7}.

\bibitem[Madhyastha and Jain(2019)]{ModelStability}
P.~Madhyastha and R.~Jain.
\newblock On model stability as a function of random seed.
\newblock In M.~Bansal and A.~Villavicencio, editors, \emph{Proceedings of the
  23rd Conference on Computational Natural Language Learning (CoNLL)}, pages
  929--939, Hong Kong, China, Nov. 2019. Association for Computational
  Linguistics.
\newblock \doi{10.18653/v1/K19-1087}.
\newblock URL \url{https://aclanthology.org/K19-1087}.

\bibitem[Meyer and Thoma(2013)]{MT}
C.~Meyer and O.~Thoma.
\newblock A priori finite element error analysis for optimal control of the
  obstacle problem.
\newblock \emph{SIAM J. Numer. Anal.}, 51\penalty0 (1):\penalty0 605--628,
  2013.
\newblock ISSN 0036-1429,1095-7170.
\newblock \doi{10.1137/110836092}.
\newblock URL \url{https://doi.org/10.1137/110836092}.

\bibitem[Mokhtari et~al.(2020)Mokhtari, Ozdaglar, and
  Pattathil]{MokhtariOzdaglarPattathil2020}
A.~Mokhtari, A.~Ozdaglar, and S.~Pattathil.
\newblock A unified analysis of extra-gradient and optimistic gradient methods
  for saddle point problems: Proximal point approach.
\newblock In \emph{Proceedings of the Twenty Third International Conference on
  Artificial Intelligence and Statistics}, volume 108 of \emph{Proceedings of
  Machine Learning Research}, pages 1497--1507. PMLR, 2020.
\newblock URL \url{https://proceedings.mlr.press/v108/mokhtari20a.html}.

\bibitem[Nemirovski(2004)]{Nemirovski2004}
A.~Nemirovski.
\newblock Prox-method with rate of convergence {$\mathcal{O}(1/t)$} for
  variational inequalities with lipschitz continuous monotone operators and
  smooth convex-concave saddle point problems.
\newblock \emph{SIAM Journal on Optimization}, 15\penalty0 (1):\penalty0
  229--251, 2004.
\newblock \doi{10.1137/S1052623403425629}.

\bibitem[Paszke et~al.(2019)Paszke, Gross, Massa, Lerer, Bradbury, Chanan,
  Killeen, Lin, Gimelshein, Antiga, Desmaison, Kopf, Yang, DeVito, Raison,
  Tejani, Chilamkurthy, Steiner, Fang, Bai, and Chintala]{NEURIPS2019_9015}
A.~Paszke, S.~Gross, F.~Massa, A.~Lerer, J.~Bradbury, G.~Chanan, T.~Killeen,
  Z.~Lin, N.~Gimelshein, L.~Antiga, A.~Desmaison, A.~Kopf, E.~Yang, Z.~DeVito,
  M.~Raison, A.~Tejani, S.~Chilamkurthy, B.~Steiner, L.~Fang, J.~Bai, and
  S.~Chintala.
\newblock Pytorch: An imperative style, high-performance deep learning library.
\newblock In \emph{Advances in Neural Information Processing Systems 32}, pages
  8024--8035. Curran Associates, Inc., 2019.
\newblock URL
  \url{http://papers.neurips.cc/paper/9015-pytorch-an-imperative-style-high-performance-deep-learning-library.pdf}.

\bibitem[Raissi et~al.(2019)Raissi, Perdikaris, and Karniadakis]{RAISSI2019686}
M.~Raissi, P.~Perdikaris, and G.~Karniadakis.
\newblock Physics-informed neural networks: A deep learning framework for
  solving forward and inverse problems involving nonlinear partial differential
  equations.
\newblock \emph{Journal of Computational Physics}, 378:\penalty0 686--707,
  2019.
\newblock ISSN 0021-9991.
\newblock \doi{https://doi.org/10.1016/j.jcp.2018.10.045}.
\newblock URL
  \url{https://www.sciencedirect.com/science/article/pii/S0021999118307125}.

\bibitem[Razaviyayn et~al.(2020)Razaviyayn, Huang, Lu, Nouiehed, Sanjabi, and
  Hong]{razaviyayn2020nonconvex}
M.~Razaviyayn, T.~Huang, S.~Lu, M.~Nouiehed, M.~Sanjabi, and M.~Hong.
\newblock Nonconvex min-max optimization: Applications, challenges, and recent
  theoretical advances.
\newblock \emph{IEEE Signal Processing Magazine}, 37\penalty0 (5):\penalty0
  55--66, 2020.

\bibitem[Rodrigues(1987)]{Rodrigues}
J.-F. Rodrigues.
\newblock \emph{Obstacle problems in mathematical physics}, volume 134 of
  \emph{North-Holland Mathematics Studies}.
\newblock North-Holland Publishing Co., Amsterdam, 1987.
\newblock ISBN 0-444-70187-7.
\newblock Notas de Matem\'{a}tica [Mathematical Notes], 114.

\bibitem[Schwab and Stein(2022)]{MR4434275}
C.~Schwab and A.~Stein.
\newblock Deep solution operators for variational inequalities via proximal
  neural networks.
\newblock \emph{Res. Math. Sci.}, 9\penalty0 (3):\penalty0 Paper No. 36, 35,
  2022.
\newblock ISSN 2522-0144,2197-9847.
\newblock \doi{10.1007/s40687-022-00327-1}.
\newblock URL \url{https://doi.org/10.1007/s40687-022-00327-1}.

\bibitem[Shin et~al.(2023)Shin, Zhang, and Karniadakis]{Shin2023}
Y.~Shin, Z.~Zhang, and G.~E. Karniadakis.
\newblock Error estimates of residual minimization using neural networks for
  linear pdes.
\newblock \emph{Journal of Machine Learning for Modeling and Computing},
  4\penalty0 (4):\penalty0 73--101, 2023.
\newblock ISSN 2689-3967.

\bibitem[{Shugart} and {Altschuler}(2025)]{2025arXiv250501423S}
H.~{Shugart} and J.~M. {Altschuler}.
\newblock {Negative Stepsizes Make Gradient-Descent-Ascent Converge}.
\newblock \emph{arXiv e-prints}, art. arXiv:2505.01423, May 2025.
\newblock \doi{10.48550/arXiv.2505.01423}.

\bibitem[Sofonea and Matei(2009)]{MR2488869}
M.~Sofonea and A.~Matei.
\newblock \emph{Variational inequalities with applications}, volume~18 of
  \emph{Advances in Mechanics and Mathematics}.
\newblock Springer, New York, 2009.
\newblock ISBN 978-0-387-87459-3.
\newblock A study of antiplane frictional contact problems.

\bibitem[Sukumar and Srivastava(2022)]{SUKUMAR2022114333}
N.~Sukumar and A.~Srivastava.
\newblock Exact imposition of boundary conditions with distance functions in
  physics-informed deep neural networks.
\newblock \emph{Computer Methods in Applied Mechanics and Engineering},
  389:\penalty0 114333, 2022.
\newblock ISSN 0045-7825.
\newblock \doi{https://doi.org/10.1016/j.cma.2021.114333}.
\newblock URL
  \url{https://www.sciencedirect.com/science/article/pii/S0045782521006186}.

\bibitem[Valsecchi~Oliva et~al.(2022)Valsecchi~Oliva, Wu, He, and
  Ni]{MR4416735}
P.~Valsecchi~Oliva, Y.~Wu, C.~He, and H.~Ni.
\newblock Towards fast weak adversarial training to solve high dimensional
  parabolic partial differential equations using {XNODE}-{WAN}.
\newblock \emph{J. Comput. Phys.}, 463:\penalty0 Paper No. 111233, 17, 2022.
\newblock ISSN 0021-9991,1090-2716.
\newblock \doi{10.1016/j.jcp.2022.111233}.
\newblock URL \url{https://doi.org/10.1016/j.jcp.2022.111233}.

\bibitem[Wang et~al.(2010)Wang, Han, and Cheng]{MR2670002}
F.~Wang, W.~Han, and X.-L. Cheng.
\newblock Discontinuous {G}alerkin methods for solving elliptic variational
  inequalities.
\newblock \emph{SIAM J. Numer. Anal.}, 48\penalty0 (2):\penalty0 708--733,
  2010.
\newblock ISSN 0036-1429,1095-7170.
\newblock \doi{10.1137/09075891X}.
\newblock URL \url{https://doi.org/10.1137/09075891X}.

\bibitem[Wang et~al.(2021)Wang, Teng, and
  Perdikaris]{DBLP:journals/siamsc/WangTP21}
S.~Wang, Y.~Teng, and P.~Perdikaris.
\newblock Understanding and mitigating gradient flow pathologies in
  physics-informed neural networks.
\newblock \emph{{SIAM} J. Sci. Comput.}, 43\penalty0 (5):\penalty0
  A3055--A3081, 2021.
\newblock \doi{10.1137/20M1318043}.
\newblock URL \url{https://doi.org/10.1137/20M1318043}.

\bibitem[Wloka(1987)]{MR895589}
J.~Wloka.
\newblock \emph{Partial differential equations}.
\newblock Cambridge University Press, Cambridge, 1987.
\newblock \doi{10.1017/CBO9781139171755}.
\newblock URL \url{https://doi.org/10.1017/CBO9781139171755}.
\newblock Translated from the German by C. B. Thomas and M. J. Thomas.

\bibitem[Wu et~al.(2019)Wu, Shen, and {van den Hengel}]{WU2019119}
Z.~Wu, C.~Shen, and A.~{van den Hengel}.
\newblock Wider or deeper: Revisiting the resnet model for visual recognition.
\newblock \emph{Pattern Recognition}, 90:\penalty0 119--133, 2019.
\newblock ISSN 0031-3203.
\newblock \doi{https://doi.org/10.1016/j.patcog.2019.01.006}.
\newblock URL
  \url{https://www.sciencedirect.com/science/article/pii/S0031320319300135}.

\bibitem[Yamashita and Fukushima(1997)]{MR1430295}
N.~Yamashita and M.~Fukushima.
\newblock Equivalent unconstrained minimization and global error bounds for
  variational inequality problems.
\newblock \emph{SIAM J. Control Optim.}, 35\penalty0 (1):\penalty0 273--284,
  1997.
\newblock ISSN 0363-0129.
\newblock \doi{10.1137/S0363012994277645}.
\newblock URL \url{https://doi.org/10.1137/S0363012994277645}.

\bibitem[Yao et~al.(2023)Yao, Su, Hao, Liu, Su, and
  Zhu]{DBLP:conf/icml/YaoSHL0Z23}
J.~Yao, C.~Su, Z.~Hao, S.~Liu, H.~Su, and J.~Zhu.
\newblock Multiadam: Parameter-wise scale-invariant optimizer for multiscale
  training of physics-informed neural networks.
\newblock In A.~Krause, E.~Brunskill, K.~Cho, B.~Engelhardt, S.~Sabato, and
  J.~Scarlett, editors, \emph{International Conference on Machine Learning,
  {ICML} 2023, 23-29 July 2023, Honolulu, Hawaii, {USA}}, Proceedings of
  Machine Learning Research, pages 39702--39721. {PMLR}, 2023.
\newblock URL \url{https://proceedings.mlr.press/v202/yao23c.html}.

\bibitem[Zamani et~al.(2024)Zamani, Abbaszadehpeivasti, and {de
  Klerk}]{ZamaniAbbaszadehpeivastiDeKlerk2024}
M.~Zamani, H.~Abbaszadehpeivasti, and E.~{de Klerk}.
\newblock Convergence rate analysis of the gradient descent--ascent method for
  convex--concave saddle-point problems.
\newblock \emph{Optimization Methods and Software}, 39\penalty0 (5):\penalty0
  967--989, 2024.
\newblock \doi{10.1080/10556788.2024.2360040}.

\bibitem[Zang et~al.(2020)Zang, Bao, Ye, and Zhou]{Zang}
Y.~Zang, G.~Bao, X.~Ye, and H.~Zhou.
\newblock Weak adversarial networks for high-dimensional partial differential
  equations.
\newblock \emph{J. Comput. Phys.}, 411:\penalty0 109409, 14, 2020.
\newblock ISSN 0021-9991.
\newblock \doi{10.1016/j.jcp.2020.109409}.
\newblock URL \url{https://doi.org/10.1016/j.jcp.2020.109409}.

\bibitem[Zeidler(1986)]{ZeidlerI}
E.~Zeidler.
\newblock \emph{Nonlinear functional analysis and its applications. {I}}.
\newblock Springer-Verlag, New York, 1986.
\newblock ISBN 0-387-90914-1.
\newblock \doi{10.1007/978-1-4612-4838-5}.
\newblock URL \url{https://doi.org/10.1007/978-1-4612-4838-5}.
\newblock Fixed-point theorems, Translated from the German by Peter R. Wadsack.

\bibitem[Zhao et~al.(2022)Zhao, Hao, and Hu]{TwoNN}
X.~E. Zhao, W.~Hao, and B.~Hu.
\newblock Two neural-network-based methods for solving elliptic obstacle
  problems.
\newblock \emph{Chaos Solitons Fractals}, 161:\penalty0 Paper No. 112313, 10,
  2022.
\newblock ISSN 0960-0779.
\newblock \doi{10.1016/j.chaos.2022.112313}.
\newblock URL \url{https://doi.org/10.1016/j.chaos.2022.112313}.

\end{thebibliography}

\end{document}